\pdfoutput=1
 \documentclass[authoryear]{elsarticle}
 \usepackage[margin=1in,footskip=0.25in]{geometry}

\usepackage[fleqn]{amsmath}
\usepackage{amsthm,amsfonts,amssymb,mathtools,etoolbox, url}
\usepackage{graphicx, enumitem, subcaption}
\usepackage{comment}
\usepackage[dvipsnames]{xcolor}
\usepackage[section]{placeins}
\usepackage[colorlinks,citecolor=blue,urlcolor=blue]{hyperref}

\journal{Stochastic Processes and Their Applications}

\newtheorem{theorem}{Theorem}[section]
\newtheorem*{theorem*}{Theorem}
\newtheorem{lemma}[theorem]{Lemma}
\newtheorem{proposition}[theorem]{Proposition}
\newtheorem{corollary}[theorem]{Corollary}
\newtheorem{remark}[theorem]{Remark}
\newtheorem{examples}[theorem]{Examples}

\newtheoremstyle{named}{}{}{\itshape}{}{\sc}{.}{.8em}{\indent \thmnote{#3}}
\newtheorem{definition}[theorem]{Definition}

\theoremstyle{named}
\newtheorem*{namedtheorem}{Theorem}

\def\P{\mathbb{P}}
\DeclareMathOperator\supp{supp}

\DeclareMathOperator*{\argmin}{arg\,min}

\newcommand{\E}{\mathbb{E}}

\newcommand{\Cov}{\mathbb{C}\mathrm{ov}}

\newcommand{\V}{\mathbb{V}}

\newcommand{\deriv}[1]{\frac{\partial}{\partial #1}}
\newcommand{\mb}[1]{\boldsymbol{#1}}

\newcommand{\dprime}{{\prime\prime}}

\newcounter{suppcounter}[suppcounter]

\makeatletter
\def\ps@pprintTitle{%
 \let\@oddhead\@empty
 \let\@evenhead\@empty
 \def\@oddfoot{\hfill \small\emph{30 May, 2025}}%
 \let\@evenfoot\@oddfoot}
\makeatother

\begin{document}

\begin{frontmatter}

\title{Smoothness Estimation for Whittle-Mat\'ern Processes on Closed Riemannian Manifolds}

\author[mks]{Moritz Korte-Stapff\corref{cor1}}
\cortext[cor1]{Corresponding author}
\ead{kortest@umich.edu}

\affiliation[mks]{organization={Department of Statistics, University of Michigan},
            city={Ann Arbor},
            postcode={48104}, 
            state={Michigan},
            country={USA}}

\author[tk]{Toni Karvonen}
\ead{toni.karvonen@lut.fi}
\affiliation[tk]{organization={School of Engineering Sciences, Lappeenranta--Lahti University of Technology LUT},
            city={Lappeenranta},
            postcode={53850}, 
            country={Finland}}

\author[em]{\'Eric Moulines}
\ead{eric.moulines@polytechnique.edu}
\affiliation[em]{organization={CMAP, Ecole Polytechnique},
            city={Palaiseau},
            postcode={91120}, 
            country={France}}
            
\begin{abstract}
The family of Mat\'ern kernels are often used in spatial statistics, function approximation and Gaussian process methods in machine learning.
One reason for their popularity is the presence of a smoothness parameter that controls, for example, optimal error bounds for kriging and posterior contraction rates in Gaussian process regression.
On closed Riemannian manifolds, we show that the smoothness parameter can be consistently estimated from the maximizer(s) of the Gaussian likelihood when the underlying data are from point evaluations of a Gaussian process and, perhaps surprisingly, even when the data comprise evaluations of a non-Gaussian process.
The points at which the process is observed need not have any particular spatial structure beyond quasi-uniformity.
Our methods are based on results from approximation theory for the Sobolev scale of Hilbert spaces.
Moreover, we generalize a well-known equivalence of measures phenomenon related to Mat\'ern kernels to the non-Gaussian case by using Kakutani's theorem.
\end{abstract}


\end{frontmatter}


\section{Introduction}
When modelling the covariance structure of spatial or spatio-temporal observations, one of the most common choices is to use parametric models based on Mat\'ern kernels.
In the simplest isotropic form, they give the covariance between the observations $u(x_1)$ and $u(x_2)$, indexed by sampling sites $x_1, x_2 \in \mathbb{R}^d$, as
\begin{align}
    \Cov (u(x_1), u(x_2)) = \sigma^2 \frac{2^{1-(s-d/2)}}{\Gamma(s-d/2)}\sqrt{\tau}^{s-d/2}\lVert x_1 - x_2 \rVert^{s-d/2} \mathcal{K}_{s-d/2}\left( \sqrt{\tau} \lVert x_1 - x_2 \rVert\right),  \label{eq:matern_cov}
\end{align}
where $s > d/2$ and $\tau, \sigma^2 > 0$ are covariance parameters, $\lVert \cdot \rVert$ denotes the Euclidean norm on $\mathbb{R}^d$ and $\mathcal{K}_s$ is the modified Bessel function of the second kind.
The Matérn covariance functions are attributed to \citet{matern_1960} and are ubiquitously used in spatial statistics.
In addition, they are popular as priors for Bayesian Gaussian process regression and as radial basis functions for scattered data interpolation.
For a review of Mat\'ern kernels and its application in a varitey of fields such as machine learning and function approximation, we refer to the recent review paper by \citet{porcu_2023}.
An important characterization was proved by \citet{whittle_1963}.
Let $\mathcal{W}$ be spatial Gaussian white noise.
Whittle showed that the stationary solution $u$ of the stochastic partial differential equation (SPDE)
\begin{align}
 (\tau - \Delta)^{s/2}u = \sigma \frac{2^{d/2} \pi^{d/4} \sqrt{\Gamma(s)}}{\sqrt{\Gamma(s-d/2)}}  \sqrt{\tau}^{s - d/2}\, \mathcal{W} , 
 \label{eq:spde}
\end{align}
on $\mathbb{R}^d$ is a Gaussian process with covariance~\eqref{eq:matern_cov}.
Here $\Delta$ is the Laplacian and $\mathcal{W}$ is Gaussian white noise.
This characterization was used in \citet{Lindgren_2011}, among others, to define Whittle-Mat\'ern processes on Riemannian manifolds by replacing the Euclidean Laplace operator in \eqref{eq:spde} by the Laplace--Beltrami operator of the respective Riemannian manifold.
Whittle-Mat\'ern Gaussian processes on Riemannian manifolds were further investigated, for example, in \citet{lang_2015, borovitsky_2020, li_2023}.

Stein strongly advocates the use of Mat\'ern covariance functions~\citep[p.\@~14: ``Use the Mat\'ern.'']{stein_1999}, mainly because the parameter $s$ controls the local behaviour of corresponding Gaussian processes.
More precisely, a Gaussian process with the covariance structure~\eqref{eq:matern_cov} is $\lceil s - d/2\rceil -1$ times mean square differentiable~\citep[p.\@~31]{stein_1999}.
This is related to the fact that Mat\'ern covariance functions are reproducing kernels for the scale of Sobolev spaces $H^s(\mathbb{R}^d)$.
In fact, for many tasks for which Mat\'ern Kernels are used, the smoothness parameter $s$ is the \emph{only} parameter that matters asymptotically. 
For example, it is shown in \citet{VaartZanten2011} that for Bayesian Gaussian process regression optimal contraction rates can only be attained if the smoothness parameter of the Gaussian process prior matches the smoothness of the unknown regression function.
This is similar to results in, for example, \citet{wang_2020} and \citet{narcowich_2006}, where optimal approximation and kriging error rates depend only on the smoothness of the kernel and the function to be approximated.
Moreover, estimation of smoothness parameters is necessary to guarantee the reliability of uncertainty quantification provided by non-parametric methods, such as Gaussian process regression~\citep{GineNickl2010, Bull2012, Szabo2015, HadjiSzabo2021}.

Despite the ubiquity of (Whittle-)Mat\'ern kernels in spatial statistics, Gaussian process regression and function approximation, and the importance of the smoothness parameter $s$, the consistency of the maximum likelihood estimator of $s$ has been an open problem in the \textit{infill} (or \emph{fixed-domain}) asymptotic framework \citep{zhang_2005}, i.e., when the sampling sites become increasingly dense in a bounded domain over which the process is observed.
This is the case even when the observations do stem from a correctly specified Gaussian process and is in contrast to the case of \textit{increasing-domain} asymptotics, where it has long been known that maximum likelihood estimators of all parameters are consistent and asymptotically normal in the well-specified case~\citep{mardia_1984}.
In addition to the comparative lack of positive results, it is known for Gaussian processes that consistent estimation of the parameters $\tau$ and $\sigma^2$ in \eqref{eq:matern_cov} or \eqref{eq:spde} is \textit{not} possible if $d < 4$.
In this case only the so-called microergodic parameter $\tau^{s - d/2}\sigma^2$ can be consistently estimated~\citep{zhang_2004, bolin_2023, li_2023}.
This is due to equivalence of the involved Gaussian measures.
In dimension $5$ or higher, consistent estimators are given in \citet{anderes_2010}.
Due to the difficulty of establishing theoretic guarantees for estimation of the smoothness parameter, there is a large amount of literature on estimation of the microergodic parameter under the seldomly satisfied assumption that the correct smoothness is known~\citep[e.g.,][]{stein_1999, zhang_2004, du_2009, wang_2011, Kaufman2013, bachoc_2014, Li2022}.
Similar results on related compactly support covariance functions can be found in \citet{bevilacqua_2019}, see also \citet{stein_hp}.

Recently, however, progress has been made in estimating the smoothness parameter $s$ in the context of infill asymptotics.
A consistent estimator for $s$ based on observations of a Gaussian process with covariance function \eqref{eq:matern_cov} on $[0,1)^d$ for $d = 1,2,3$ was presented in \citet{loh_2021}.
This work was extended in \citet{LohSun2023} to take into account, for example, noisy observations of the process.
On manifolds, consistency of the maximum likelihood estimator for the smoothness parameter on the torus and the circle was investigated by \citet{Chen2021} and \citet{Petit2023} under somewhat restrictive assumptions.
The estimation of Sobolev regularity of deterministic functions observed over bounded Euclidean domains using Gaussian likelihoods was studied in~\cite{karvonen2023asymptotic}.

We focus on the maximum likelihood estimation for stochastic processes with Whittle--Mat\'ern covariance structure observed on closed Riemannian manifolds.
Such manifolds are necessarily compact and therefore we work in the infill asymptotic framework.
Our main contributions are as follows:
\begin{enumerate}    
    \item \textsc{Theorem}~\ref{thm:consistency} and \textsc{Corollary}~\ref{cor:consistency}: We show that maximizers of the Gaussian likelihood for the smoothness parameter $s$ are consistent estimators under the assumption that the points at which the process is observed are \textit{quasi-uniform}.
      This continues to hold for observations from a variety of non-Gaussian processes and thus justifies the use of Gaussian maximum likelihoods for smoothness estimation even when there is no reason to assume that the underlying processes is Gaussian.
    \item \textsc{theorem} \ref{thm:equivalent_measures}: We show that the phenomenon of equivalence of measures and consequently the \emph{inconsistent} separation of range scale $(\tau)$ and variance ($\sigma^2$) parameters is not restricted to Gaussian processes. Specifically, we show that for a variety of other non-Gaussian stochastic processes the same phenomenon of equivalence of measures occurs as in the Gaussian case.
\end{enumerate}

Let us describe Theorem~\ref{thm:consistency} and Corollary~\ref{cor:consistency} in more detail.
Let $\{ e_i \}$ be eigenfunctions of the negative Laplacian $-\Delta$ on a closed Riemannian manifold and $\{\lambda_i\}$ the corresponding eigenvalues.
We observe a stochastic process $u$ at pairwise distinct points $(x_1, \dots, x_n) = \mb{x}$ on the manifold.
The process, which has fixed positive magnitude and range scale parameters $\sigma_0$ and $\tau_0$ and smoothness $s_0 > d/2$, is defined by
\begin{align*}
    u(x) = \sigma_0 \sqrt{\smash[b]{v(s_0, \tau_0)}} \sum_{i = 1}^\infty (\tau_0 + \lambda_i)^{-s_0/2}\xi_i e_i(x) ,
\end{align*}
where $\{\xi_i\}$ are i.i.d.\ zero-mean random variables with unit variance and $v(s_0, \tau_0)$ is a suitable normalization constant.
Note that the random variables $\{\xi_i\}$ need not be Gaussian.
We allow the likelihood to be misspecified in two ways.
First, we do not assume the family of candidate covariance functions contains the true covariance of $u$.
We consider any suitable family of kernels $\{K_s\}_{s > d/2}$ such that the reproducing kernel Hilbert space of $K_s$ is norm-equivalent to a Sobolev space of order $s$.
Second, in the place of the true likelihood, which may be non-Gaussian, we use the Gaussian log-likelihood 
\begin{align*}
  \ell(s; u(\mb{x})) = -\frac{n}{2}\log(2\pi) - \frac{1}{2}\log\left(\det( K_s(\mb{x})\right) - \frac{1}{2}u(\mb{x})^{\top} K_s(\mb{x})^{-1} u(\mb{x}), 
\end{align*}
where $K_s(\mb{x})$ is the $n \times n$ covariance matrix with entries $(K_s(\mb{x}))_{ij} = K_s(x_i, x_j)$ and $u(\mb{x}) = (u(x_1), \ldots, u(x_n)) \in \mathbb{R}^n$.
Theorem~\ref{thm:consistency} and Corollary~\ref{cor:consistency} can then be formulated somewhat informally as follows.

\begin{theorem*}
  Assume that the sequence of points $\{x_i\}$ is quasi-uniform and that the family of kernels $\{K_s\}_{s > d/2}$ has a certain norm monotonicity property.
  Let $\{\hat{s}_n\}$ be a sequence of maximizers of the Gaussian log-likelihood $\ell(s; u(\mb{x}))$ on a bounded interval containing $s_0$.
  \begin{enumerate}
  \item If $s_0 > d$, then $\hat{s}_n \to s_0$ in probability.
  \item If $s_0 > d/2$ and $\{\xi_i\}$ are Gaussian, then $\hat{s}_n \to s_0$ almost surely.
  \end{enumerate}
\end{theorem*}

These results are of practical importance, as we show that the smoothness parameter can be consistently estimated by using \textit{any} family of kernels whose reproducing kernel Hilbert spaces are the Sobolev scale of Hilbert spaces.
This enables the estimation of smoothness by using, for example, the generalized Wendland class of compactly supported kernels that are computationally cheaper than Matérn kernels.
It also shows that the parameters $\sigma^2$ and $\tau$ do not need to be known to estimate the smoothness.
Moreover, in Section~\ref{sec:magnitude-estimation} we establish an interesting connection between the parameters $\sigma^2$ and $s$, namely that the estimation of $\sigma^2$ compensates to some extent for misspecification of smoothness.
Again, these results do not require any Gaussianity assumptions.
This is similar to the results in~\cite{Karvonen2020, Karvonen2021}; and \citet{Naslidnyk2023}.
Although our main focus is on stochastic processes on Riemannian manifolds, we provide essentially the same results for processes observed on bounded Euclidean domains based on Dirichlet Laplacian operators \citep{Lindgren_2011, bolin_2023}.

All of the above mentioned work on estimating smoothness assumes that the sampling sites are sufficiently uniform in some sense.
We assume that the sampling sites are quasi-uniform, an assumption that is frequently used in the function approximation literature and therefore allows us to exploit the wealth of results obtained there.
Although it is a rather strict assumption, quasi-uniformity is more general than the uniformity assumptions used in~\cite{Chen2021} and \citet{Petit2023}.
As far as we know, no consistency results have yet been obtained under this assumption in the context of estimating covariance parameters for stochastic processes.
However, since the samples of a stochastic process with a Whittle--Matérn covariance kernel are homogeneous in the sense that their smoothness does not vary spatially, \emph{any} sequence of sampling sites should suffice.
Indeed, many results for estimating the microergodic parameter do not include assumptions on the sampling sites~\citep[e.g.,][]{Ying1991, zhang_2004, Kaufman2013, Li2022}.
Refined techniques are needed to avoid uniformity assumptions when estimating smoothness.

The rest of the work is structured as follows.
In Section~\ref{sec:background} we introduce the manifold setting in which we work and give some functional analytical background information.
In Section~\ref{sec:matern_processes_manifold} we introduce our model and provide some regularity results.
Our main results are presented in Section~\ref{sec:MLE}.
Our techniques are similar to those used in \citet{karvonen2023asymptotic} and build on known results from the literature on function approximation.
Furthermore, we provide additional results for the Gaussian case, discuss the estimation with other kernels and briefly the estimation of the microergodic parameters.
In Section~\ref{sec:equivalent_measures} we state our results concerning the equivalence of measures.
The main tool is a theorem of Kakutani~\citep{kakutani_1948} on the equivalence of infinite product measures.
In Section \ref{sec:dirichlet_laplacian} we present a simple translation of our results for processes on bounded Euclidean domains based on Dirichlet Laplacian operators.
Finally, in Section \ref{sec:simulation_study} we present a simulation study assessing our theoretical results regarding the stability of smoothness estimation against model misspecification and evaluating finite sample properties.

We briefly introduce some standard notation.
We equip the Euclidean space $\mathbb{R}^k$ with the Euclidean norm $\lVert \cdot \rVert$.
All other norms are specifically defined.
For two sequences $a_n$ and $b_n$ we write $a_n \lesssim b_n$ if $a_n \leq C b_n$ for a positive constant $C$.
If $a_n \lesssim b_n$ and $b_n \lesssim a_n$, then $a_n \asymp b_n$.
Finally, we generally use $c$ and $C$ for generic positive constants that occur in lower and upper bounds, respectively.
These constants may differ from one line to the next.

\section{Setting and Background}\label{sec:background}

This work focuses on inference for stochastic processes $\{Y(x) ~ \vert~ x \in \mathcal{M}\}$ indexed by a connected, closed, and embedded $d$-dimensional Riemannian submanifold $\mathcal{M}$ of $\mathbb{R}^k$ for some $k \in \mathbb{N}$~\citep[cf. chapter 8;][]{lee_riemannian}.  
A canonical example of such a manifold which is relevant in many applications is the sphere $\mathbb{S}^{d-1}$. 
Very briefly, closed means $\mathcal{M}$ is compact and has no boundary.
Riemannian means that it is equipped with a Riemannian metric $g$ that allows the definition of a metric on $\mathcal{M}$.
Embedded Riemannian submanifold means that $\mathcal{M} \subset \mathbb{R}^k$ and the Riemannian metric $g$ is induced by the Riemannian metric of the ambient space $\mathbb{R}^k$.
We provide a more detailed description of our manifold setting in Appendix \ref{sec:appendix_manifold} and additional background on smooth manifolds can be found in the books ~\citet{lee_smooth, lee_riemannian}.
While assuming that $\mathcal{M}$ is an embedded submanifold of $\mathbb{R}^k$ may appear restrictive, the Nash embedding theorem states that any smooth Riemannian manifold can be isometrically embedded into $\mathbb{R}^k$ for $k$ large enough.

Let $d_{\mathcal{M}}$ be the distance function on~$\mathcal{M}$ based on the Riemannian metric $g$.
Then $(\mathcal{M}, d_{\mathcal{M}})$ is a complete metric space.
We let $\mathrm{d}\lambda_{\mathcal{M}}$ denote the Riemannian volume form, the canonical measure on $\mathcal{M}$ and define $L_2(\mathcal{M})$ as the Hilbert space of equivalence classes of square-integrable functions equipped with the usual inner product $\langle \cdot, \cdot \rangle_0$ and the norm $\lVert \cdot \rVert_0$.
We denote the space of continuous functions by $C(\mathcal{M})$ and the infinitely often differentiable functions by $C^\infty (\mathcal{M})$.

Of some importance to this work is the manifold analog of the Laplacian operator, called the Laplace--Beltrami operator, which we denote by $\Delta$.
The standard definition can be found in Appendix \ref{sec:appendix_manifold}.
By Theorem~1.3 in \citep[Chapter~3]{Craioveanu_2001} there is an orthonormal basis $\{e_i\}$ of $L_2(\mathcal{M})$ consisting of infinitely often differentiable eigenfunctions of $-\Delta$ such that the corresponding eigenvalues $\{\lambda_i\}$ satisfy Weyl's law.
That is, $\lambda_1 = 0$ and there are $C,c > 0$ such that
\begin{align} \label{eq:weyls-law}
    ci^{2/d} \leq \lambda_i \leq Ci^{2/d} \quad \text{ for all } \quad i \geq 2.
\end{align}

\subsection{Function Spaces} \label{sec:function-spaces}

We now present two families of function spaces that are useful for studying the sample paths of stochastic processes.

We define the Sobolev spaces on $\mathbb{R}^d$ as Bessel potential spaces and refer to \citep[Chapters 1 and 2]{triebel_1991} for other equivalent definitions.
Define the Fourier transform $\hat{f}$ of $f \in L_2(\mathbb{R}^d)$ as $\hat{f}(\xi) = \int_{\mathbb{R}^d} f(x)e^{-i x^{\top}\xi} \, \mathrm{d}x$.
The Sobolev space of order $s \geq 0$ is 
\begin{align*}
    H^s(\mathbb{R}^d) = \left\{ f\in L_2(\mathbb{R}^d) ~ \middle \vert ~ \lVert f \rVert_{H^s(\mathbb{R}^d)}^2 \coloneqq  \int_{\mathbb{R}^d} (1 + \lVert \xi \rVert^2)^s \lvert \hat{f}(\xi)\rvert^2 \, \mathrm{d}\xi < \infty \right\}. 
\end{align*}

On closed manifolds we define Sobolev spaces via the eigenpairs $\{(e_i, \lambda_i)\}$ of $-\Delta$.
For other definitions and equivalence results, see Chapter~7 in~\citet{triebel_1991}.
For a real $a$, we write $a - \Delta$ for $aI - \Delta$, where $I$ is the identity.
The operator $1 - \Delta$ must have positive eigenvalues $\{1 + \lambda_i\}$ and from Weyl's law~\eqref{eq:weyls-law} we have $ci^{2/d} \leq 1 + \lambda_i \leq C i^{2/d}$ for all $i \geq 1$.
For $s \in \mathbb{R}$ we define the fractional powers $(1-\Delta)^s$ via the functional calculus 
\begin{align*}
    (1 - \Delta)^s f = \sum_{i = 1}^\infty (1 + \lambda_i)^{s}\langle f, e_i \rangle_0 e_i .
\end{align*}
We may now define the Sobolev space of order $s \geq 0$ on $\mathcal{M}$ as 
\begin{align*}
    H_L^{s} &= \bigg\{f \in L_2(\mathcal{M}) ~ \biggl\vert ~ \lVert f \rVert_{H^s_L}^2 \coloneqq \lVert (1 - \Delta)^{s/2} f \rVert_0^2 =  \sum_{i = 1}^\infty (1 + \lambda_i)^{s} \langle f, e_i \rangle^2_0 < \infty \bigg\} \\
    &= \big\{f \in L_2(\mathcal{M}) ~\bigl\vert ~ (1- \Delta)^{s/2}f \in L_2(\mathcal{M}) \big\}.
\end{align*}
It is well known and easy to show that these spaces are Hilbert spaces with the inner product 
\begin{align*}
    \langle f, g \rangle_{H^s_L} = \langle (1 - \Delta)^{s/2}f, (1 - \Delta)^{s/2}g \rangle_0,
\end{align*}
see \citet{strichartz_83}.
The following result nicely illustrates the connection between Sobolev spaces on manifolds and their ambient Euclidean spaces.
See~\citet{grosse_2013} for a result for more general submanifolds.
\begin{theorem}[Theorem 1 in \citealp{skrzypczak_1990}] \label{thm:manifold_extension_theorem}
    Let $s \geq 0$. The restriction operator $T_{\mathcal{M}} f = f\vert_{\mathcal{M}}$ is a bounded linear operator $T_{\mathcal{M}}: H^{s + (k-d)/2}(\mathbb{R}^k) \to H^s_L$.
    Moreover, $T_{\mathcal{M}}$ has a bounded inverse operator $E_{\mathcal{M}}: H^s_L \to H^{s + (k-d)/2}(\mathbb{R}^k)$.
\end{theorem}

In the definition of $H^s_L$, we could have used the operator $\tau - \Delta$ instead for any $\tau > 0$.
Clearly, the resulting spaces are norm equivalent.
In the statistical context considered in the remainder of the paper, $\tau$ has a concrete interpretation as a correlation range scale parameter.
To account for $\tau$, we introduce a modified inner product and norm on $H_L^s$.
Let $\theta = (s, \tau)$ and 
\begin{align*}
    \langle f, g \rangle_{\theta} \coloneqq \frac{1}{v(\theta)} \big\langle (\tau - \Delta)^{s/2}f, (\tau - \Delta)^{s/2}g \big\rangle_0 \quad \text{ and } \quad \lVert f \rVert_\theta = \sqrt{\langle f, f \rangle_\theta},
\end{align*}
where $v(\theta)$ is a positive ``normalizing'' function.
It is easy to show that the norms $\lVert \cdot \rVert_{H^s_L}$ and $\lVert \cdot \rVert_{\theta}$ are equivalent.
The function $v(\theta)$ will be discussed more extensively in Section~\ref{sec:kernel_properties}.

We also need the following closely related sequences spaces.
For $s \in \mathbb{R}$, define the scale of Hilbert spaces \smash{$\hat{H}_L^s$} and their inner product as
\begin{align*}
    \hat{H}^{s}_L = \bigg\{ a = \{a_i\}_{i = 1}^\infty \subset \mathbb{R} ~ \biggl\vert ~ \sum_{i = 1}^\infty a_i^2 (1 + \lambda_i)^s < \infty \bigg\} \:\: \text{ and } \:\: \langle a, b \rangle_{\hat{\theta}} = \frac{1}{v(\theta)}\sum_{i = 1}^\infty (\tau + \lambda_i)^s a_ib_i.
\end{align*}
For $s \geq 0$ we have the obvious isomorphism \smash{$\iota : H^{s}_L \to \hat{H}^{s}_L$} given by $(\iota(f))_i = \langle f, e_i \rangle_0$.
Observe that \smash{$\hat{H}^0_L$} equals the Hilbert space of square summable sequences.

Finally for $l \in \mathbb{N}, 0 \le \alpha < 1$, we define $C^{l, \alpha}(\mathcal{M})$ as $l$-times continuously differentiable functions such that $l$-th derivative is $\alpha$-Hölder continuous (a precise definition can be found in Appendix \ref{sec:appendix_manifold}). 
The Hölder spaces will appear in connection with sample path properties of Gaussian processes.
Perhaps more importantly for our purposes, the Sobolev embedding theorem states that the Sobolev space $H^{s}_L$ for $s > d/2$ can be continuously embedded in $C^{l, \alpha}(\mathcal{M})$ for all $ l + \alpha < s - d/2$; see \citet[Theorem~2.20]{aubin_1998} or \citet[Section~3.3]{hangelbroek_2010}.
This implies that the point evaluation functionals on $H^{s}_L$ are bounded and thus the spaces $H^s_L$ are reproducing kernel Hilbert spaces (RKHS) for $s > d/2$.
We identify the reproducing kernels, which are precisely the Whittle-Mat\'ern Kernels, and discuss their properties in the next section.

\subsection{Kernels}\label{sec:kernel_properties}
Note that the operator $\sqrt{\smash[b]{v(\theta)}}(\tau - \Delta)^{-s/2} : L_2(\mathcal{M}) \to H^{s}_L \subset L_2(\mathcal{M})$ is Hilbert--Schmidt for $s > d/2$ as 
\begin{align*}
    \sum_{i = 1}^\infty \big\lVert \sqrt{\smash[b]{v(\theta)}}(\tau - \Delta)^{-s/2} e_i \big\rVert_0^2 = v(\theta)\sum_{i = 1}^\infty(\tau + \lambda_i)^{-s} \leq Cv(\theta) \sum_{i = 1}^\infty  i^{- 2s/d } < \infty.
\end{align*}
It thus follows \citep[Theorem VI.23]{reed_1980} that $\sqrt{\smash[b]{v(\theta)}}(\tau - \Delta)^{-s/2}$ is an integral operator with the corresponding kernel $K_\theta \in L_2(\mathcal{M} \times \mathcal{M})$ given by 
\begin{align}
    K_{\theta}(x, y) = v(\theta)\sum_{i = 1}^\infty (\tau + \lambda_i)^{-s}e_i(x) e_i(y) \label{eq:kernel}.
\end{align}
The kernel $K_\theta$ is the reproducing kernel of $H^{s}_L$ endowed with the $\langle \cdot, \cdot \rangle_{\theta}$ inner product.
The pointwise definition makes sense as $e_i$ can be identified with $C^{\infty}(\mathcal{M})$ functions.
We write $K_{\theta}(x)$ for $K_{\theta}(x,x)$.
The following proposition collects some facts about $K_\theta$.
Additional regularity properties of the kernel can be established via the results obtained in \citet{Kerkyacharian_2018}.

\begin{proposition}\label{prop:kernel_properties}
    Let $K_{\theta}$ be the kernel defined in \eqref{eq:kernel} with $s > d/2$.
    Then 
    \begin{enumerate}[label=(\roman*)]
        \item $K_{\theta}$ is continuous and the series in \eqref{eq:kernel} converges uniformly and absolutely. \label{enum:kernel_prop_cont}
        \item $K_{\theta}$ is strictly positive definite. \label{enum:kernel_prop_pd}
    \end{enumerate}
\end{proposition}

\begin{proof}
    Let $\alpha \in (0,1)$ be such that $\alpha < s-d/2$.
    We first prove \ref{enum:kernel_prop_cont}.
    According to the Sobolev embedding theorem, $H^{s}_L \subset C^{0, \alpha}(\mathcal{M})$ and the inclusion operator is continuous.
    It is a well-known consequence of the Arz\`ela--Ascoli theorem that the embedding of $C^{0, \alpha}(\mathcal{M})$ in $C(\mathcal{M})$ is compact, and thus also the embedding of $H^{s}_L$ in $C(\mathcal{M})$.
    The continuity of $K_\theta$ then follows from Lemma~2.2 in~\citet{ferreira_2013}.
    The remaining assertions in \ref{enum:kernel_prop_cont} follow from the continuity of the kernel and the compactness of $\mathcal{M}$ in combination with Dini's theorem; see for example Lemma 4.6.6 in \citet{hsing_2015}.
    To prove \ref{enum:kernel_prop_pd}, we repeat an argument from \citet{hangelbroek_2010}.
    Positive semi-definiteness is obvious.
    For $\mb{x} = (x_1, \dots, x_n)$, define $(K_{\theta}(\mb{x}))_{ij} = K_{\theta}(x_i, x_j)$.
    A straightforward calculation yields
    \begin{align*}
        \mb{c}^{\top}K_{\theta}(\mb{x}) \mb{c} = \sum_{i, j = 1}^n \langle c_i K_{\theta}(\cdot, x_i), c_jK_{\theta}(\cdot, x_j) \rangle_{\theta}^2 = \left\lVert \sum_{i = 1}^n c_i K_\theta(\cdot, x_i) \right\rVert_{\theta}^2.
    \end{align*}
    Assume that $x_i \neq x_j$ for $i \neq j$.
    If the right hand side above is 0 for a non-zero $\mb{c}$, then $\{K_{\theta}(\cdot, x_i)\}$ are \textit{not} linearly independent. 
    This means that, for every $f \in H^{s}_L$,
    \begin{align*}
        f(x_i) = \langle f, K_{\theta}(\cdot, x_i) \rangle_{\theta} = \sum_{j \neq i}\langle f, a_jK_{\theta}(\cdot, x_j) \rangle_{\theta} = \sum_{j \neq i} a_j f(x_j)
    \end{align*}
    for some coefficients $a_j$.
    Because $H^s_L$ contains bump functions such that $f(x_i) \neq 0$ and $f(x_j) = 0$ for $j \neq i$, this yields a contradiction and thus $K_\theta$ is strictly positive definite.
\end{proof}

In addition to $s$ and $\tau$, we consider a magnitude parameter $\sigma^2$, which is included in the model via the scaled kernel function $(x, y) \mapsto \sigma^2 K_{\theta}(x,y)$.

At this point, we want to briefly discuss the normalizing function $v(\theta)$.
Different choices for $v(\theta)$ have been advocated for in the literature.
The function $v(\theta) = \tau^{-s+d/2}$ is used in~\citet{sanz-alonso_2022}, who provide a compelling justification in their Remark~2.1: This choice balances the average marginal variance of a process with covariance function $K_{\theta}(x, y)$ and thus the parameter $\sigma^2$ approximately controls the marginal variance.
The function $v(\theta) = \tau^{-s+d/2}$ also appears in the spectral density of the Euclidean Mat\'ern covariance function given in~\eqref{eq:matern_cov}.
Similarly motivated is the use in \citet{borovitsky_2020} and \citet{li_2023} of the normalizing function
\begin{align} \label{eq:norm-func-narc-boro}
    v(\theta) = \bigg( \frac{1}{\mathrm{vol}(\mathcal{M})}\int_{\mathcal{M}} \sum_{i = 1}^\infty (\tau + \lambda_i)^{-s}e_i(x)^2 \, \mathrm{d} \lambda_{\mathcal{M}} \bigg)^{-1}.
\end{align}
On certain manifolds (e.g., the sphere) $K_{\theta}(x)$ does not depend on $x$~\citep{narcowich_2002, borovitsky_2020}.
With $v(\theta)$ in~\eqref{eq:norm-func-narc-boro}, $\sigma^2$ is exactly the marginal variance.
However, explicit computation of this normalization is difficult.
We do not choose a specific form so that $v(\theta)$ is allowed to depend on the underlying manifold or on specific modelling choices.
Nonetheless, for our results we will require that the Sobolev norms are monotone in the sense that there exists a constant $C$ such that for any $\theta = (s, \tau), \theta^\prime=(s^\prime, \tau^\prime)$, we have 
\begin{align}
    \| f \|_\theta \le C \| f \|_{\theta^\prime} \label{eq:monotone_norms}
\end{align}
whenever $s \le s^\prime$.
If $\theta$ varies only in a compact set $\Theta$ uniformly bounded away from the axis $\tau = 0$ and $s = d/2$ and $v(\theta)$ is positive and continuous, then it is easy to show that \eqref{eq:monotone_norms} holds.

\section{Whittle-Mat\'ern Processes on Closed Manifolds}\label{sec:matern_processes_manifold}
We will consider the following random elements in $L_2(\mathcal{M})$.
Let $(\Omega, \mathcal{F}, \P)$ be a probability space and let $\{\xi_i\}$ be a sequence of i.i.d. random variables defined on $\Omega$ such that $\E(\xi_i) = 0$ and $\E(\xi_i^2) =1$.
Set
\begin{align}
    u =  \sigma \sqrt{\smash[b]{v(\theta)}} \sum_{i = 1}^\infty (\tau + \lambda_i)^{-s/2} \xi_i e_i \label{eq:karhunen_loeve}. 
\end{align}
For $s > d/2$ this is indeed an element of $L_2(\mathcal{M})$.

\begin{proposition}\label{prop:smoothness_class}
  The random element $u$ defined in~\eqref{eq:karhunen_loeve} is almost surely in $H^{s-d/2-\epsilon}_L$ for all $s-d/2 > \epsilon > 0$.
    In particular, $u \in L_2(\mathcal{M})$ if $s > d/2$.
\end{proposition}
\begin{proof}
    Let $s,\epsilon > 0$ be such that $s - d/2 - \epsilon > 0$, $\tau > 0$ and $\theta = (s-d/2 -\epsilon, \tau)$.
    Let $u_n$ denote the partial sums up to term $n$ and let $A_n = \| u_n \|_{\theta}^2.$
    By monotonicity, $A := \sup_n A_n$ exists and the monotone convergence theorem allows us to compute
    \begin{align*}
        \E (A) = \E (\sup_{n} \| u_n \|^2_\theta) &=\lim_{n \to \infty}\E \bigg\lVert \sigma \sqrt{v(\theta)} \sum_{i = 1}^n (\tau + \lambda_i)^{-s/2} \xi_i e_i \bigg\rVert_{\theta}^2\\
        &= \lim_{n \to \infty} \sigma^2 \sum_{i = 1}^n (\tau + \lambda_i)^{s-d/2-\epsilon - s}\E\left( \xi_i^2\right) \\
        &\leq \lim_{n \to \infty} C \sum_{i = 1}^n  (i^{2/d} )^{-d/2 - \epsilon} \\
        &= \lim_{n \to \infty}C \sum_{i = 1}^n  i^{-1-2\epsilon/d} < \infty,
    \end{align*}
    where we used Weyl's law \eqref{eq:weyls-law} in Line 3.
    In particular, $A$ is almost surely finite so that $A - A_n 
    \to 0$ almost surely. 
    But then for any $N \le n \le m$, 
    \begin{align*}
        \| u_n - u_m \|^2_{\theta} \le A- A_N \to 0, 
    \end{align*}
    showing that $\{u_n\}$ is almost surely a Cauchy sequence in $H_L^{s-d/2 - \epsilon}$ with respect to the equivalent norm $\| \cdot \|_\theta$ and by completeness has thus limit $u \in H_L^{s-d/2 - \epsilon}$.
\end{proof}
Note that $\E(u) = 0$.
Furthermore, the covariance operator of $u$ is by construction
\begin{align*}
 \Cov(u) = \E (u \otimes u) = \sigma^2 v(\theta)(\tau - \Delta)^{-s},
\end{align*}
where for two elements $v_1, v_2 \in L_2(\mathcal{M})$ $(v_1 \otimes v_2)(f) =\langle v_1, f \rangle_0 v_2$. 
If $s > d$, the Sobolev embedding theorem guarantees that $u$ almost surely has continuous sample paths so that the measurability of the pointwise evaluations is not a problem.
Normally, however, the smoothness parameter is considered over the range $s > d/2$, which only guarantees that $u$ is an element in $L_2(\mathcal{M})$.
Fortunately, it follows from the results of Section \ref{sec:kernel_properties} that the assumptions of \citet[Theorem~3.3]{steinwart_2019} are fulfilled for any $s > d/2$.
This theorem ensures the existence of a measurable version $\tilde{u}$ of $u$ such that $\Cov(\tilde{u}(x), \tilde{u}(y)) = \sigma^2 K_{\theta}(x, y)$
and, moreover, $\tilde{u}$ admits the Karhunen--Lo\`eve expansion 
\begin{align*}
    \tilde{u}_{\omega}(x) = \sum_{i = 1}^\infty \langle \tilde{u}_{\omega}, e_i \rangle_0 e_i(x) = \sigma \sqrt{\smash[b]{v(\theta)}}\sum_{i = 1}^\infty  (\tau + \lambda_i)^{-s/2} \xi_i(\omega) e_i(x)
\end{align*}
for $\P$-almost all $\omega$.
Here measurability of $\tilde{u}$ means that the map $(x, \omega) \mapsto \tilde{u}_{\omega}(x)$
is measurable with respect to the product $\sigma$-algebra $\mathcal{F} \otimes \mathcal{B}(\mathcal{M})$, where $\mathcal{B}(\mathcal{M})$ is the $\sigma$-algebra generated by the open subsets of $\mathcal{M}$.
We may therefore assume that point evaluations of $u$ are measurable.

An alternative way to define $u$ is to define it as the (unique) solution to the Whittle--Mat\'ern stochastic partial differential equation
\begin{align}
    (\tau - \Delta)^{s/2}u = \sigma \sqrt{\smash[b]{v(\theta)}} \mathcal{W}, \label{eq:matern_spde}
\end{align}
where $\mathcal{W}$ is white noise.
This approach is due to Peter Whittle \citep{whittle_1963} and is the approach taken in \cite{Lindgren_2011, bolin_2014, bolin_2020}.
In both works more general differential operators than $(\tau - \Delta)$ are considered.
To make sense of \eqref{eq:matern_spde} and see that this yields the same model as \eqref{eq:karhunen_loeve}, define white noise to be any sequence of i.i.d. centered random variables $\mathcal{W} = \{\xi_i\}$ with finite second moments and then note that (i) for $\theta = (s, \tau)$ where $s < -d/2$, the norm $\lVert \mathcal{W} \rVert_{\hat{\theta}}$ is finite almost surely by the argument used in the proof of Proposition~\ref{prop:smoothness_class}, and (ii) the operator $\iota^{-1}\circ (\tau - \hat{\Delta})^s$ defined through
\begin{align*}
    ( \iota^{-1} \circ (\tau - \hat{\Delta})^{-s}) (a) = \sum_{i = 1}^\infty (\tau + \lambda_i)^{-s}a_i e_i
\end{align*}
is an isometric isomorphism from $\hat{H}^{\beta}_L$ to $H_L^{\beta + 2s}$ for all $\beta \geq -2s$.
A more standard construction can be found in~\cite{Lindgren_2011} or~\cite{bolin_2014}.

We emphasize that we do \emph{not} assume that the $\xi_i$ are Gaussian random variables.
Unless otherwise specified, we only assume that $\xi_i$ have finite second moments.
However, we do assume that the random variables are independent.
In the Gaussian case, this is not a restriction.
Any mean zero Gaussian process with covariance function $K_{\theta}$ admits a Karhunen--Lo\`eve decomposition \eqref{eq:karhunen_loeve} with independent $\{\xi_i\}$.
For non-Gaussian processes, $\{\xi_i\}$ are only guaranteed to be uncorrelated.

\section{Gaussian Maximum Likelihood Estimation}\label{sec:MLE}

We want to recover the smoothness parameter $s$ via maximum likelihood estimation.
Fix therefore the true parameter vector $\theta_0 = (s_0, \tau_0)$ with $s_0 > d/2$ and let $u$ be defined by 
\begin{align}
    u(x) = \sigma_0 \sqrt{\smash[b]{v(\theta_0)}} \sum_{i = 1}^\infty (\tau_0 + \lambda_i)^{-s_0/2}\xi_i e_i(x), \label{eq:true_model}
\end{align}
where $\{\xi_i\}$ are i.i.d.\ random variables with $\E(\xi_i) = 0$ and $\E(\xi_i^2) = 1$.

Assume that we observe the process $u$ at pairwise distinct points $(x_1, \dots, x_n) = \mb{x}$.
Using this notation, we write $u(\mb{x})^{\top} = (u(x_1), \dots, u(x_n))$.
To estimate the smoothness parameter $s_0$ of the process $u$, we will use the maximum likelihood estimators.
However, we allow the likelihood to be misspecified in two different ways.
First, instead of requiring that the true family of covariance kernels $\sigma^2 K_{\theta}$ is used, we only require the use of a family of kernels $\{K_s\}_{s > d/2}$ such that for each $s$ their RKHS $\mathcal{H}^s$ with the norm $\lVert \cdot \rVert_s$ is norm equivalent to $H^s_L$.
This includes the correctly specified kernels $\sigma^2K_{\theta}$ with $\sigma \neq \sigma_0$ or $\tau \neq \tau_0$.
Moreover, it turns out that there are other well-known and useful kernels whose RKHS are norm equivalent to $H^s_L$.
We discuss some of these kernels in Section~\ref{sec:other_kernels}.
Second, we always use the Gaussian log-likelihood 
\begin{align}
    \ell(s; u(\mb{x})) = -\frac{n}{2}\log(2\pi) - \frac{1}{2}\log\left(\det( K_s(\mb{x})\right) - \frac{1}{2}u(\mb{x})^{\top} K_s(\mb{x})^{-1} u(\mb{x}). \label{eq:gaussian_ll}
\end{align}
Here, we write $K_s(\mb{x})$ for the matrix with entries $(K_s(\mb{x}))_{ij} = K_s(x_i, x_j)$.
Of course, this can be a correctly specified likelihood only if $u$ is a Gaussian process. 

It is well known that the spatial sampling of the points $x_1, \dots, x_n$ has significant impact on the quality of maximum likelihood estimators~\citep[e.g.,][]{bachoc_2014}.
To establish consistency of the maximum likelihood estimator, we will make a fairly restrictive assumption on the uniformity of sampling.
However, this assumption is standard in the function approximation literature~\citep[e.g.,][]{wendland_2004, narcowich_2006} and thus allows us to use the wealth of results established there.
Moreover, all existing results on smoothness estimation for Gaussian processes require some type of uniformity assumption~\citep{Chen2021, loh_2021, LohSun2023, Petit2023, karvonen2023asymptotic}.
Define the \textit{fill-distance} $h_{\mb{x}}$ and the \textit{separation radius} $q_{\mb{x}}$ as 
\begin{align} \label{eq:fill-distance}
    h_{\mb{x}} = \sup_{x \in \mathcal{M}} \min_{1 \leq i \leq n} d_{\mathcal{M}}(x, x_i) \quad \text{ and } \quad q_{\mb{x}} = \frac{1}{2}\min_{i \neq j} d_{\mathcal{M}} (x_i, x_j).
\end{align}
Also define the \emph{mesh norm} (or \emph{mesh ratio}) $\rho_{\mb{x}} = h_{\mb{x}} / q_{\mb{x}}$.
Note that by definition $q_{\mb{x}} \leq h_{\mb{x}}$ and therefore $\rho_{\mb{x}} \geq 1$.
Our consistency results assume that the point sequence is quasi-uniform.

\begin{definition}
    A sequence of points $\{x_i\} \subset \mathcal{M}$ is \textit{quasi-uniform} if
    \begin{align*}
        h_{\mb{x}} \asymp q_{\mb{x}} \asymp n^{-1/d}.
    \end{align*}
    This implies $\rho_{\mb{x}}$ is uniformly bounded from above.
\end{definition}

Note that uniformly sampled random points are \emph{not} quasi-uniform.
In \citet[Section~4]{loh_2021}, an estimator for the smoothness of Gaussian random fields on bounded Euclidean domains with the Mat\'ern covariance function is introduced, which is consistent when $x_i$ are sampled from a density that is uniformly lower bounded on the domain.
However, the estimator there effectively uses only a subset of the sampling sites that are quasi-uniform up to factors of $\log(n)$, but whose cardinality is of order $n$.
A modification of the construction in \citet[Section~4]{loh_2021} for the manifold $\mathcal{M}$ could be used to establish variants of our consistency results for random points.

Our proofs are based on connecting the quadratic form $u(\mb{x})^{\top}K_s(\mb{x})^{-1} u(\mb{x})$ and the log-determinant $\log(\det(K_s(\mb{x}))$ in the Gaussian likelihood~\eqref{eq:gaussian_ll} to quantities related to optimal interpolation in the RKHS $\mathcal{H}^s$.
Namely, the quadratic form is the squared RKHS norm of the \textit{minimum norm interpolant} and the log-determinant is related to \textit{worst-case approximation error}.
Let $K$ be a positive definite reproducing kernel of an RKHS $\mathcal{H}$ and $u$ a function that has been observed at the points $x_1, \dots, x_n$.
The unique minimum norm interpolant $m_{\mathcal{H},n}$ of $u$ is
\begin{align*}
    m_{\mathcal{H},n} = \argmin_{f \in \mathcal{H}} \big\{ \lVert f \rVert_{\mathcal{H}} ~ \vert~ f(x_i) = u(x_i), ~ 1 \leq i \leq n \big\}.
\end{align*}
We write $m_{s,n}$ for the minimum norm interpolant based on the kernel $K_s$.
It is well known and easy to verify that the quadratic form in~\eqref{eq:gaussian_ll} equals the squared norm of this interpolant:
\begin{align}
    \lVert m_{s,n} \rVert_{s}^2 = u(\mb{x})^{\top}K_{s}(\mb{x})^{-1}u(\mb{x}). \label{eq:minimum_norm_interpolant_quad_form}
\end{align}
For a proof, see for example Theorem 3.5 in \citet{KanHenSejSri18}.

For the connection to worst-case approximation error, define the ``conditional variances''
\begin{align} \label{eq:gp-conditional-variance}
     \V_s(y \vert \mb{x}) = K_s(y) - K_s(y, \mb{x})K_s(\mb{x})^{-1}K_s(\mb{x}, y).
\end{align}
Here, $K_s(y, \mb{x})$ is the $1 \times n$ vector with entries $K_s(y, \mb{x})_j = K_s(y, x_j)$.
If $u$ is a Gaussian process with covariance function given by $K_s$, this is indeed the conditional variance of $u(y)$ given $\{u(x_j) \: \vert \: 1 \leq j \leq n\}$.
The variance $\V_s(\cdot \vert \mb{x})$ is the worst-case approximation error in that
\begin{align*}
    \sup_{\lVert f \rVert_{\mathcal{H}^s} \leq 1} \lvert f(y) - m_{s, n}(y) \rvert = \sqrt{\smash[b]{\V_s(y \vert \mb{x})}}.
\end{align*}
See, for example, \citet[Corollary~3.11]{KanHenSejSri18}.
Now, let $\mb{x}^m$ denote the set consisting of $x_1, \dots, x_m$.
Using the Shermann--Morrison--Woodbury identity and~\eqref{eq:gp-conditional-variance} repeatedly, we write
\begin{equation}
\begin{aligned}
    \det(K_s(\mb{x})) &= \left(K_s(x_n) - K_s(x_n, \mb{x}^{n-1})K_s(\mb{x}^{n-1})K_s(\mb{x}^{n-1}, x_n) \right) \det(K_s(\mb{x}^{n-1})) \\
    &= \V_s(x_n \vert \mb{x}^{n-1}) \det(K_s(\mb{x}^{n-1}))\\
    &= \prod_{i = 1}^n \V_s(x_i \vert \mb{x}^{i-1}). \label{eq:worst_case_approx_error}
\end{aligned}    
\end{equation}
The identities \eqref{eq:minimum_norm_interpolant_quad_form} and \eqref{eq:worst_case_approx_error} show that precise control of the $\V_s(y \vert \mb{x})$ and $\lVert m_{s,n} \rVert_{s}$ allows for precise control over the log-likelihood \eqref{eq:gaussian_ll}.
We next establish bounds for these quantities.
From here on, we fix a family of kernels $K_s$, their reproducing kernel Hilbert spaces $(\mathcal{H}^s, \lVert \cdot \rVert_s)$ required to be equal to $H^s_L$ up to equivalent norms and the corresponding minimum norm interpolants $m_{s,n}$ of $u$.

\subsection{Norm and Determinant Bounds}\label{sec:norm_det_bounds}

In this section we present a number of auxiliary results through which we obtain precise control over the log-determinant and the quadratic from appearing the Gaussian log-likelihood.
The proofs are moved to \ref{sec:tech_proofs} as they are either technical in nature or modifications of well-known results.
We assume throughout that $u$ is defined by \eqref{eq:true_model}.
To simplify the notation, we use $u$ for both the random element and its sample paths.

First, we set up bounds for $\lVert m_{s, n}\rVert_{s}$.
The size of the norm depends on the smoothness of $u$.
If $u$ is very smooth, then $\lVert m_{s, n} \rVert_s$ is small because $m_{s,n}$ is ``simple''; if $u$ is very coarse, then the norm is large because $m_{s,n}$ must be ``complicated''.
The following proposition restricts how smooth $u$ can be by providing a lower bound for the decay rate of its Fourier coefficients.
The proofs are given in the Appendix.

\begin{proposition}\label{prop:smoothness_limit}
  Let $s_0 > d/2$ and let $\{\xi_i\}$ be the sequence of i.i.d. random variables used in the definition of $u$ in \eqref{eq:true_model}.
    Then almost surely, there is a positive constant $c$ such that 
    \begin{align*}
        \sum_{i = R}^\infty (\tau_0 + \lambda_i)^{-s_0} \xi_i^2 \geq c R^{-2s_0/d + 1} \quad \text{ for all } \quad R \in \mathbb{N}.
    \end{align*} 
\end{proposition}

Next, we introduce a lower bound for $\lVert m_{s, n} \rVert_{s}$ under an approximation condition.
That this condition holds for some useful kernels, in particular for $\sigma^2 K_{\theta}$, will be confirmed later.

\begin{proposition}\label{prop:lower_bound_norm_interpolant}
  Let $\theta = (s, \tau)$ with $s > d/2$.
  Suppose that $f \in \mathcal{H}^s$ and, almost surely,
    \begin{align}
        \lVert u - f \rVert_0 \leq C_1 n^{- s_0/d + 1/2 +\epsilon/d} + C_2 n^{-s/d} \lVert f \rVert_s \label{eq:assumption}
    \end{align}
    for every $\epsilon > 0$, where $C_1,C_2 > 0$ can depend on $s, s_0, \epsilon,\mathcal{M}$ and the sample path.
    Then almost surely
    \begin{align}
        \lVert f \rVert_{s}^2 \geq Cn^{1 + 2(s-s_0)/d - \epsilon^\prime}
    \end{align} 
    for every $\epsilon^\prime > 0$, where $C > 0$ depends on $s, s_0,\epsilon^\prime, \mathcal{M}$ and the sample path.
\end{proposition}
\begin{remark}
    We would like to discuss the assumption in~\eqref{eq:assumption}, which may seem unusual.
    We are mainly interested in applying Proposition~\ref{prop:lower_bound_norm_interpolant} with $f = m_{s, n}$. In this scenario one commonly has the error estimate
    \begin{align*}
        \lVert u - m_{s,n} \rVert_0 \leq C n^{- s_0/d + 1/2 +\epsilon/d} \lVert u \rVert_{s_0-d/2-\epsilon},
    \end{align*}
    which clearly implies \eqref{eq:assumption}, and we will set up a similar bound in Proposition \ref{prop:escape_manifold} below.
 Note, however, that we do not assume that $f = m_{s,n}$.
 The proposition holds for any $f$ that approximates $u$ sufficiently well.
 This additional flexibility will be important later.
\end{remark}

The following proposition shows that the assumption for Proposition \ref{prop:lower_bound_norm_interpolant} is generally satisfied if $s \geq s_0$.
Since they apply to functions outside the RKHS of $\mathcal{H}^s$, results of this type are often referred to as ``escape'' results and are based on the fundamental work by \citet{narcowich_2002} and \citet{narcowich_2006}.
The statement and the proof are essentially the same as in~\citet{fuselier_2012}.
For completeness, we provide the proof in the Appendix.
The proof combines known results for the Euclidean case from \citet{narcowich_2006} with Theorem~\ref{thm:manifold_extension_theorem}.

\begin{proposition}\label{prop:escape_manifold}
    Let $f \in H^{s_0}_L$ with $s \geq s_0 > d$ and let $m_{s,n}^f$ be its minimum norm interpolant from the RKHS $\mathcal{H}^s$.
    Then there exists a positive constant $C$ that depends on $s_0$ and $\mathcal{M}$ such that we have almost surely
    \begin{align}
        \lVert f - m_{s,n}^f \rVert_{0} \leq C h_{\mb{x}}^{s_0} \rho_{\mb{x}}^{s-s_0} \lVert f \rVert_{s_0} \label{eq:escape_inequality}
    \end{align}
    when $h_{\mb{x}}$ is small enough.
    Thus, if $h_{\mb{x}} \to 0 $ and $C$ is allowed to depend on $f$, then \eqref{eq:escape_inequality} holds for all $n$.
\end{proposition}

See Lemma~3.8 in \citet{hangelbroek_2010} for a precise statement on how small $h_{\mb{x}}$ needs to be.
Under quasi-uniformity, $h_{\mb{x}} \leq C n^{-1/d}$ and $\rho_{\mb{x}}$ is uniformly upper bounded and for $f = u$ we thus obtain
\begin{equation}
\begin{aligned}
    \lVert u - m_{s,n} \rVert_{0} &\leq C h_{\mb{x}}^{s_0-d/2-\epsilon} \rho_{\mb{x}}^{s-(s_0 - d/2-\epsilon)}  \lVert u \rVert_{s_0 - d/2-\epsilon} \\
    &\leq C n^{-s_0/d + 1/2 + \epsilon /d}  \lVert u \rVert_{s_0 - d/2 - \epsilon} \label{eq:concrete_escape}
\end{aligned}
\end{equation}
for all $\epsilon > 0$, where $C$ can depend on the sample path.
Together, Propositions \ref{prop:lower_bound_norm_interpolant} and \ref{prop:escape_manifold} therefore yield the almost sure lower bound
\begin{equation} \label{eq:mn-lower-bound-auxiliary}
    \lVert m_{s,n} \rVert_{s}^2 \geq C n^{1 + 2(s-s_0)/d - \epsilon}
\end{equation}
for any $\varepsilon > 0$ when $s \geq s_0$ and the sequence of points is quasi-uniform.
Next we provide an upper bound.

\begin{proposition}\label{prop:upper_bound_minimum_norm_interpolant}
    Let $u$ be as in \eqref{eq:true_model}.
    Suppose that $s_0 > d$ and $s > d/2$ satisfy $s \geq s_0 - d/2$.
    Let $m_{s,n}$ be the minimum norm interpolant for $u$.
    Then almost surely
    \begin{align}
        \lVert m_{s,n} \rVert_s^2 \leq C q_{\mb{x}}^{s_0-s-d/2 - \epsilon} \label{eq:norm_upper_bound_interpolant-q}
    \end{align}
    for all $\epsilon > 0$, where $C$ depends on $s, \theta_0, \epsilon$ and the sample path.
    If $\xi_i$ are Gaussian, then it is only necessary to assume $s_0 > d/2$ to obtain
    \begin{align}
        c n \leq \lVert m_{s_0, n} \rVert_{s_0}^2 \leq C n \label{eq:norm_upper_bound_interpolant_gaussian}
    \end{align}
    almost surely for positive constants $c$ and $C$ that depend only on $\theta_0$ and the sample path.
\end{proposition}

See the Appendix for a proof.
The bound in~\eqref{eq:norm_upper_bound_interpolant-q} is well known in the Euclidean setting and its proof only needs to be slightly adapted.
Note that under quasi-uniformity this bound becomes
\begin{align}
    \lVert m_{s,n} \rVert_s^2 \leq C n^{1 + 2(s-s_0) + \epsilon} \label{eq:norm_upper_bound_interpolant}
\end{align}
The bound in~\eqref{eq:norm_upper_bound_interpolant_gaussian} follows from a simple argument based on the law of large numbers.
In the Gaussian case, the bound~\eqref{eq:norm_upper_bound_interpolant_gaussian} is strictly stronger than \eqref{eq:norm_upper_bound_interpolant-q} if $s = s_0$.

Finally, we obtain bounds for the worst-case approximation error.
These results are standard and we merely adapt known proofs for manifolds.
The proofs are included in the Appendix.

\begin{proposition}\label{prop:cond_var_bounds}
    Let $s > d/2$.
    \begin{enumerate}
        \item  Suppose the family of norms $\{\lVert \cdot \rVert_s\}_{s > d/2}$ is such that $\lVert f \rVert_{s} \leq C \lVert f \rVert_{s^\prime}$ for every $f$ whenever $s \leq s^\prime$.
          Let $\mb{x}^n_i$ be all points in $\mb{x}$ except $x_i$.
          Then, for $s \leq S$, there is a constant $c > 0$, that only depends on $S$, such that
        \begin{align*}
            \min_{1 \leq i \leq n}\V_{s} (x_i \vert \mb{x}^n_i) \geq c q_{\mb{x}}^{2S-d}.
        \end{align*}
        \item There exists a positive constant $C$, that only depends on $s$ and $\mathcal{M}$, such that
        \begin{align*}
            \sup_{y \in \mathcal{M}}\V_{s}(y \vert \mb{x}) \leq C h_{\mb{x}}^{2s-d}.
        \end{align*}
    \end{enumerate}
\end{proposition}

\subsection{Estimation of the Smoothness Parameter}
Based on the previous auxiliary results, we may now prove the main result of this paper.

\begin{theorem}\label{thm:consistency}
  Let $u$ be defined as in \eqref{eq:true_model} with $s_0 > d$.
  Let the family of kernels $\{K_{s}\}_{s> d/2}$ be such that the associated norms $\{\lVert \cdot \rVert_s\}_{s > d/2}$ satisfy the monotonicity condition that, whenever $s \leq s^\prime$,
    \begin{align}
        \lVert f \rVert_s \leq C\lVert f \rVert_{s^\prime} \label{eq:norm_monotonicity}
    \end{align}
    for some positive $C$ that does not depend on $s$ and $s^\prime$.    
    Assume that the sequence of points $\{x_i\}$ is quasi-uniform.
    If $\{\hat{s}_n\}$ is a sequence of maximizers of the Gaussian log-likelihood
    \begin{align}
        \ell(s; u(\mb{x})) = -\frac{n}{2}\log(2\pi) - \frac{1}{2}\log(\det (K_{s}(\mb{x}))) - \frac{1}{2}u(\mb{x})^{\top}K_{s}(\mb{x})^{-1}u(\mb{x}) \label{eq:log_lik_thm}
    \end{align}
    in $(d/2, S_{\max}]$ with $s_0 \leq S_{\max}$ for arbitrary but fixed $S_{\max}$, then
    \begin{align*}
        \hat{s}_n \to s_0 \quad \text{ in probability}.
    \end{align*}
\end{theorem}

The proof is based on using the results from Section \ref{sec:norm_det_bounds} to balance the log-determinant and the quadratic form appearing in the Gaussian log-likelihood~\eqref{eq:log_lik_thm}.
By~\eqref{eq:mn-lower-bound-auxiliary},~\eqref{eq:norm_upper_bound_interpolant} and~\eqref{eq:norm_monotonicity}, the quadratic form behaves approximately as $n^{2(s-s_0) + 1}$ if $s \geq s_0$ and has an upper bound of order $n$ if $s \leq s_0$.
The log-determinant is of order $\sum_{i=1}^n \log( i^{-2s/d} ) \approx -(2s/d) n \log (n)$ by~\eqref{eq:worst_case_approx_error} and Proposition~\ref{prop:cond_var_bounds}.
For large $n$ we thus have
\begin{equation*}
    \ell(s; u(\mb{x})) \approx \frac{2s}{d} n \log (n) - n^{2\max\{(s-s_0), 0\} + 1},
\end{equation*}
which is maximized by $s \approx s_0$ since, roughly speaking, this ensures simultaneously that (a) the first term has a large constant coefficient and (b) the expression does not tend to negative infinity.
The proof casts this argument in rigorous language.

\begin{proof}[Proof of Theorem~\ref{thm:consistency}]
    Any maximizer of $\ell(s; u(\mb{x}))$ also maximizes the logarithm of the likelihood ratio
    \begin{align*}
        \mathcal{L}(s) &\coloneqq 2(\ell(s; u(\mb{x})) - \ell(s_0; u(\mb{x}))) \\
        &= \log\left(\frac{\det(K_{s_0}(\mb{x}))}{\det(K_{s}(\mb{x}))}\right) + u(\mb{x})^{\top}K_{s_0}(\mb{x})^{-1}u(\mb{x}) -  u(\mb{x})^{\top}K_{s}(\mb{x})^{-1}u(\mb{x}).
    \end{align*}
    Note that 
    \begin{align*}
        \left\lbrace \lvert \hat{s}_n - s_0 \rvert > \delta \right\rbrace \subset \left\lbrace \sup_{d/2 < s \leq s_0 - \delta} \mathcal{L}(s) \geq 0 \right\rbrace \cup \left\lbrace \sup_{ s_0 - \delta \leq s \leq S_{\max}} \mathcal{L}(s)  \geq 0 \right\rbrace
    \end{align*}
    for every $\delta > 0$ for which the suprema are over non-empty sets.
    Thus it suffices to show that
    \begin{align}
        \P\left( \sup_{d/2 < s \leq s_0 - \delta} \mathcal{L}(s) \geq 0 \right) + \P \left( \sup_{ s_0 - \delta \leq s \leq S_{\max}} \mathcal{L}(s)  \geq 0 \right) \to 0 \quad \text{ as } n \to \infty\label{eq:ell-difference-limit}
    \end{align}
    for every such $\delta > 0$.
    
    Consider first the case $s \geq s_0 + \delta$.
    Choose $0<\epsilon < \min\{2\delta/d, s_0 - d\}$.
    It follows from combining Propositions \ref{prop:lower_bound_norm_interpolant} and \ref{prop:escape_manifold} and quasi-uniformity [see~\eqref{eq:mn-lower-bound-auxiliary}] that almost surely
    \begin{align*}
        u(\mb{x})^{\top}K_{s_0 + \delta}(\mb{x})^{-1}u(\mb{x}) = \lVert m_{s_0 + \delta, n}\rVert_{s_0 + \delta}^2 \geq C n^{1 + 2\delta /d - \epsilon}. 
    \end{align*}
    Moreover, for any $s > s_0 + \delta$, we have by \eqref{eq:norm_monotonicity} and the minimum norm property that
    \begin{align*}
        C n^{1 + 2\delta /d - \epsilon} \leq \lVert m_{s_0 + \delta,n} \rVert_{s_0 + \delta}^2 \leq \lVert m_{s,n} \rVert_{s_0 + \delta}^2 \leq C_2\lVert m_{s,n} \rVert_s^2. \label{eq:rkhs-norm-lower-bound-for-ell}
    \end{align*}
    Note that $C$ in the above display depends on $\delta$ and $s_0$ but not on $s$ and $C_2$ also doesn't depend on $s$.
    Next, we have by equation \eqref{eq:norm_upper_bound_interpolant} that almost surely
    \begin{align*}
        u(\mb{x})^{\top} K_{s_0}(\mb{x})^{-1}u(\mb{x}) = \lVert m_{s_0, n} \rVert_{s_0} \leq C n^{1 + \epsilon^\prime}
    \end{align*}
    for some $\epsilon^\prime < 2\delta/d - \epsilon$, where $C$ only depends on $\theta_0, \epsilon^\prime$ and the sample path.
    Clearly, we can find such an $\epsilon^\prime$ for every $\delta$.
    This gives
    \begin{equation}
    \begin{aligned}
        \sup_{s \geq s_0 + \delta} u(\mb{x})^{\top}K_{s_0}(\mb{x})^{-1}u(\mb{x}) -  u(\mb{x})^{\top}K_{s}(\mb{x})^{-1}u(\mb{x})        &\lesssim n^{1 + \epsilon^\prime} - n^{1 + 2\delta - \epsilon}\\
        &\lesssim - n^{1 + 2\delta - \epsilon}, \label{eq:rkhs-norm-lower-bound-for-ell}
    \end{aligned}
    \end{equation}
    almost surely, where the hidden constants only depend on $\delta, \epsilon, \theta_0$ and the sample path.
    For the terms involving the log-determinants we have, as in \eqref{eq:worst_case_approx_error},
    \begin{align*}
        -\log(\det(K_s(\mb{x})) +  \log(\det(K_{s_0}(\mb{x})) = \sum_{i = 1}^n \log\left(\frac{\V_{s_0}(x_i \vert \mb{x}^{i-1})}{\V_{s}(x_i \vert \mb{x}^{i-1})}\right).
    \end{align*}
    By quasi-uniformity, the bounds in Proposition~\ref{prop:cond_var_bounds} apply to each $\V_{s}(x_i \vert \mb{x}^{i-1})$.
    Therefore for all $s > d/2$
    \begin{align*}
        \V_s(x_i \vert \mb{x}^{i-1}) \geq c i^{-(2S_{\max} - d) / d} \quad  \text{ and } \quad \V_{s_0}(x_i \vert \mb{x}^{i-1}) \leq C i^{-(2s_0 - d)/d}.
    \end{align*}
    Thus, for $s \geq s_0 + \delta$,
    \begin{equation}
    \begin{aligned}
        \sum_{i = 1}^n \log\left(\frac{\V_{s_0}(x_i \vert \mb{x}^{i-1})}{\V_{s}(x_i \vert \mb{x}^{i-1})}\right) &\leq n \log(C / c)  + \sum_{i = 1}^n \log\left(i^{(2S_{\max}  - 2s_0)/d}\right) \\
        &= n\log(C/c)  + \frac{2(S_{\max} - s_0)}{d} \sum_{i = 1}^n \log(i) \asymp n \log(n), \label{eq:log_dets}
    \end{aligned}
    \end{equation}
    where we have used Stirling's approximation $\sum_{i=1}^n \log(i) = \log(n!) = n \log (n) - O(n)$ and the hidden constants depend on $s_0$ and $\delta$ but not on $s$.
    It follows from the estimates~\eqref{eq:rkhs-norm-lower-bound-for-ell} and \eqref{eq:log_dets} that almost surely there exist positive constants $C_1$ and $C_2$ such that
    \begin{align*}
        \lim_{n \to \infty}\sup_{s_0 + \delta \leq s < S_{\max}} \mathcal{L}(s) \leq  \lim_{n \to \infty} C_1 n\log(n) - C_2 n^{1 + 2\delta/d - \epsilon} = -\infty .
    \end{align*}
    From this it follows that $\limsup_{n \to \infty} \hat{s}_n \leq s_0$ almost surely, which is stronger than what is required for \eqref{eq:ell-difference-limit}.
    
    Consider then the case $d/2 < s \leq s_0 - \delta$.
    By repeating the steps in \eqref{eq:log_dets}, we obtain
    \begin{align}
        \sum_{i = 1}^n \log\left(\frac{\V_{s_0}(x_i \vert \mb{x}^{i-1})}{\V_{s}(x_i \vert \mb{x}^{i-1})}\right) \lesssim (s_0 - \delta - s_0)n\log(n) = -\delta n \log(n)\label{eq:liminf_determinant_bound}.
    \end{align}
    Let $m_{\theta_0,n}$ denote the minimum norm interpolant based on the correct kernel $\sigma_0^2 K_{\theta_0}$.
    By norm equivalence and the minimum norm property, 
    $\lVert m_{s_0, n} \rVert_{s_0} \leq \lVert m_{\theta_0, n} \rVert_{s_0} \leq C \lVert m_{\theta_0, n} \rVert_{\theta_0}$.
    Thus
    \begin{equation}
    \begin{aligned}
        u(\mb{x})^{\top}K_{s_0}(\mb{x})^{-1}u(\mb{x}) &\leq \sigma_0^{-2}C u(\mb{x})^{\top}K_{\theta_0}(\mb{x})^{-1}u(\mb{x}) \\
        & = C \tilde{\mb{\xi}}_n^{\top} K_{\theta_0}(\mb{x})^{1/2}K_{\theta_0}(\mb{x})^{-1}K_{\theta_0}(\mb{x})^{1/2} \tilde{\mb{\xi}}_n \leq C \sum_{i = 1}^n \tilde{\xi}_{n,i}^2, \label{eq:liminf_quad_form_bound}
    \end{aligned}
    \end{equation}
    where $\{\tilde{\mb{\xi}}_n = (\xi_{n,1}, \dots, \tilde{\xi}_{n,n})\}$ forms a triangular array of pairwise uncorrelated random variables with mean 0 and variance 1.
    The estimates \eqref{eq:liminf_determinant_bound} and \eqref{eq:liminf_quad_form_bound} and Markov's inequality yield
    \begin{align*}
        \P \bigg(\sup_{d/2 < s \leq s_0 - \delta }\mathcal{L}(s) \geq 0\bigg) &\leq \P \bigg( - C_1 \delta n \log(n) + C_2 \sum_{i = 1}^n \tilde{\xi}_{n,i}^2 -u(\mb{x})^{\top}K_s(\mb{x})^{-1} u(\mb{x}) \geq 0  \bigg) \\
        &\leq \P \bigg(\sum_{i = 1}^n \tilde{\xi}_{n,i}^2 \geq  c n\log(n) \bigg) \leq \frac{1}{c\log(n)} \to 0.
    \end{align*}    
    This concludes the proof of~\eqref{eq:ell-difference-limit}.
\end{proof}

The assumption $s_0 > d$ was only used to prove $\P\left(\limsup_{n \to \infty} \hat{s}_n \geq s_0\right) = 1$.
Thus if we relax the assumption to $s_0 > d/2$, we still obtain $\P(\hat{s}_n \leq s_0) \to 0$.
However, if we impose additional assumptions on the $\{\xi_i\}$ we can strengthen the statement of Theorem \ref{thm:consistency}.

\subsection{The Gaussian Case}\label{sec:gaussian_case}

In this section, we will discuss the particular case where the $\{\xi_i\}$ are standard Gaussian random variables, for concreteness and because of the importance of Gaussian processes in spatial statistics and other related fields.
We first establish a path regularity result, similar to~\citet[Theorem 4.6]{lang_2015}, which gives the result for the sphere $\mathcal{M} = \mathbb{S}^{d-1}$.
A stronger result for the sphere is given in~\citet{lan_2018}.
For various other path regularity results, see~\cite{Scheuerer2010, meershaert_2013, steinwart_2019, Henderson2022}.

\begin{proposition}\label{prop:path_property}
    Let $\{\xi_i\}$ be an i.i.d. sequence of standard Gaussian random variables.
    Let $l + \alpha  < s_0 - d/2$ with $l \in \mathbb{N}$ and $0 < \alpha < 1$ and let $u$ be defined as in \eqref{eq:true_model}.
    Then there is a version of $u$ whose sample paths are almost surely elements of $C^{l, \alpha}(\mathcal{M})$.
\end{proposition}
\begin{proof}
    As in the proof of Theorem 4.6 in \citet{lang_2015}, it suffices to establish the claim for $d/2 < s_0 \leq d/2+1$ because for any $u$  and $1 < l \in \mathbb{N}$ such that $s_0 - l > d/2$ we have 
    \begin{align*}
        \frac{1}{\sigma_0}u = \frac{v(\theta_0)}{v(\tilde{\theta}_0)} (\tau_0 - \Delta)^{-l/2} \sum_{i = 1}^\infty v(\tilde{\theta}_0) (\tau_0 - \lambda_i)^{-s_0/2 + l/2}\xi_i e_i \eqqcolon \frac{v(\theta_0)}{v(\tilde{\theta}_0)} (\tau_0 - \Delta)^{-l/2} \tilde{u}.
    \end{align*}
    and for integer $l$ the operator $(\tau_0 - \Delta)^{-l/2}$ is an isomorphism from $C^{0, \alpha}$ to $C^{l, \alpha}$ for $0 < \alpha  < 1$ by Theorem XI.2.5 in \citet{taylor_1981}.
    Moreover, $\tilde{u}$ is a mean-zero Gaussian process with covariance function $K_{\tilde{\theta}_0}$ and $d/2 < \tilde{s}_0 \leq d/2 + 1$ where $\tilde{s}_0$ is the smoothness parameter in $\tilde{\theta}_0 = (\tilde{s}_0, \tau_0)$.
    The claim is now an immediate consequence of the results in \citet{Kerkyacharian_2018}. 
    More concretely, note that the results there apply to the Laplace--Beltrami operator on compact manifolds as mentioned in the Section 3 and the Examples on their page 283.
    Then apply first their Corollary 5.2, then their Theorem 4.6 and finally their Proposition 3.2.
\end{proof}

To simplify the notation in what follows we assume that $s_0 - d/2 = l + \alpha$ with $0 < \alpha \leq 1$.
It follows from Proposition~\ref{prop:path_property} that in the Gaussian case we have in fact 
\begin{align*}
    u \in H_L^{s_0 - d/2 - \epsilon}(\mathcal{M}) \cap C^{l, \alpha - \epsilon}(\mathcal{M}) = H_L^{l + \alpha - \epsilon}(\mathcal{M}) \cap C^{l, \alpha - \epsilon}(\mathcal{M})
\end{align*}
for all $\epsilon > 0$ small enough.
That is, $u$ possesses Sobolev and Hölder regularity of the same order.
Thus one might hope to have the following analog of Proposition \ref{prop:escape_manifold}:
\begin{align}
    \lVert u - m_{s, n} \rVert_{0} \leq C n^{-s_0/d  + 1/2 + \epsilon/d} \lVert u \rVert_{C^{l, \alpha - \epsilon}(\mathcal{M})} \label{eq:holder_approximation}
\end{align}
for all $s_0 > d/2$ when the sequence of points is quasi-uniform.
That this is indeed the case on closed manifolds has recently been shown in Corollary 1 of \citet{hangelbroek_2018}, which is based on results in \citet{hangelbroek_2010}.
However, their result applies only to even $s \in \mathbb{N}$.
While it appears likely that results like \eqref{eq:holder_approximation} should hold for general $s > d/2$, the result in \citet{hangelbroek_2018} suffices for our purposes.

\begin{corollary}\label{cor:consistency}
    Let the family of kernels $\{K_s\}_{s > d/2}$ and the sequence of points $\{x_i\}$ be as in Theorem \ref{thm:consistency}.
    Assume that the sequence $\{\xi_i\}$ used to define $u$ in \eqref{eq:true_model} is a sequence of i.i.d. standard Gaussian random variables and that $s_0 > d/2$.
    Let $\{\hat{s}_n\}$ be a sequence of maximizers of the Gaussian log-likelihood \eqref{eq:log_lik_thm} in $(d/2, S_{\max})$.
    Then
    \begin{align*}
        \hat{s}_n \to s_0 \quad \text{ almost surely}.
    \end{align*}
\end{corollary}
\begin{proof}
    The conclusion is stronger than that of Theorem~\ref{thm:consistency} in two ways: (i) $s_0 > d/2$ instead of $s_0 > d$ and (ii) the convergence is almost sure rather than in probability.
    The assumption $s_0 > d$ was used in the proof of Theorem~\ref{thm:consistency} only to show that almost surely
    \begin{align*}
        \lVert m_{s, n} \rVert_{s}^2 \geq n^{1 + 2(s - s_0)/d  - \epsilon}
    \end{align*}
    via Proposition \ref{prop:lower_bound_norm_interpolant}, which is only needed to show that $\P(\limsup_{n \to \infty}\hat{s}_n \leq s_0 ) = 1$.
    Thus it suffices to show that in the Gaussian case the assumption for Proposition \ref{prop:lower_bound_norm_interpolant} is fulfilled for $s_0 > d/2$.
    For this purpose, let $b$ be a natural number such that $2b > \max \{ S_{\max}, s_0 \}$ and let $\tilde{m}_{2b, n}$ denote the minimum norm interpolant from $\mathcal{H}^{2b}$.
    Note that because $u(x_i) = m_{s, n}(x_i)$ for all $1 \leq i \leq n$, it follows that $\tilde{m}_{2b, n}$ is the minimum norm interpolant for both $u$ and $m_{s, n}$.
    Then Proposition \ref{prop:escape_manifold} and Corollary 1 in \citet{hangelbroek_2018} give
    \begin{align*}
        \lVert u - m_{s, n} \rVert_0 &\leq \left\lVert u - \tilde{m}_{2b, n} \right\rVert_0 + \left\lVert \tilde{m}_{2b, n} - m_{s, n} \right\rVert_0 \\
        & \leq C_1 n^{-s_0 / d + 1/2 + \epsilon/d} \lVert u \rVert_{C^{k, \alpha - \epsilon}(\mathcal{M})} + C_2 n^{-s/d} \lVert m_{s, n} \rVert_{s}
    \end{align*}
    for constants $C_1, C_2 > 0$ that do not depend on $n$.
    Thus the assumptions of Proposition \ref{prop:lower_bound_norm_interpolant} are fulfilled.
    Repeating the proof of Theorem \ref{thm:consistency} yields $\P(\limsup_{n \to \infty}\hat{s}_n \leq s_0 ) = 1$.

    It remains to prove $\P(\liminf_{n \to \infty}\hat{s}_n \geq s_0 ) = 1$.
    For this, note that the estimate~\eqref{eq:norm_upper_bound_interpolant_gaussian} in Proposition~\ref{prop:upper_bound_minimum_norm_interpolant} shows that almost surely there exists a positive constant $C$ such that 
    \begin{align*}
        u(\mb{x})^{\top} K_{s_0}(\mb{x})^{-1}u(\mb{x}) \leq C n.
    \end{align*}
    Repeating the same steps as in the proof of Theorem \ref{thm:consistency} yields that, almost surely,
    \begin{align*}
        \sup_{d/2 < s \leq s_0 - \delta} \mathcal{L}(s) &\leq -C_1 \delta n\log(n) + C_2 n - u(\mb{x})^{\top} K_{s}(\mb{x})^{-1}u(\mb{x}) \lesssim -n \log(n) \to - \infty,
    \end{align*}
    which implies $\P(\liminf_{n \to \infty}\hat{s}_n \geq s_0 ) = 1$ and thus completes the proof.
\end{proof}

\subsection{Estimation Using Other Kernels} \label{sec:other_kernels}

In this section we discuss some kernels to which Theorem~\ref{thm:consistency} applies.
Of course, the theorem applies to the correctly specified kernels $\sigma^2 K_{\theta}$ from Section \ref{sec:background}, as long as the corresponding norms fulfill the norm monotonicity property.
Although these kernels have many nice theoretical properties, their evaluation requires the calculation of an infinite sum.
Furthermore, the eigenvalue decomposition of the Laplace--Beltrami operator is only known in certain special cases, e.g. on the sphere or the torus.
Therefore, even obtaining an approximate evaluation of $K_{\theta}$ can be hopeless.
Another aspect of practical importance is that the matrices $K_{\theta}(\mb{x})$ are dense, so that the evaluation of the Gaussian likelihood for large $n$ is computationally intensive.
Therefore, it would be desirable to find other classes $\{K_s\}_{s > d/2}$ of kernels that have good practical properties and allow a consistent estimation of $s_0$.
The crucial property that these kernels have to satisfy is that their RKHS are equal to $H^s_L$ up to equivalent norms.

One class of kernels we consider here are kernels from the ambient space $\mathbb{R}^k$.
Such kernels were applied to function approximation on closed manifolds in \citet{fuselier_2012} and some of the results in this section are essentially from their work.
For manifolds $\mathcal{M} \subset \mathbb{R}^k$ it is quite easy to arrive at positive definite kernels by using kernels from the ambient space $\mathbb{R}^k$.
Specifically, if $\tilde{K}$ is a positive definite kernel on $\mathbb{R}^k \times \mathbb{R}^k$, then its restriction is 
\begin{align*}
    K(\cdot, \cdot) = \tilde{K}(\cdot, \cdot)\vert_{\mathcal{M} \times \mathcal{M}}
\end{align*}
is of course positive definite as well.
The RKHS of a particular class of kernels $\tilde{K}$ were identified in \citet{fuselier_2012}.

\begin{theorem}[Theorem 2 in \citealp{fuselier_2012}]\label{thm:fusilier_theorem}
    Let $\tilde{K}_s(x, y) = \Phi_s(\lVert x- y \rVert)$ be a positive definite kernel on $\mathbb{R}^k$ such that the Fourier transform of $\Phi_s$ satisfies 
    \begin{align}
        \hat{\Phi}_s(\xi) \asymp (1 + \lVert \xi \rVert^2)^{-(s +(k-d)/2)}. \label{eq:kernel_fourier_decay}
    \end{align}
    Then the RKHS $\mathcal{H}^s$ with norm $\lVert \cdot \rVert_s$ of the restriction $K_s(\cdot, \cdot) = \tilde{K}_s(\cdot, \cdot) \vert_{\mathcal{M} \times \mathcal{M}}$ is $H^{s}_L$ up to equivalent norms.
\end{theorem}

It is also known that the kernel $\tilde{K}_s$ is the kernel for an RKHS equal to $H^{s + (k-d)/2}(\mathbb{R}^{k})$ up to equivalent norms \citep[Chapter 10]{wendland_2004}.
In Theorem~\ref{thm:consistency} it is also assumed that the norms fulfil the monotonicity property~\eqref{eq:norm_monotonicity}.
Fortunately, this is inherited from the norm monotonicity property of the original kernel.
The following result is a direct consequence of \citet[Lemma~4]{fuselier_2012}, which in turn is a reformulation of the results in \citet[Section~9]{schaback_1999}, which can be found also in \citet[Theorem~10.46 and Corollary~10.47]{wendland_2004}.
See the Appendix, Proposition \ref{prop:generalized_wendland_monotone} for a proof. 

\begin{proposition}\label{prop:inherited_norm_monotonicity}
  Let $\tilde{K}_s(\cdot, \cdot)$ be as in Theorem \ref{thm:fusilier_theorem}.
  Let $\mathcal{H}^s(\mathbb{R}^k)$ be the RKHS of $\tilde{K}_s$ and $\mathcal{H}^s$ the RKHS of $K_s(\cdot, \cdot) = \tilde{K}_s(\cdot, \cdot)\vert_{\mathcal{M} \times \mathcal{M}}$.
    Let $s \leq s^\prime$.
    If $\lVert f \rVert_{\mathcal{H}^s(\mathbb{R}^k)} \leq C \lVert f \rVert_{\mathcal{H}^{s^\prime}(\mathbb{R}^k)}$ for all $f \in \mathcal{H}^{s^\prime}(\mathbb{R}^k)$ and some $C$ independent of $s$, then $\lVert f \rVert_{s} \leq C \lVert f \rVert_{s^\prime}$ for all $f \in \mathcal{H}^s$.  
\end{proposition}

The prime examples of such kernels are the isotropic Mat\'ern kernels on $\mathbb{R}^k$ given in \eqref{eq:matern_cov}.
The fact that their RKHS have monotone norms is shown for example in~\citet[Lemma~3.9]{karvonen2023asymptotic}.

Another class of kernels satisfying \eqref{eq:kernel_fourier_decay} are the so-called generalized Wendland functions
\begin{align*}
  \Phi_{\kappa}(\lVert x- y \rVert) = \Phi_{\kappa}(r)
  =
  \begin{dcases} 
        \frac{\sigma^2 2^{1-\kappa}}{\Gamma(\kappa)} \int_{r/\beta}^1 u\left(u^2 - \frac{r^2}{\beta^2}\right)^{\kappa - 1}(1 - u)^\mu \, \mathrm{d}u & \text{ if } \quad 0 < \frac{r}{\beta} < 1, \\
        0 & \text{ if } \quad \frac{r}{\beta} \geq 1;
  \end{dcases}
\end{align*}
see \cite{gneiting_2002, bevilacqua_2019, hubbert_2023}.
Here $\beta$ and $\sigma^2$ are positive and $\mu \geq (d+1)/2 + \kappa$, while $\Gamma(\cdot)$ is the gamma function.
For $s = \kappa + (d+1)/2$ the RKHS of $\Phi_{\kappa}$ is equal to $H^s(\mathcal{M})$~\citep[Theorem 1]{bevilacqua_2019}.
Generalized Wendland functions, in a slightly different form, are studied as covariance functions for Gaussian processes in~\citet{bevilacqua_2019}.
The difference is irrelevant for our purposes.
The advantage of the generalized Wendland functions is that they are compactly supported and thus yield sparse kernel matrices.
The monotonicity of the corresponding norms is proved in Proposition~\ref{prop:generalized_wendland_monotone}.

If $\mathcal{M} = \mathbb{S}^{d-1}$ is the sphere, the restriction of the generalized Wendland functions to $\mathcal{M}$ was investigated by \citet{hubbert_2023}.
Further covariance functions on the sphere with explicit smoothness parameters are presented in \citet{gneiting_2013} and \citet{guinnes_2016}.

\subsection{Estimation of the Magnitude Parameter} \label{sec:magnitude-estimation}
The smoothness parameter $s$ is fixed to a largely arbitrary value in many applications.
Here we consider a setting in which the model smoothness $s \geq s_0$ is fixed but the \emph{magnitude parameter} $\sigma$ is estimated via maximum likelihood.
Define therefore $K_{s,\sigma} = \sigma^2 K_s$.
Straightforward differentiation of the Gaussian log-likelihood~\eqref{eq:gaussian_ll} with respect to $\sigma$ then yields the maximum likelihood estimator
\begin{align*}
  \hat{\sigma}_n^2 = \frac{u(\mb{x})^{\top} K_s(\mb{x})^{-1} u(\mb{x})}{n} = \frac{\lVert m_{s, n} \rVert_{s}^2 }{n}.
\end{align*}
It is now an immediate consequence of Propositions~\ref{prop:lower_bound_norm_interpolant} to~\ref{prop:upper_bound_minimum_norm_interpolant} that $\hat{\sigma}_n^2$ blows up with a polynomial rate that depends on the extent, $s - s_0$, to which the model oversmooths the truth.

\begin{corollary} \label{corollary:magnitude-estimation}
  Let $K_s$ be a kernel whose RKHS $\mathcal{H}^s$ is equal to $H^s_L$ up to equivalent norms and let $u$ be defined as in \eqref{eq:true_model} with $s \geq s_0 > d$.
  Assume that the sequence of points $\{ x_i \}$ is quasi-uniform.
  Then almost surely 
  \begin{align*}
    c \, n^{2(s-s_0)/d - \epsilon} \leq \hat{\sigma}_n^2 \leq C n^{2(s-s_0)/d + \epsilon}
  \end{align*}
  for any $\epsilon > 0$, where the constants $c, C > 0$ depend only on $s$, $\theta_0$, $\mathcal{M}$, $\epsilon$ and the sample path.
\end{corollary}

If $u$ is a Gaussian process, an improved version of this corollary can be obtained by proceeding as in Section~\ref{sec:gaussian_case}.
According to Proposition~\ref{prop:escape_manifold}, the minimum-norm interpolant $m_{s,n}$ with misspecified smoothness $s \geq s_0$ tends to $u$ at essentially the same rate $n^{-s_0/d + 1/2}$ as $m_{s_0, n}$.
Note that the estimation of the magnitude does not affect the minimum norm of the interpolant.
Corollary~\ref{corollary:magnitude-estimation} is useful because it shows that a computationally simple magnitude estimate correctly compensates for the misspecification of smoothness and ensures that the conditional variance is a reliable indicator of predictive uncertainty.
Proposition~\ref{prop:cond_var_bounds} and Corollary~\ref{corollary:magnitude-estimation} imply that for quasi-uniform points the correct variance $\mathbb{V}_{s_0}$ and the variance $\mathbb{V}_{s,n} = \hat{\sigma}_n^2 \mathbb{V}_s$ for the kernel $K_{s,\sigma}$ with the plug-in maximum likelihood estimator $\sigma = \hat{\sigma}_n$ fulfill
\begin{align*}
 \sup_{ x \in \mathcal{M}} \mathbb{V}_{s_0}(x|\mb{x}) \asymp n^{-2s_0/d + 1} \quad \text{ and } \quad n^{-2s_0/d + 1 - \epsilon} \lesssim \sup_{x \in \mathcal{M}} \mathbb{V}_{s,n}(x|\mb{x}) \lesssim n^{-2s_0/d + 1 + \epsilon}.
\end{align*}
This means that $\mathbb{V}_{s,n}$ decays at the essentially ``correct'' rate.

When smoothness is well specified (i.e. $s = s_0$), a number of consistency and asymptotic normality results for magnitude estimation in Matérn and related models on Euclidean domains have been obtained in the last 30 years~\cite[e.g.,][]{Ying1991, zhang_2004, Kaufman2013}.
\citet{li_2023} considered essentially the same setting as in this work and established consistency of the microergodic parameter in the Gaussian case with correctly specified kernel.
A corresponding asymptotic normality result of the form (assuming $d \leq 3$)
\begin{align}
    \sqrt{n} (\hat{\sigma}^2_n v(\hat{\theta}_n) - \sigma_0^2 v(\theta_0)) \to \mathcal{N}(0, 2 \sigma_0^4 v(\theta_0)^2) \label{eq:microergodic_clt}
\end{align}
can be derived with known arguments, which are used for example in \cite{du_2009, wang_2011, Kaufman2013, bolin_graph_2023}.
A formal statement and proof can be found in the Appendix.
The proof is essentially from \citet[Proposition 3]{bolin_graph_2023}, where a different scenario is considered.

The results for misspecified smoothness are recent.
Versions of Corollary~\ref{corollary:magnitude-estimation} for the expected value of $\hat{\sigma}_n^2$ have been proved in~\citet{Karvonen2021} and \citet{Naslidnyk2023}.
\citet{Karvonen2020} proved the corollary for a class of self-similar functions on Euclidean domains, while \citet{XuStein2017} obtained rates for the magnitude estimator when the model has Gaussian covariance $K(x, y) = \exp(-(x-y)^2/(2 \tau^2))$ for $\tau > 0$ and the observations are generated by a low-order monomial on the interval $[0, 1]$.
When the observations are contaminated by noise, posterior contraction rates and the coverage of credible sets for a Brownian motion model with an estimated magnitude parameter were investigated in~\cite{SniekersVaart2015a, SniekersVaart2015b, SniekersVaart2020} and \citet{TravisRay2023}.

\section{Equivalence of Measures}\label{sec:equivalent_measures}
In the previous section, we discussed briefly the estimation of the parameters $\tau$ and $\sigma^2$. 
Note that the central limit theorem in \eqref{eq:microergodic_clt} is stated only for the product $\sigma^2 v(\theta)$, the so-called microergodic parameter.
This is due to the well-known fact that for Gaussian processes, the parameters $\sigma^2$ and $\tau$ cannot be consistently estimated if the dimension $d \leq 3$.
For bounded Euclidean domains, see for example \citet{zhang_2004} and \citet{bolin_2023}; for closed manifolds, see~\citet{li_2023}.
This is due to equivalence of the Gaussian measures describing the distribution of the processes. 
More concretely, in the context considered here the cited results state that if $d \leq 3$ and if $\theta_1$ and $\theta_2$ are such that $s_1 = s_2$ and 
\begin{align*}
    \sigma^2_1 v(\theta_1) = \sigma^2_2 v(\theta_2),
\end{align*}
then the Gaussian measures $\mathcal{N}(0, \sigma_1^2 v(\theta_1)(\tau_1 - \Delta)^{-s_1})$ and $\mathcal{N}(0, \sigma_2^2 v(\theta_2)(\tau_2 - \Delta)^{-s_2})$ are equivalent.
Otherwise they are orthogonal.
The results in \citet{bolin_2023} and \citet{li_2023} apply only to Gaussian measures.

However, as far as we know, there are no results for the non-Gaussian case.
In this section, we will show that this continues to hold for non-Gaussian processes.
To address this scenario, we will use a theorem of \citet{kakutani_1948}, for which the assumption that $\{\xi_i\}$ are independent is important.
Here we briefly establish the necessary measure-theoretic framework with more details in~\ref{sec:measure_kakutani}.

A measure $m_1$ is absolutely continuous with respect to a measure $m_2$, denoted by $m_1 \ll m_2$, if
$m_2(B) = 0$ implies that $m_1(B) = 0$.
Two measures $m_1$ and $m_2$ are equivalent if $m_1 \ll m_2$ and $m_2 \ll m_1$.
Two probability measures are orthogonal if there is a set $B$ for which $m_1(B) = 1$ and $m_2(B) = 0$.
We start with the complete probability space $(\Omega, \mathcal{F}, \P)$ on which the sequence $\{\xi_i\}$ is defined.
To simplify the notation in this section, let
\begin{align*}
 \lambda_i(\theta) = \sqrt{\smash[b]{v(\theta)}}(\tau + \lambda_i)^{-s/2}.
\end{align*}
These are the eigenvalues of the operator $\sqrt{\smash[b]{v(\theta)}}(\tau -\Delta)^{-s/2}$.

Let $\{(\mathbb{R}, \mathcal{B}(\mathbb{R}), \P_{\sigma\lambda_i(\theta)})\}$ be a sequence of probability spaces with $\mathcal{B}(\mathbb{R})$ the Borel $\sigma$-algebra generated by the open sets, and $\P_{\sigma\lambda_i(\theta)}$ the distribution of
\begin{align*}
    \sigma\sqrt{\smash[b]{v(\theta)}}(\tau + \lambda_i)^{-s/2} \xi_i = \sigma \lambda_i(\theta) \xi_i.
\end{align*}
Consider the infinite direct product measure space $(\mathbb{R}^{\infty}, \mathcal{B}^{\infty}, \P_{\sigma, \theta})$, where $\mathbb{R}^\infty$ is the infinite Cartesian product, $\mathcal{B}^\infty$ is the smallest $\sigma$-algebra making all coordinate projections measurable and $\P_{\sigma, \theta}$ is the infinite product measure \smash{$\P_{\sigma, \theta} \coloneqq \bigotimes \P_{\sigma \lambda_i(\theta)}$}.

Recall that \smash{$\hat{H}^0_L$} is the Hilbert space of square-summable sequences. 
One can show \smash{$\hat{H}^0_L$} is $\mathcal{B}^{\infty}$-measurable and that the $\sigma$-algebra \smash{$\mathcal{B}(\hat{H}^0_L)$} generated by the open sets in \smash{$\hat{H}^0_L$} in fact equals \smash{$\{A \cap \hat{H}^0_L ~ \vert ~ A \in \mathcal{B}^{\infty}\}$}, see Proposition \ref{prop:sigma_algebra_coincide}.
Note that $\P_{\sigma, \theta}(\hat{H}^0_L) = 1$ for $s > d/2$ by Proposition \ref{prop:smoothness_class}.
Thus we may understand each measure $\P_{\sigma, \theta}$ as a measure on the space \smash{$(\hat{H}^0_L, \mathcal{B}(\hat{H}^0_L))$}.
They induce measures on $(L_2(\mathcal{M}), \mathcal{B}(L_2(\mathcal{M})))$ via $P_{\sigma,\theta} \circ \iota^{-1}$ as $\iota^{-1}$ is an isomorphism.
Of course, the measures $\P_{\sigma, \theta} \circ \iota^{-1}$ are the distributions of the random elements $u$ defined in \eqref{eq:karhunen_loeve}, see Proposition \ref{prop:correct_measures}.

Given two probability measures $m_1$ and $m_2$, let $\lambda$ be a measure such that $m_1 \ll \lambda$ and $m_2 \ll \lambda$. 
Such a measure always exists (e.g., $\lambda = m_1 + m_2$).
The Hellinger distance between $m_1$ and $m_2$ is
\begin{align*}
    \rho(m_1, m_2) = \int \sqrt{\frac{\mathrm{d} m_1}{ \mathrm{d} \lambda}} \sqrt{\frac{\mathrm{d} m_2}{\mathrm{d} \lambda}} \mathrm{d}\lambda,
\end{align*}
where $\mathrm{d} m_1 / \mathrm{d} \lambda$ and $\mathrm{d} m_2 / \mathrm{d}\lambda$ are Radon--Nikodym derivatives.
It is known that $\rho$ does not depend on the choice of $\lambda$ \cite[Section~4]{kakutani_1948}.
Moreover, $\rho$ has the following properties: (i) $0 < \rho(m_1, m_2) \leq 1$, (ii) $\rho(m_1,m_2) = 1 \iff m_1 = m_2$, and (iii) $\rho(m_1, m_2) = \rho(m_2, m_1)$ \citep[p.\@~215]{kakutani_1948}.
Kakutani's theorem \citep[p.\@~218]{kakutani_1948} links equivalence or orthogonality of the infinite product measures $\P_{\sigma,\theta}$ and $\P_{\sigma^\prime,\theta^\prime}$ to the behaviour of $\rho(\P_{\sigma\lambda_i(\theta)}, \P_{\sigma^\prime \lambda_i(\theta^\prime)})$ as $i \to \infty$.
These Hellinger distances have a particularly convenient form in the scenario under consideration.
Concretely, let $\mathbb{P}_{1}$ denote the distribution of the $\xi_i$ and assume it admits the density $f$ with respect to the Lebesgue measure.
Then a simple change of variables yields
\begin{align*}
    \rho\left(\P_{\sigma\lambda_i(\theta)}, \P_{\sigma^\prime \lambda_i(\theta^\prime)}\right) &= \int_{\mathbb{R}}  \frac{1}{\sqrt{\smash[b]{\sigma\lambda_i(\theta)\sigma^\prime\lambda_i(\theta^\prime)}}} \sqrt{f \left(\frac{x}{\sigma\lambda_i(\theta)}\right)f \left(\frac{x}{\sigma^\prime\lambda_i(\theta^\prime)}\right)}\, \mathrm{d}x \\
    &=  \int_{\mathbb{R}} \frac{\sqrt{\smash[b]{\sigma^\prime \lambda_i(\theta^\prime)}}}{\sqrt{\smash[b]{\sigma\lambda_i(\theta)}}}\sqrt{f(x) f\left(x\frac{\sigma^\prime\lambda_i(\theta^\prime)}{\sigma\lambda_i(\theta)}\right)} \, \mathrm{d}x \\
    &= \rho \big( \P_{\sigma\lambda_i(\theta) / (\sigma^\prime \lambda_i(\theta^\prime))}, \P_{1} \big).
\end{align*}
Thus, the Hellinger distance only depends on the ratio $\sigma\lambda_i(\theta) / (\sigma^\prime \lambda_i(\theta^\prime))$.
This motivates the function $\varphi(a) = \rho (\P_{a}, \P_{1})$, which has a global maximum $\varphi(1) = 1$.

Sufficient regularity of the function $\varphi$ implies $\varphi^\prime(1) = 0$ and $\varphi^{\dprime}(1) < 0$ and so a Taylor expansion shows that 
\begin{align*}
    0 \leq -\log(\varphi(a)) &= -\frac{\varphi^\prime(1)}{\varphi(1)}(a-1) - \bigg[\frac{\varphi^{\dprime}(a^*)}{\varphi(a^*)} - \left(\frac{\varphi^\prime(a^*)}{\varphi(a^*)} \right)^2 \bigg] (a-1)^2 \\
    &= \bigg[-\frac{\varphi^{\dprime}(a^*)}{\varphi(a^*)} + \bigg(\frac{\varphi^\prime(a^*)}{\varphi(a^*)} \bigg)^2 \bigg] (a-1)^2
\end{align*}
for some $a^*$ such that $\lvert a^* -1 \rvert \leq \lvert a- 1\rvert$.
Hence $-\log(\varphi(a))$ behaves like $(a-1)^2$ in a neighborhood of 1 if $\varphi$ is sufficiently regular.
Under this condition, we can provide a criterion for equivalence and orthogonality for two infinite product measures $\P_{\sigma_1^2, \theta_1}$ and $\P_{\sigma^2_2, \theta_2}$.

\begin{theorem}\label{thm:equivalent_measures}
    Let $\{\xi_i\}$ be a sequence of i.i.d. random variables such that $\E(\xi_1) = 0$ and $E(\xi_1^2) < \infty$.
    Assume that, for all $a, b > 0$ and each $i$, the distributions of $a\xi_i$ and $b\xi_i$ are equivalent measures and the distribution of $\xi_i$ admits a density with respect to the Lebesgue measure.
    Moreover, assume that the function $\varphi(a) = \rho (\P_{a}, \P_{1})$ satisfies
    \begin{align}
        c(a-1)^2 \leq -\log(\varphi(a)) \leq C (a-1)^2 \label{eq:taylor2}
    \end{align}
    in a neighborhood of 1.
    If the parameters $s_1, \tau_1, \sigma_1^2$ and $s_2, \tau_2, \sigma_2^2$ are such that
    \begin{align}
        s_1 = s_2 \quad \text{ and } \quad \sigma_1^2 v(\theta_1) = \sigma_2^2 v(\theta_2), \label{eq:equiv_measures}
    \end{align}
    then the measures $\P_{\sigma_1^2, \theta_1}$ and $\P_{\sigma_2^2, \theta_2}$ are equivalent if $d \leq 3$ and orthogonal if $d > 3$.
    If \eqref{eq:equiv_measures} fails and if $\varphi$ is uniformly bounded away from 1 outside of $(1-\epsilon, 1 + \epsilon)$ for every $\epsilon > 0$, the measures $\P_{\sigma_1^2, \theta_1}$ and $\P_{\sigma_2^2, \theta_2}$ are orthogonal.
\end{theorem}
\begin{proof}
    We start with a simple observation that, for $a$ in a neighborhood of 1,
    \begin{align}
        c(a-1)^2 \leq \log(a)^2 \leq C(a-1)^2.\label{eq:taylor1}
    \end{align}
    Kakutani's theorem states that $\mathbb{P}_{\sigma_1,\theta_1}$ and $\mathbb{P}_{\sigma_2,\theta_2}$ are equivalent if
    \begin{align} \label{eq:kakutani-eq}
        \sum_{i = 1}^\infty -\log\left[\rho(\P_{\sigma_1\lambda_i(\theta_1)}, \P_{\sigma_2\lambda_i(\theta_2)})\right] < \infty
    \end{align}
    and orthogonal otherwise.
    Let 
    \begin{align*}
        a_i = \frac{\sigma_1\sqrt{v(\theta_1)}(\tau_1 + \lambda_i)^{-s_1/2}}{\sigma_2\sqrt{v(\theta_2)}(\tau_2 + \lambda_i)^{-s_2/2}}.
    \end{align*}
    Under the assumptions in~\eqref{eq:equiv_measures},
    \begin{align}
        a_i = \left(\frac{\tau_1 + \lambda_i}{\tau_2 + \lambda_i}\right)^{-s_1/2} \: \overset{i \to \infty}{\longrightarrow} \: 1. \label{eq:log_quadratic}
    \end{align}
    Thus, combining \eqref{eq:taylor2}, \eqref{eq:taylor1} and \eqref{eq:log_quadratic}, we obtain
    \begin{align*}
        -\log\left[ \rho(\P_{\sigma_1\lambda_i(\theta_1)}, \P_{\sigma_2\lambda_i(\theta_2)})\right] = -\log(\varphi(a_i)) \asymp \log \left(\frac{\tau_2 + \lambda_i}{\tau_1 + \lambda_i}\right)^2 \asymp \left(\frac{\tau_2 - \tau_1}{\tau_1 + \lambda_i}\right)^2 \asymp i^{-4/d},
    \end{align*}
    where the last $\asymp$ follows from Weyl's law~\eqref{eq:weyls-law}.
    Thus~\eqref{eq:kakutani-eq} holds if and only if $\sum_{i = 1}^\infty i^{-4/d} < \infty$.
    This sum is finite if and only if $d < 4$.
    To show that $\varphi$ being uniformly bounded away from $1$ outside of every neighborhood of $1$ implies orthogonality of the resulting measures, note that 
    \begin{align*}
      a_i \to 0 \:\: \text{ if $s_1 > s_2$}, \quad a_i \to \infty \:\: \text{ if $s_1 < s_2$} \quad \text{ and } \quad s_i \to \sigma_1^2 v(\theta_1) / \sigma_2^2 v(\theta_2) \:\: \text{ else.}
    \end{align*}
    Thus, if \eqref{eq:equiv_measures} does not hold, there exist $\delta$ and $N$ such that $\sum_{i = N}^\infty-\log(\varphi(a_i)) \geq \sum_{i = N}^\infty \delta$.
\end{proof}

We have assumed that the distributions of $a\xi_i$ and $b\xi_i$ are equivalent measures for each positive $a$ and $b$.
If this is not the case, one can easily see that the resulting infinite product measures cannot be equivalent, see \citet[p.\@~218]{kakutani_1948}.
Note that this restricts the distributions that can lead to equivalent measures.
For example, if the distribution $\xi_i$ has a Lebesgue density that vanishes on a set with positive Lebesgue measure, then the distributions $a\xi_i$ and $b\xi_i$ need not be equivalent measures for all $a$ and $b$.
However, this does not imply that the resulting infinite product measures are necessarily orthogonal.
This means that a consistent estimate may still not be possible.
For example, let $\P_{\sigma_1, \theta_1}$ and $\P_{\sigma_2, \theta_2}$ be non-orthogonal measures, even though $\sigma_1 \neq \sigma_2$.
Argue as in \citet{zhang_2004} and assume that $\{\hat{\sigma}_n\}$ is a sequence of estimators that converge to $\sigma_1$ in probability under $\P_{\sigma_1, \theta_1}$.
Then there exists a subsequence $\{\hat{\sigma}_{k_n}\}$ such that 
\begin{align*}
    \P_{\sigma_1, \theta_1} \Big(\lim_{k_n \to \infty} \hat{\sigma}_{k_n} = \sigma_1 \Big) = 1.
\end{align*}
Because $\P_{\sigma_2, \theta_2}$ is not orthogonal to $\P_{\sigma_1, \theta_1}$, we have 
\begin{align*}
    \P_{\sigma_2, \theta_2} \Big(\lim_{k_n \to \infty} \hat{\sigma}_{k_n} = \sigma_1 \Big) > 0
\end{align*}
and so under $\P_{\sigma_2, \theta_2}$ consistent estimation of $\sigma_2$ may still not occur.

To show that two measures $m_1$ and $m_2$ are not orthogonal, it suffices to show that either $m_1 \ll m_2$ or $m_1 \ll m_2$.
A simple deduction from the proofs in \citet{kakutani_1948} gives a condition for $\P_{\sigma_1, \theta_1} \ll \P_{\sigma_2, \theta_2}$.
This is essentially part of Kakutani's theorem.

\begin{corollary}\label{cor:abs_cont}
  Consider the setup of Theorem \ref{thm:equivalent_measures} but assume only that the distribution of $a \xi_i$ is absolutely continuous with respect to $b \xi_i$ if $a \leq b$.
    Assume \eqref{eq:equiv_measures} holds and $d \leq 3$.
    Moreover, let $\theta_1$ and $\theta_2$ be such that $1 \leq \tau_1 \leq \tau_2$. 
    Then $\P_{\sigma_1, \theta_1} \ll \P_{\sigma_2, \theta_2}$.
\end{corollary}
\begin{proof}
    This is a consequence of Lemma 5 in \citet{kakutani_1948} and the proof of Theorem \ref{thm:equivalent_measures}.
\end{proof}

\begin{examples}
We give three examples for the distributions of $\xi_i$:
\begin{enumerate}
\item In the standard Gaussian case we recover the well-known results because a simple computation gives \smash{$\rho(\P_{a}, \P_1) = \sqrt{\smash[b]{2a/(1+a^2)}}$} and it is straightforward to see that \eqref{eq:taylor2} holds.
  \item More generally, it suffices to show that $\varphi$ is twice continuously differentiable to obtain \eqref{eq:taylor2}, which can be shown to be the case if $\xi_i$ has $t$-distribution with degrees of freedom larger than 2 in order to guarantee finite second moments.
    \item An example that pertains to Corollary \ref{cor:abs_cont} is $\xi_i + 1 \sim \mathrm{Exp}(1)$, where $\mathrm{Exp}(1)$ is the exponential distribution with scale 1.
    A simple computation shows that $\varphi(a) = 2\sqrt{a}/(1+a)$ and so \eqref{eq:taylor2} can be shown to hold.
\end{enumerate}
\end{examples}
We note here that while the independence assumption is crucial for the use of Kakutani's theorem, the assumption that $\{\xi_i\}$ are identically distributed may be relaxed.
The proofs remain essentially unchanged if additional technical assumptions are imposed.

\section{Results for the Dirichlet Laplacian}\label{sec:dirichlet_laplacian}

Some of our results can be directly applied to Laplacian operators on bounded Euclidean domains augmented with boundary conditions.
They are frequently studied in connection with Markov random field approximations of Mat\'ern type fields \citep{Lindgren_2011}.
As a concrete example, we take the Dirichlet Laplacian operator and present some of the results from the Sections~\ref{sec:background} to~\ref{sec:equivalent_measures}.
Most of the proofs are analogous, so most of the time we only provide sketches or point out where previously presented results need to be adapted.

In the following we consider an open and connected $M \subset \mathbb{R}^d$ with Lipschitz boundary $\partial M$ \citep[p.\@~189]{stein_1970}, which satisfies a \textit{interior cone condition} \citep[Definition 3.6]{wendland_2004}.
Let $\bar{M} = M \cup \partial M$ denote the closure of $M$.
For such a set $M$ we define the Sobolev spaces $H^s(M)$ as
\begin{align*}
    H^s(M) \coloneqq \big\{ f \in L_2(M) ~ \bigl\vert ~ \text{ there is } g \in H^s(\mathbb{R}^d) \text{ such that } g\vert_{M} = f \big\},
\end{align*}
which is equipped with the norm $\lVert f \rVert_{H^s(M)} = \inf\{ \lVert g \rVert_{H^s(M)} ~ \vert ~ g \in H^s(\mathbb{R}^d) \text{ s.t. } g\vert_M = f \}$.
Note that $H^{s^\prime}(M) \subset H^s(M) \subset L_2(M)$ if $s^\prime \geq s \geq 0$.
The Hilbert space $L_2(M)$ and the Hölder space $C^{k, \alpha}(\bar{M})$ are defined analogously to the spaces in Section~\ref{sec:background}.

For $M$ we consider the usual Laplacian operator
\begin{align*}
 -\Delta = -\sum_{i = 1}^d \frac{\partial^2}{\partial^2 x_i}.
\end{align*}
However, to obtain an eigendecomposition to define \eqref{eq:true_model} or to ensure that the SPDE in~\eqref{eq:matern_spde} has a unique solution, we need to augment the Laplacian with boundary conditions.
We consider Dirichlet boundary conditions and denote the resulting operator by $\Delta_D$.
It is known that there is a basis for $L_2(M)$ consisting of eigenfunctions $e_i$ of $-\Delta_D$ with eigenvalues $\lambda_i$.
The functions $e_i$ can be chosen so that they are infinitely differentiable on $M$ and so that $e_i|_{\partial M} \equiv 0$.
Moreover, Weyl's law~\eqref{eq:weyls-law} holds, that is $c i^{2/d} \leq \lambda_i \leq C i^{2/d}$. 
See Section~6 in~\citet{Davies_1995}, in particular Corollary~6.2.2 and Theorem~6.2.3.
We can thus define the spaces
\begin{align*}
    H^s_{L_D} &= \big\lbrace f \in L_2(M) ~ \bigl\vert ~ (1 - \Delta_D)^{s/2} f \in L_2(M) \big\rbrace, \\
    \hat{H}^s_{L_D} &= \bigg\lbrace \{a_i\} \subset \mathbb{R} ~ \biggl\vert ~ \sum_{i = 1}^\infty a_i^2 (1 + \lambda_i)^{s} < \infty \bigg\rbrace ,
\end{align*}
with corresponding norms $\lVert \cdot \rVert_{\theta}$ and $\lVert \cdot \rVert_{\hat{\theta}}$ for some $\tau> 0$.
These spaces are analogous to the spaces $H^s_L$ and \smash{$\hat{H}^s_L$} that we defined in Section~\ref{sec:function-spaces}.

The spaces $H^s_{L_D}$ have been characterized in Lemma~4.3 of~\citet{bolin_2023} for wide range of values of~$s$.
Define $\mathfrak{C} = \left\{s \in \mathbb{R} ~ \middle \vert ~ s \neq 2k + 1/2 \text{ for some } k \in \mathbb{N} \right\}$.
For any $s \notin \mathfrak{C}$ we have
\begin{align}
    H^{s}_{L_D} = \bigg\{f \in H^s(M) ~ \biggl \vert ~ (1 - \Delta)^j f = 0 \text{ in } L^2(\partial M) \text{ for all } j \leq \left\lfloor \frac{2s - 1}{4} \right\rfloor \bigg\}, \label{eq:dirichlet_space_characterization}
\end{align}
and the norm $\lVert \cdot \rVert_{\theta}$ is equivalent to $\lVert \cdot \rVert_{H^s(M)}$ on $H^s_{L_D}$.
This presents the key difference to the manifold case: rather than equal to the Sobolev spaces, the spaces $H^s_{L_D}$ are now proper subspaces.
This makes the analysis more difficult.
However, we do retain norm equivalence.

Denote the resulting kernel of $H_{L_D}^s$ by $K_{\theta}^D$.
Our assumptions on $M$ guarantee that the Sobolev imbedding theorem holds \citep[Theorem 2.33]{aubin_1998} and thus $H^s(M) \subset C^{0, \alpha}(\bar{M})$ for all $\alpha > s - d/2$ and the imbedding is compact.
Thus, Proposition \ref{prop:kernel_properties} is essentially unchanged, we only need $x_i \notin \partial M$ to guarantee strict positive definiteness.
The results in Section \ref{sec:equivalent_measures} apply without changes.
Let $u$ now be defined as
\begin{align}\label{eq:true_model_dirichlet}
 u(x) = \sigma_0 \sqrt{\smash[b]{v(\theta_0)}} \sum_{i = 1}^\infty (\tau_0 + \lambda_i)^{-s_0/2} \xi_i e_i(x),
\end{align}
where $\{(\lambda_i, e_i)\}$ are the eigenpairs of $\Delta_D$.

Unfortunately, we cannot show that all results in Section~\ref{sec:MLE} hold in the same generality.
First, we need to slightly adapt our sampling assumption.
Define
\begin{align*}
    \tilde{q}_{\mb{x}} = \min\left\{ \frac{1}{2} \mathrm{dist}(\mb{x}, \partial M), q_{\mb{x}} \right\} \quad \text{ and } \quad \quad h_{\mb{x}} = \sup_{x \in M} \min_{1 \leq i \leq n} \lVert x_i - x \rVert,
\end{align*}
where $q_{\mb{x}} = \min_{i \neq j} \lVert x_i - x_j \rVert/ 2$ and $h_{\mb{x}}$ are analogous to the quantities in~\eqref{eq:fill-distance}.
The difference is that we now have to work with $\tilde{q}_{\mb{x}}$ to make sure that the points are far enough away from the boundary $\partial M$.
Set $\tilde{\rho}_{\mb{x}} = \tilde{q}_{\mb{x}} / h_{\mb{x}}$.
We then assume that $\tilde{q}_{\mb{x}} \asymp h_{\mb{x}} \asymp n^{-1/d}$ and call points that fulfil this condition quasi-uniform.

Then the results of Propositions~\ref{prop:smoothness_limit} and~\ref{prop:lower_bound_norm_interpolant} apply unchanged.
However, the result of the Proposition~\ref{prop:escape_manifold} must be changed.
Firstly, we need $s \notin \mathfrak{C}$.
Also, at the moment we cannot show the analogue of \eqref{eq:concrete_escape},
\begin{align*}
    \lVert u - m_{s, n} \rVert_{0} \leq h_{\mb{x}}^{s_0-d/2-\epsilon}\tilde{\rho}_{\mb{x}}^{s-s_0 + d/2 + \epsilon} \lVert u \rVert_{H^{s_0 -d/2 - \epsilon}_{L_D}}
\end{align*}
for the minimum norm interpolant $m_{s, n}$ from $H^{s}_{L_D}$.
However, we can show the following.

\begin{proposition}\label{prop:almost_escape}
    Let $m_{s, n}$ be the minimum norm interpolant from an RKHS $\mathcal{H}^s_D$ that is equal to $H^s_{L_D}$ up to equivalent norms.
    Let $u$ be defined as in \eqref{eq:true_model_dirichlet} with $s_0 > d$ and $s \geq s_0$, $s \notin \mathfrak{C}$.
    Then almost surely for every $0 < \epsilon < s_0 - d/2$ such that $s_0 - d/2 - \epsilon \notin \mathfrak{C}$ we have
    \begin{align*}
        \lVert u - m_{s, n} \rVert_0 \leq C h_{\mb{x}}^{s_0-d/2-\epsilon}\tilde{\rho}_{\mb{x}}^{s-s_0 + d/2 + \epsilon} \lVert u \rVert_{s_0 - d/2 - \epsilon} + C h_{\mb{x}}^{s}\lVert m_{s, n} \rVert_{s}.
    \end{align*}
\end{proposition}
\begin{proof}
    Let $\tilde{m}_{s,n}$ be the minimum norm interpolant from $H^s(M)$.
    Since $m_{s, n} \in H^s(M)$, $\tilde{m}_{s,n}$ is also the minimum norm interpolant for $m_{s, n}$.
    Using the escape result from \citet[Theorem~4.2]{narcowich_2006} and norm equivalence, we have 
    \begin{align*}
        \lVert u - m_{s, n} \rVert_0 &\leq \lVert u - \tilde{m}_{s,n} \rVert_{0} + \lVert \tilde{m}_{s,n} - m_{s, n} \rVert_{0} \\
        &\leq Ch_{\mb{x}}^{s_0 - d/2 - \epsilon} \tilde{\rho}_{\mb{x}}^{s-s_0 + d/2 + \epsilon} \lVert u \rVert_{H^{s_0 - d/2 - \epsilon}(M)} + Ch_{\mb{x}}^{s} \lVert m_{s, n} \rVert_{H^s(M)} \\
        &\leq Ch_{\mb{x}}^{s_0 - d/2 - \epsilon} \tilde{\rho}_{\mb{x}}^{s-s_0 + d/2 + \epsilon} \lVert u \rVert_{s_0 - d/2 - \epsilon} +C h_{\mb{x}}^{s} \lVert m_{s, n} \rVert_{s}. \qedhere
    \end{align*}
\end{proof}
As long as $s \notin \mathfrak{C}$, the result of Proposition \ref{prop:upper_bound_minimum_norm_interpolant} holds without change if we replace $H^s_L$ by $H^s_{L_D}$.
Finally, the results of Proposition \ref{prop:cond_var_bounds} are well known in the Euclidean case~\citep[e.g.,][Proposition~3.6 and~3.7]{karvonen2023asymptotic}.
We only need $s \notin \mathfrak{C}$ to translate these results to our setting.

\begin{theorem}\label{cor:consistency_dirichlet}
  Let $u$ be defined as in \eqref{eq:true_model_dirichlet} with $s_0 > d$.
  Let $\{K_{s}\}_{s > d/2}$ be any family of kernels such that the RKHS of $K_s$ is equal to $H^s_{L_D}$ up to equivalent norms and such that the corresponding norms are monotone.
  Assume that the sequence of points $\{x_i\}$ is quasi-uniform.
  If $\{\hat{s}_n\}$ is a sequence of maximizers of the Gaussian log-likelihood
    \begin{align*}
        \ell(s; u(\mb{x})) = -\frac{n}{2}\log(2\pi) - \frac{1}{2}\log(\det (K_{s}(\mb{x}))) - \frac{1}{2}u(\mb{x})^{\top}K_{s}(\mb{x})^{-1}u(\mb{x})
    \end{align*}
    in $(d/2, S_{\max})$ with $s_0 \leq S_{\max}$, then
    \begin{align*}
        \hat{s}_n \to s_0 \quad \text{ in probability}.
    \end{align*}
\end{theorem}
\begin{proof}
    The proof is the same as of that of Theorem \ref{thm:consistency}, the only change being that Proposition~\ref{prop:almost_escape} is used in the place of Proposition~\ref{prop:escape_manifold}.
    Note that the restriction $s \notin \mathfrak{C}$ in \eqref{eq:dirichlet_space_characterization} is not a problem.
    Because \smash{$u \in H^{s_0 - d/2 - \epsilon}_L$} for all $\epsilon > 0$, we can choose $\delta$ and $\epsilon$ appearing in the proof of Theorem \ref{thm:consistency} in such a way that we avoid any problems.
\end{proof}

While the path regularity result in the Gaussian case can be proved in the same way as in Proposition \ref{prop:path_property}, extending Corollary \ref{cor:consistency_dirichlet} to the case $s_0 > d/2$ is less straightforward as the results in \citet{hangelbroek_2018} are proved only for closed manifolds.
However, almost sure convergence can be shown in the Gaussian case.

One should be able to establish similar results for von Neumann boundary conditions with a similar procedure.
Some results characterizing the corresponding spaces $H^s_L$, can be found in~\citet{kim_2020} and Proposition B.2 of ~\citet{bolin_2024}.

\section{Simulation Study}\label{sec:simulation_study}

We conduct simulation studies to assess finite sample properties and stability against model misspecification for smoothness estimation.
As underlying manifold, we use the sphere $\mathbb{S}^{2}$.
In this case, the spectral decomposition of the covariance operators is given as 
\begin{align*}
    (\tau-\Delta)^{-s} = \sum_{l = 0}^\infty (\tau+\lambda_l)^{-s} \sum_{m = -l}^l e_{lm} \otimes e_{lm}, 
\end{align*}
where $\lambda_l = l(l+1)$ and the functions $e_{lm}$ are the so-called real spherical harmonics for which explicit expressions are available \citep[Chapter~IV]{stein_1971}.
Even though in this case the eigenfunctions and eigenvalues are known, in practice we must truncate the infinite sums in order to compute the kernel $K_{\theta}$ of $(\tau - \Delta)^{s}$ and sample the corresponding random fields. 
With the addition theorem for spherical harmonics \citep[Chapter~IV, Corollary~2.9]{stein_1971}, it is straightforward to show that the error incurred on the diagonal of the kernel when truncating at $l$ is of order $l^{2(1 - s)}$.
An application of the Cauchy--Schwarz inequality then shows that the same bounds holds for the off-diagonal elements as well.
Let $K_\theta$ stand for the truncated kernel.
Because the addition theorem implies that the truncated kernel is constant as a function of $x$ for any truncation level $l$, we may normalize the kernel such that $K_{\theta}(x) = K_\theta(x, x) \equiv 1$.
In our simulation studies we chose $l = 100$.
That is, we use the first $\sum_{i=0}^{100} (2i+1) = 10\,201$ eigenfunctions and eigenvalues.
With $s$ chosen as in the simulation setup below, this results in a uniform pointwise error for the kernel of order $10^{-16}$.

We consider the following three cases of model misspecification: (i) we use Gaussian observations and the correct (truncated) kernel but with incorrect $\tau$ (range scale) parameters, (ii) we use non-Gaussian observations, and (iii) we estimate the smoothness using Eucliean Mat\'ern kernels and Euclidean generalized Wendland kernels restricted to the sphere even though the observations are generated using the (truncated) kernel $K_{\theta_0}$.
For scenario (ii), we always standardize the Karhunen--Lo\`eve coefficients to have mean 0 and variance 1.
In the last scenario (iii), the other kernel parameters have to be specified
In an attempt to isolate the effect of the kernel, we chose the other parameters before optimization by minimizing the difference to the true covariance.
This works very well for Euclidean Mat\'ern kernels but slightly less well for the generalized Wendland kernels.

To assess the effect of sample size, we consider sample sizes $n = 500, 1000, 2000$.
For each sample size, we generate 100 samples of a Gaussian process with true parameters $\theta_0 = (s_0, \tau_0) = (5, 20)$ and 100 samples of non-Gaussian processes for scenario (ii) in which we use the true range scale parameter during smoothness estimation.
The range over which we optimize is fixed to $[1 + 10^{-7}, 30]$.
As our theoretical results require quasi-uniformity, we use nearly equidistant points on the sphere generated by the regular placement algorithm given in \cite{deserno_2004}. 
To maximize the Gaussian log-likelihood, we use the implementation of the SHGO global optimization algorithm from the Python library \texttt{scipy}.
When using simpler (local) optimizers, we found that they could on occasion get stuck in local minima, especially for larger sample sizes, perhaps due to numerical instability of the likelihood evaluation.
A summary of the resulting estimates for the different scenarios can be found in the violin plots in Figures~\ref{fig:mis_dist}--\ref{fig:mis_kernel}.
Overall, we find that smoothness estimation appears to work well even for misspecified models.

First, from Figure \ref{fig:mis_dist}, we note that estimation appears to be stable against distribution misspecification of the Karhunen--Lo\`eve  coefficients $\{\xi_i\}$. 
When the coefficients are Bernoulli distributed, estimation appears to work even better than in the case of a correctly specified model (compare with Figure \ref{fig:mis32}).
For exponentially and $t$ distributed coefficients, the estimation of $s_0$ appears to work similarly well.
In all cases, the bias appears to be very small and decreases with increasing sample size.

Next, from Figure \ref{fig:mis_range_scale}, we note that misspecifying the range scale parameter results in biased smoothness estimates for finite sample sizes.
Moreover, there is clear positive correlation between range scale parameter and estimated smoothness parameter.
In our opinion, the bias is not very large relative to range scale misspecification.
Concretely, a normalized kernel with parameter $\theta = (s, \tau) = (5, 1)$ is somewhat pathological as the correlation is never smaller than 0.975 across the entire sphere while with $\tau_0 = 20$ the correlation is essentially 0 at large distances.
However, even when $\tau$ is not grossly misspecified, the bias appears to decrease only slowly as the sample size increases.

Finally, smoothness estimation also works well with misspecified kernels, see Figure \ref{fig:mis_kernel}.
In this setting the Euclidean Mat\'ern kernel appears to be superior to the generalized Wendland kernel, perhaps due to the fact that the true kernel is more closely related to the Mat\'ern kernel.
It is not clear whether the biases present in the estimation results are inherent to the kernel used or whether they are the result of the specification of the remaining kernel parameters.

\begin{figure}[!htb]
\begin{subfigure}{0.33\textwidth}
\centering
\includegraphics[width=\linewidth]{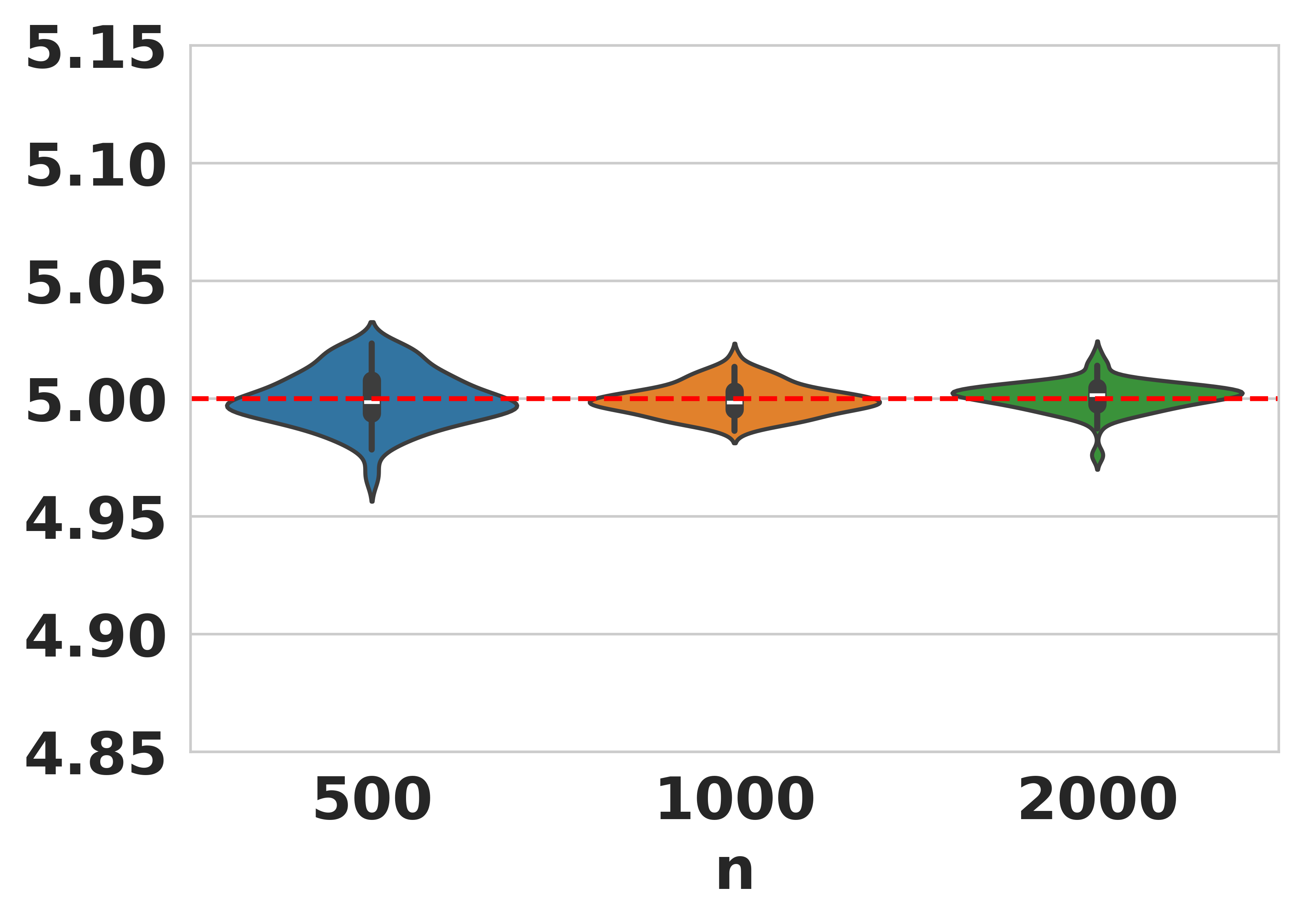}
\caption{$\xi_i \sim$ Bernoulli}
\end{subfigure}
\begin{subfigure}{0.33\textwidth}
\centering
\includegraphics[width=\linewidth]{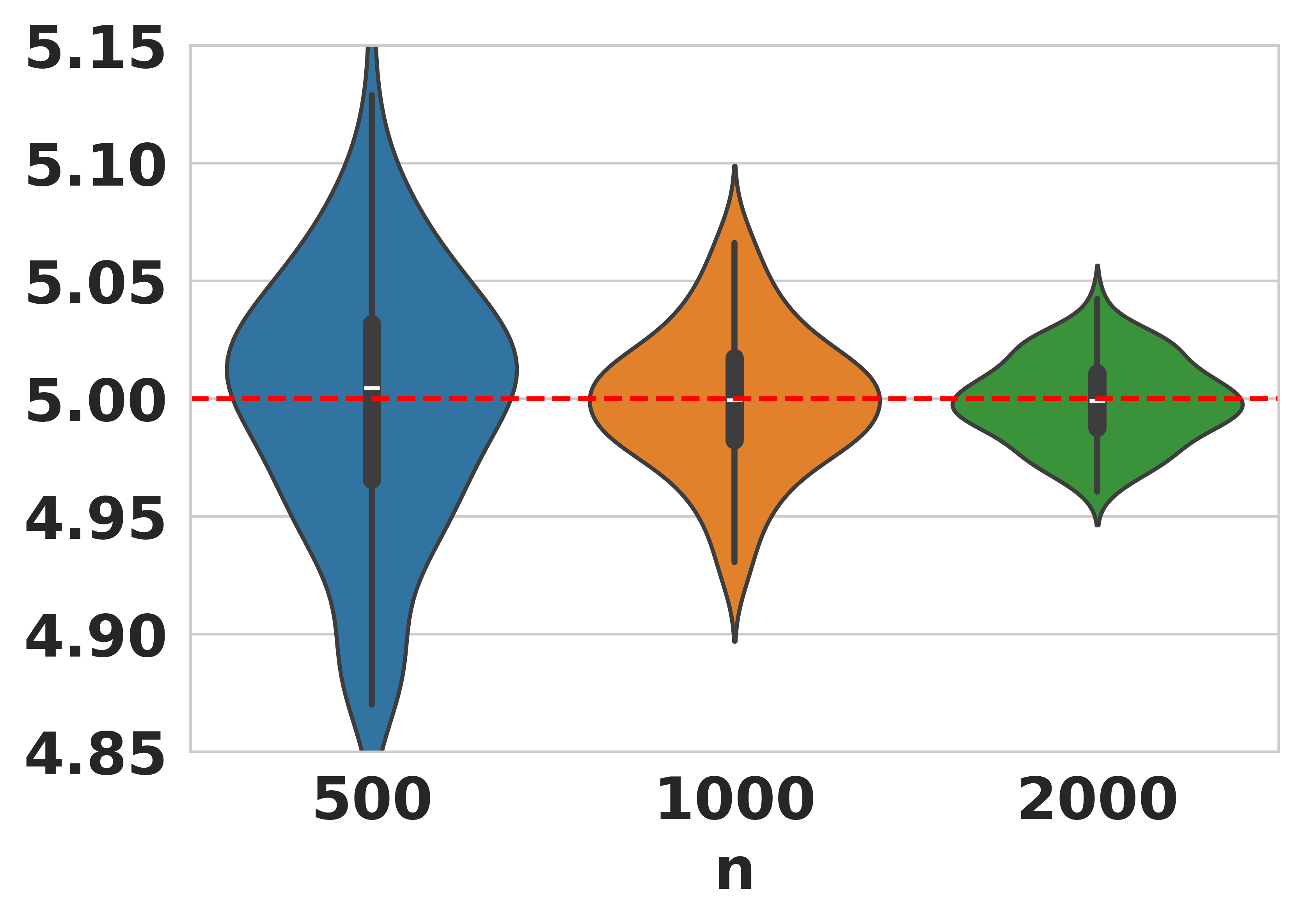}
\caption{$\xi_i \sim$ Exponential}
\end{subfigure}
\begin{subfigure}{0.33\textwidth}
\centering
\includegraphics[width=\linewidth]{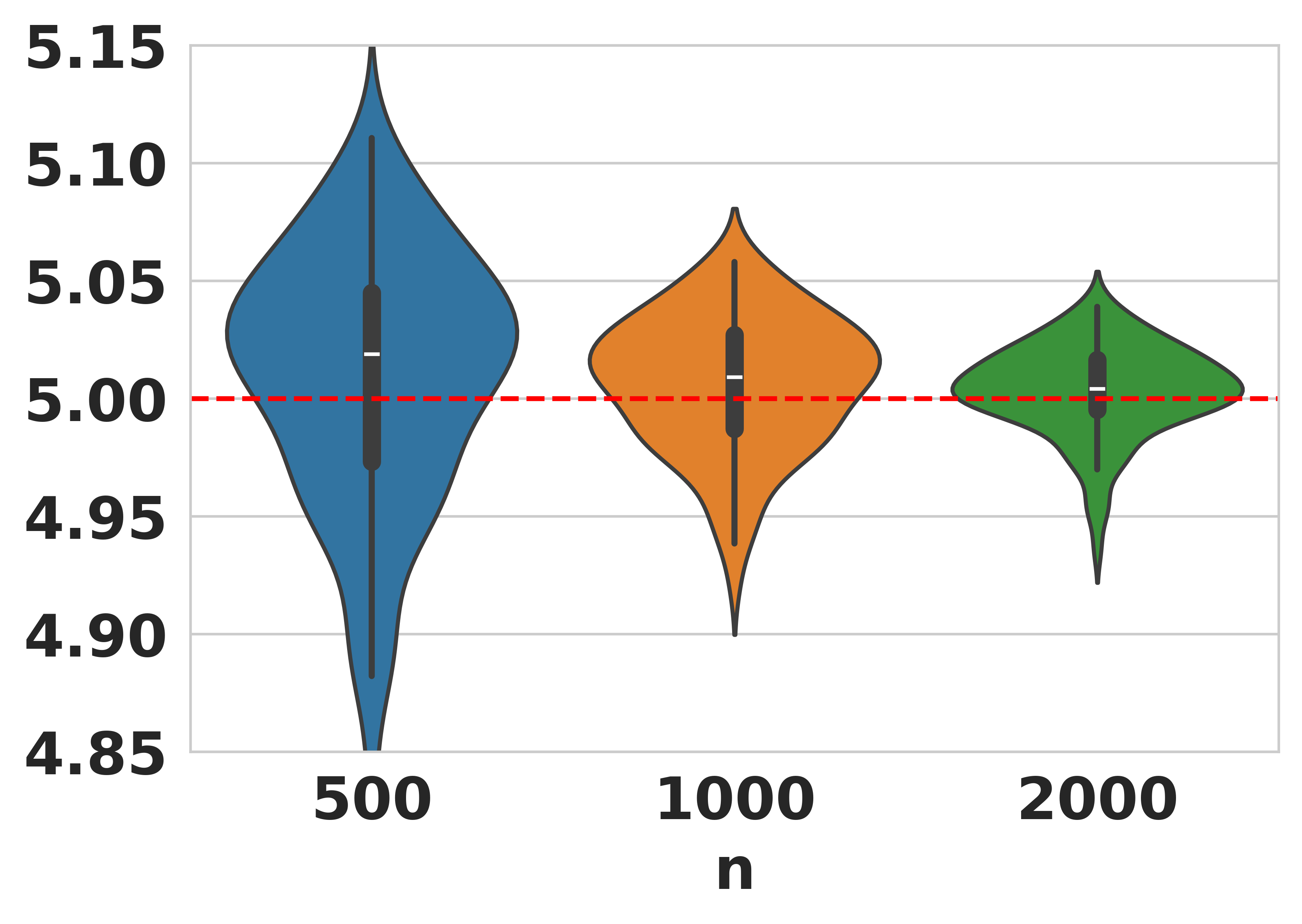}
\caption{$\xi_i \sim$ $t$ distribution (df = 4)}
\end{subfigure}
\caption{Robustness against distributional misspecification. Red dashed line corresponds to $s_0 = 5$.}
\label{fig:mis_dist}
\end{figure}

\begin{figure}[!htb]
\begin{subfigure}{0.33\textwidth}
\centering
\includegraphics[width=\linewidth]{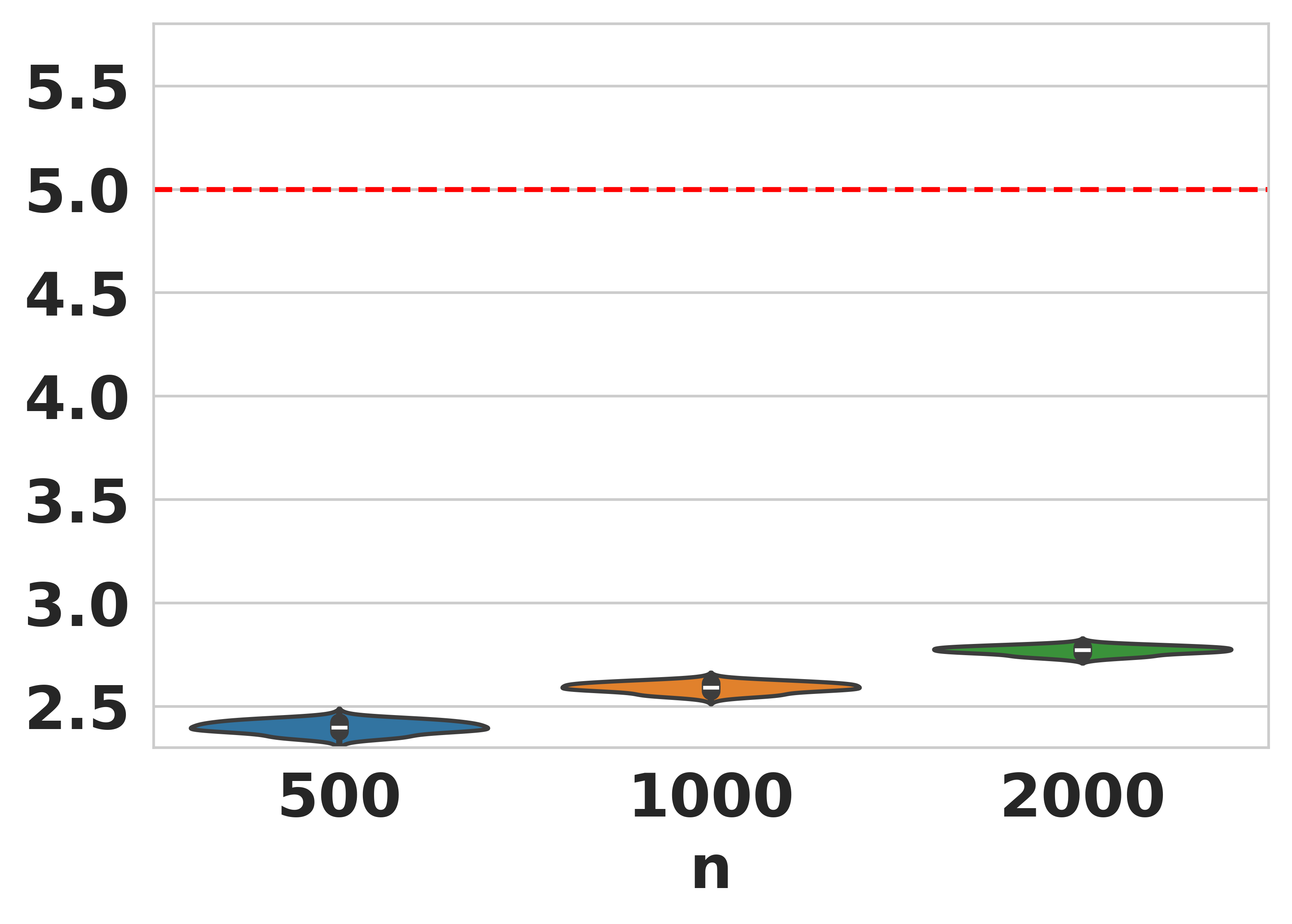}
\caption{$\tau = 1$}
\label{fig:mis11}
\end{subfigure}
\begin{subfigure}{0.33\textwidth}
\centering
\includegraphics[width=\linewidth]{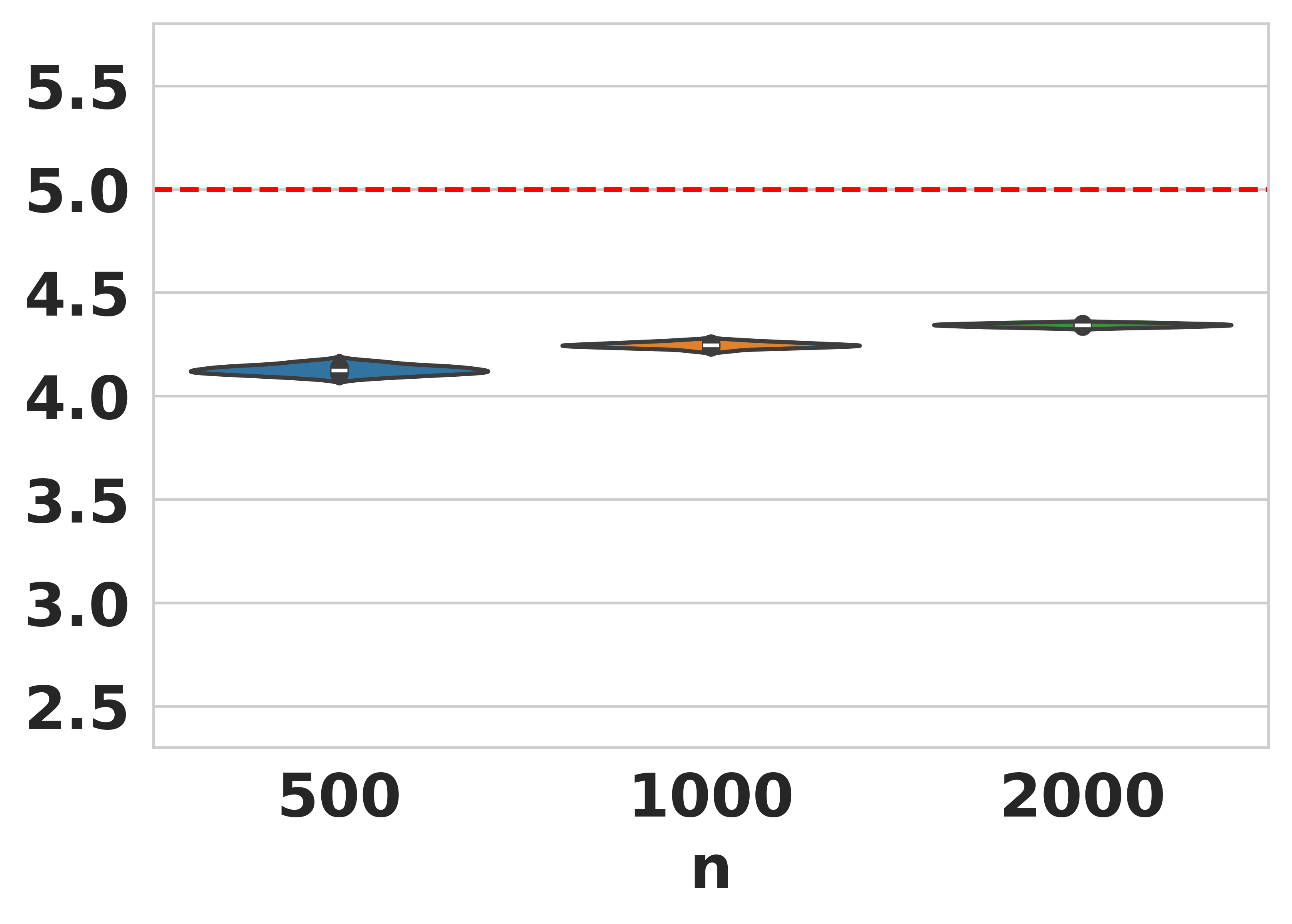}
\caption{$\tau = 10$}
\label{fig:mis12}
\end{subfigure}
\begin{subfigure}{0.33\textwidth}
\centering
\includegraphics[width=\linewidth]{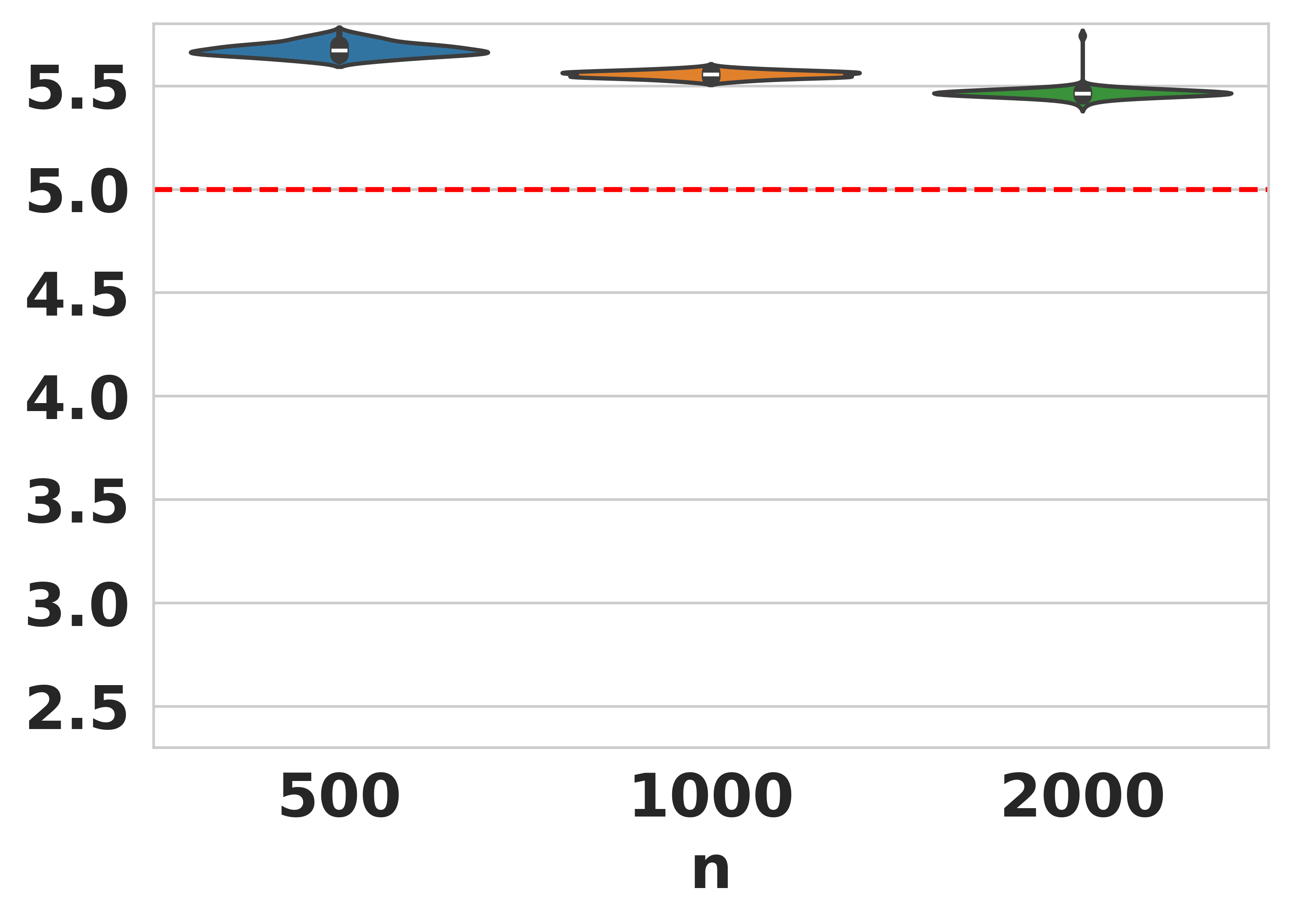}
\caption{$\tau = 30$}
\label{fig:mis13}
\end{subfigure}
\caption{Robustness against misspecification of the range scale parameter. True parameter is $\tau_0 = 20$. Red dashed line corresponds to $s_0 = 5$.}
\label{fig:mis_range_scale}
\end{figure}

\begin{figure}[!htb]
\begin{subfigure}{0.33\textwidth}
\centering
\includegraphics[width=\linewidth]{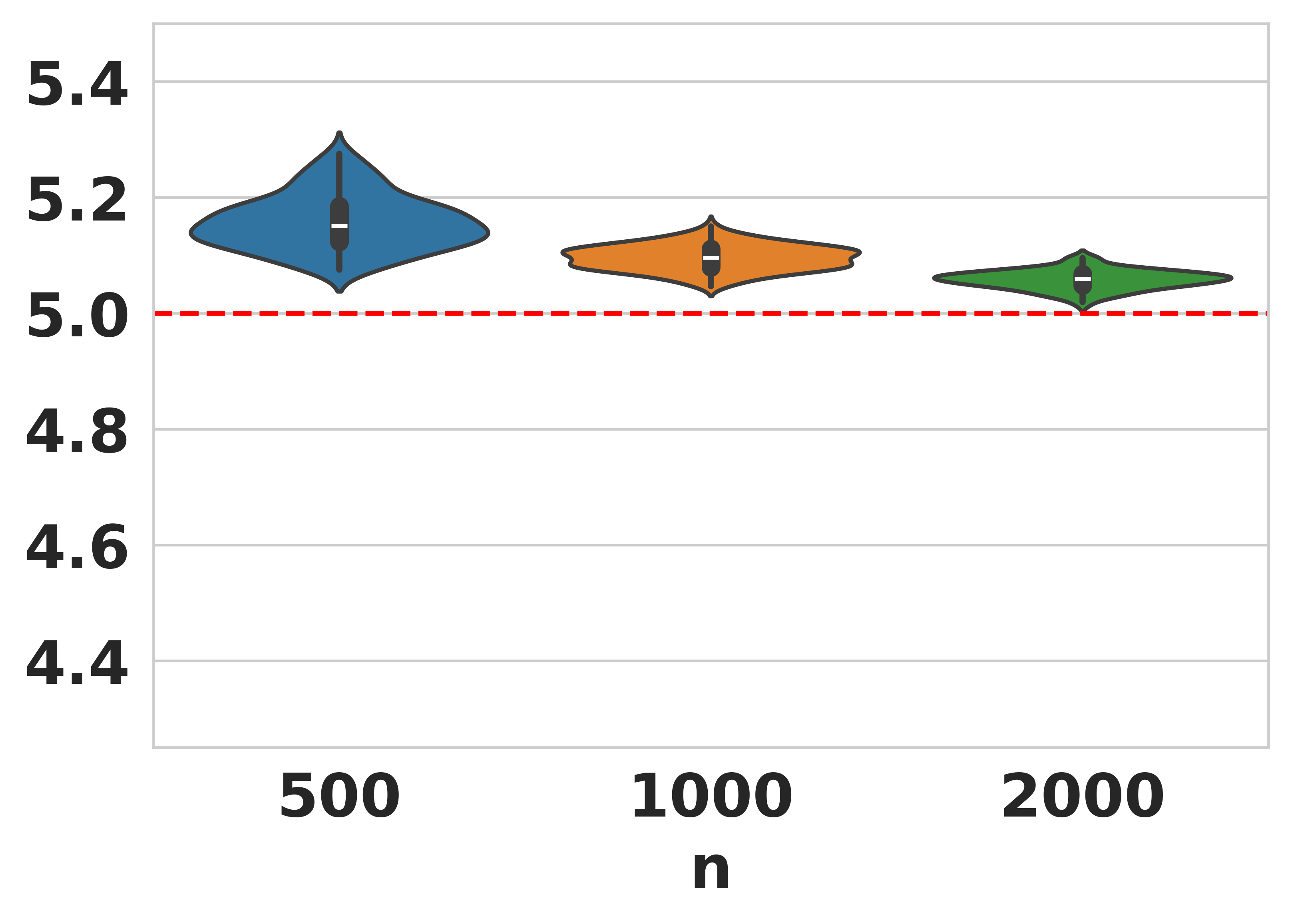}
\caption{Euclidean Mat\'ern Kernel}
\label{fig:mis31}
\end{subfigure}
\begin{subfigure}{0.33\textwidth}
\centering
\includegraphics[width=\linewidth]{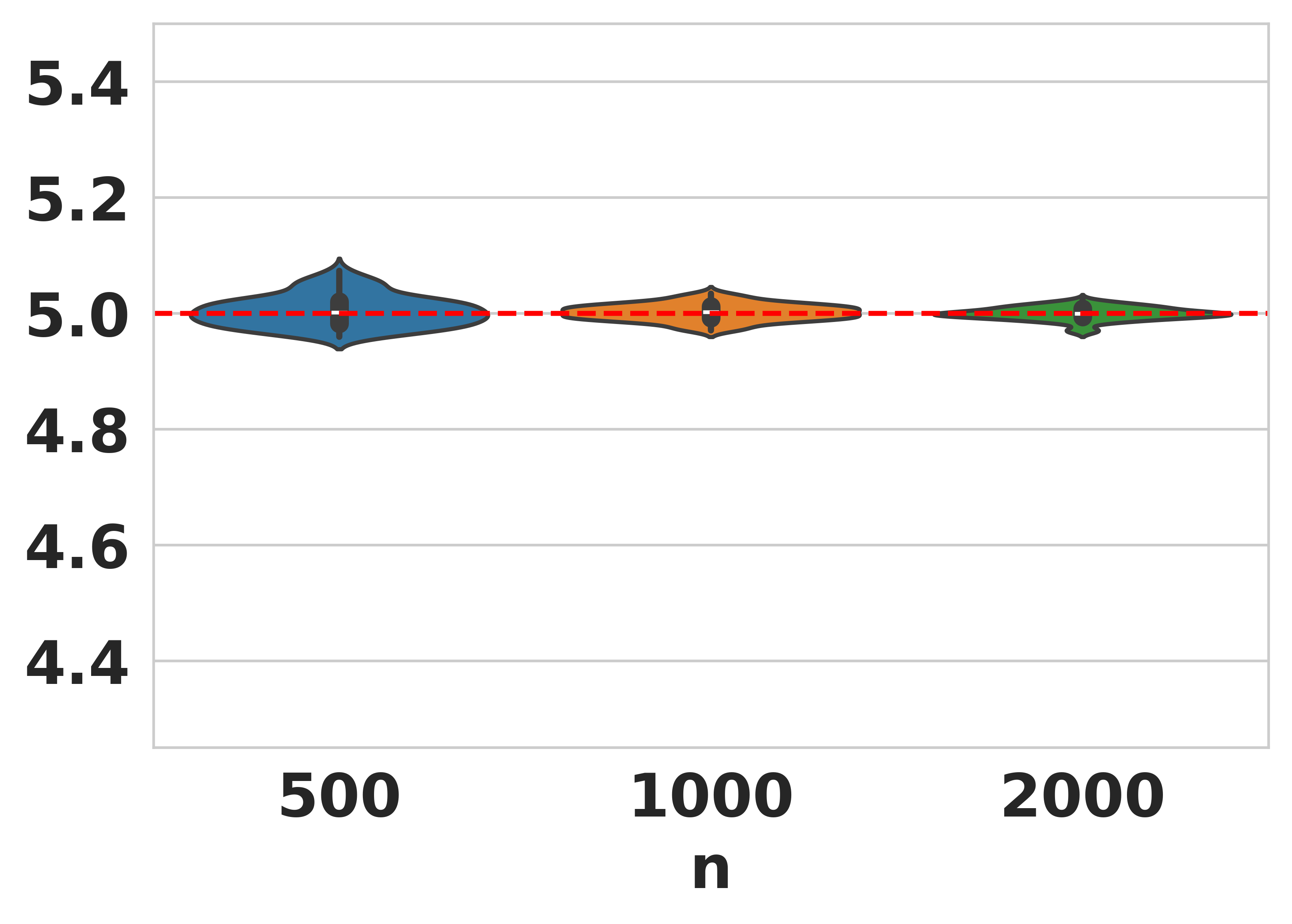}
\caption{Correctly specified kernel.}
\label{fig:mis32}
\end{subfigure}
\begin{subfigure}{0.33\textwidth}
\centering
\includegraphics[width=\linewidth]{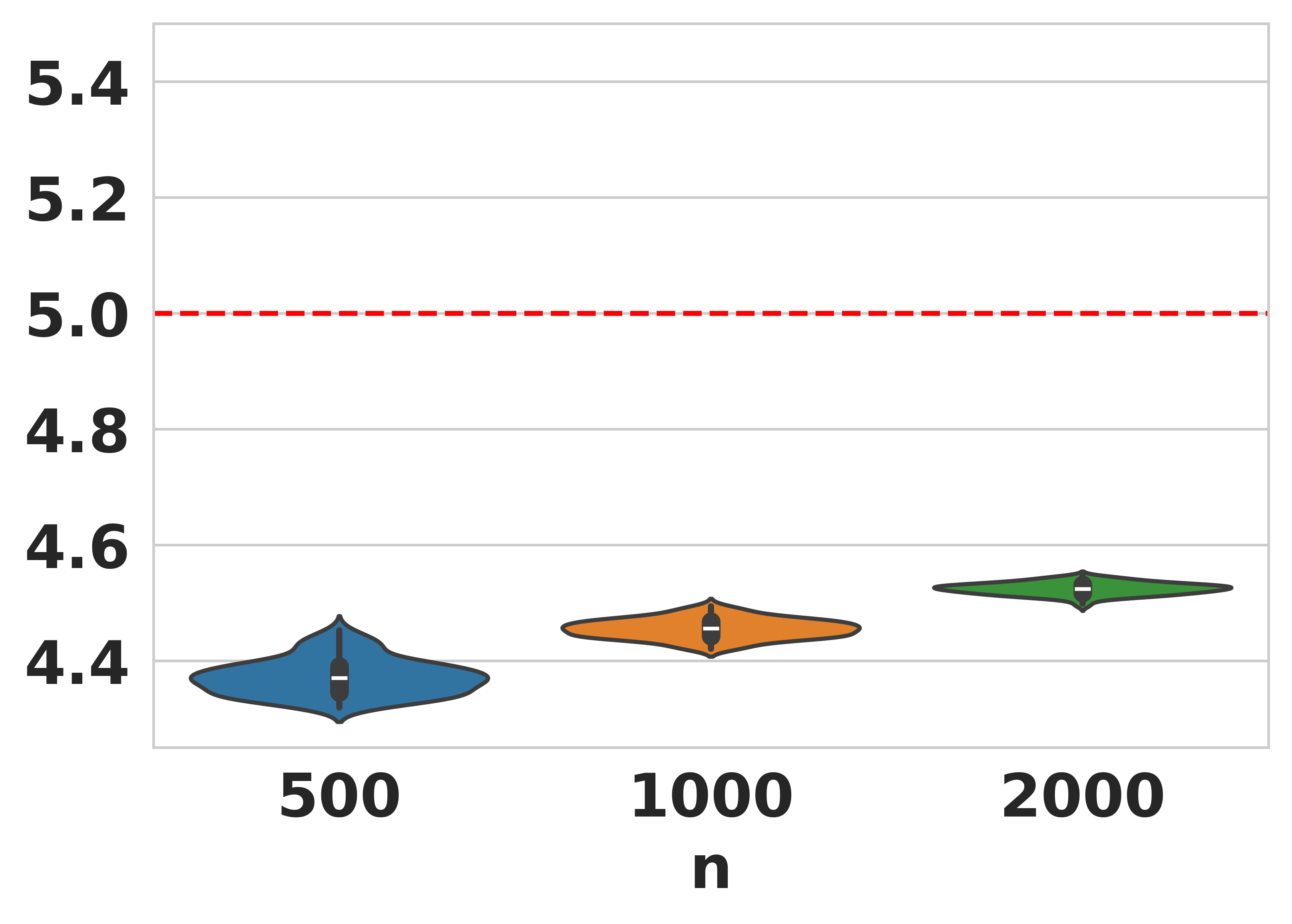}
\caption{Euclidean generalized Wendland kernel}
\label{fig:mis33}
\end{subfigure}
\caption{Robustness against misspecification of the kernel. Red dashed line corresponds to $s_0 = 5$.}
\label{fig:mis_kernel}
\end{figure}

\section*{Acknowledgements}
MK-S was supported by NSF Grant DMS-1916226.
TK was supported by the Research Council of Finland postdoctoral researcher grant number 338567, ``Scalable, adaptive and reliable probabilistic integration''.
Finally, we would like to thank the reviewers and the editor for important comments and suggestions that substantially improved the presentation of the results in the paper.

\appendix

\renewcommand{\thetheorem}{\Alph{section}.\arabic{theorem}}
\renewcommand{\theproposition}{\Alph{section}.\arabic{theorem}}


\section{Further Background on Manifolds and Function Spaces.}\label{sec:appendix_manifold}
Here, we provide additional technical details on our manifold setting that are necessary for our proofs.
Throughout, we require background on manifolds for which we again refer to~\citet{lee_smooth, lee_riemannian}.
Recall that $\mathcal{M} \subset \mathbb{R}^k$ is assumed to be a connected and closed (i.e. compact and without boundary) embedded $d$-dimensional Riemannian submanifold in the sense of \citep[Chapter 8]{lee_riemannian}.
Thus, the Riemannian metric $g$ on~$\mathcal{M}$ is the one induced by the standard Riemannian metric on $\mathbb{R}^k$.
The tangent space at a point $p \in \mathcal{M}$ is denoted by $T_p\mathcal{M}$.
We denote by $\nabla$ the Levi-Civita connection.
The exponential map $\exp^{\mathcal{M}}_p : \mathcal{E} \subset T_{p}\mathcal{M} \to \mathcal{M}$ is given by
\begin{align*}
    \exp^{\mathcal{M}}_p(v) = \gamma_v(1) ,
\end{align*}
where $\gamma_v$ is the unique geodesic with $\gamma(0) = p$ and $\gamma^\prime(0) = v$.
Since $\mathcal{M}$ is compact, it has a finite but positive injectivity radius $r_0 > 0$, which means that whenever $r \leq r_0$ the exponential map $\exp^{\mathcal{M}}_p$ is a diffeomorphism on the sets \sloppy{${B_{r}(\mb{0}) = \{ v \in \mathbb{R}^d ~ \vert ~ \lVert v \rVert < r \}}$} for all $p \in \mathcal{M}$ (where we have identified $T_p\mathcal{M}$ with $\mathbb{R}^d$) with image $\exp_p^{\mathcal{M}}(B_{r}(\mb{0})) = \Omega_p(r)$.
This results in local charts given by $\{(\Omega_p(r), (\exp^{\mathcal{M}}_p)^{-1})\}$, which give the so-called normal coordinates.
Due to compactness, we can cover $\mathcal{M}$ with finitely many $\{\Omega_{p_i}(r)\}$, where the $p_i$ are points on $\mathcal{M}$ to obtain an atlas for $\mathcal{M}$ of normal coordinates.
To simplify notation, we write $\phi_i$ for $(\exp^{\mathcal{M}}_{p_i})^{-1}$ and $\Omega_i$ for $\Omega_{p_i}(r)$.
Furthermore, let $\{\psi_i\}$ be a partition of unity subordinate to $\{(\Omega_i, \phi_i)\}$, which is to say that $\supp \psi_i \subset \Omega_i$.
Such a partition always exists, see Proposition~7.2.1 in~\citet{triebel_1991}.
We require   $0 < r \leq r_0/32$ for technical reasons~\citep[cf.][Remark 7.2.1/2]{triebel_1991}.

To define the Laplace-Beltrami operator, let $\mb{G} = (g_{ij})_{ij}$ be the component matrix of the Riemannian metric on $\mathcal{M}$ in local coordinates.
The Laplace--Beltrami operator $\Delta$ is the operator given in local coordinates for $f \in C^{\infty}(\mathcal{M})$ as
\begin{align*}
    \Delta f = \frac{1}{\sqrt{\smash[b]{\det(\mb{G})}}} \bigg(\sum_{i = 1}^n \partial_i \sqrt{\smash[b]{\det(\mb{G})}} \, \sum_{j = 1}^n g^{ij}\partial_j f \bigg),
\end{align*}
where $g^{ij}$ are the entries of $\mb{G}^{-1}$ and $\{\partial_i\}_{i = 1}^d$ are the basis for the $T_p\mathcal{M}$.
There is a unique, self-adjoint and negative semi-definite extension of $\Delta$ [originally defined on $C^{\infty}(\mathcal{M})$] to a domain $\mathcal{D}(\Delta) \subset L_2(\mathcal{M})$ \citep{strichartz_83}.
We identify the extension with $\Delta$.

Next, we give a definition of the Hölder spaces.
For $0 < \alpha < 1$ and a function $f : \mathcal{M} \to \mathbb{R}$ such that the $m$th covariant derivative $\nabla^m$ exists we define the Hölder semi-norm
\begin{align*}
    \lvert \nabla^m f \rvert_{\alpha} \coloneqq \sup_{x\neq y \in \mathcal{M}} \frac{\lVert \nabla^m f(x) - \nabla^m f(y) \rVert_g}{d_{\mathcal{M}}(x, y)^{\alpha}}.
\end{align*}
The norm $\lVert \cdot \rVert_g$ is the natural norm induced by the Riemannian metric $g$, see for example \citet{hangelbroek_2010}, Section 2.2 for details.

Following \citet{aubin_1998}, for any $l \in \mathbb{N}$ and $0 < \alpha < 1$ we define the Hölder space
\begin{align*}
    C^{l, \alpha}(\mathcal{M}) = \Big\lbrace f: \mathcal{M} \to \mathbb{R} ~ \Bigl\vert ~ \lVert f \rVert_{C^{l, \alpha}(\mathcal{M})} \coloneqq \lvert \nabla^{l} f \rvert_{\alpha} + \max_{m \leq l} \sup_{x \in \mathcal{M}} \lVert \nabla^m f(x) \rVert_g < \infty \Big\rbrace.
\end{align*}

\section{Proofs for Section \ref{sec:MLE}}\label{sec:tech_proofs}

In this section we provide proofs for some of the results in Section~\ref{sec:MLE}.

\subsection{Proofs of Propositions \ref{prop:smoothness_limit} and \ref{prop:lower_bound_norm_interpolant}}

Before we can provide the proof, we need tail bounds for sums of the form 
\begin{align*}
    \sum_{i = R}^\infty a_i (\xi_i - \E(\xi_i)),
\end{align*}
where $a_i \geq 0$ and $\xi_i$ is a sub-Gaussian random variable.
These are standard results, but we have not been able to find a reference for infinite sums.
Since the proof is short, we include it here.
Following Definition~2.2 in \citet{wainwright_2019}, a random variable $X$ is sub-Gaussian with sub-Gaussian parameter $\gamma$ if 
\begin{align*}
    \E \big(e^{t (X- \E(X))} \big) \leq e^{\gamma^2t^2 / 2}
\end{align*}
for all $t \in \mathbb{R}$.

\begin{lemma}\label{l:tail_bound}
    Let $\{a_i\}$ be a square-summable sequence with $a_i \geq 0$ and $\{\xi_i\}$ an i.i.d.\ sequence of sub-Gaussian random variables with mean $\mu$ and sub-Gaussian parameter $\gamma$. 
    Then for $ t \geq 0$, we have
    \begin{align*}
        \P \bigg(\, \bigg\lvert\sum_{i = R}^\infty a_i (\xi_i - \mu) \bigg\rvert \geq t \bigg) \leq 2\exp \bigg(-  \frac{t^2}{2\gamma^2\sum_{i = R}^\infty a_i^2}\bigg).
    \end{align*}
\end{lemma}

\begin{proof}
    Note first that by definition, we have
    \begin{align*}
        \lim_{M \to \infty}\log \left[\E \left( \exp \bigg(\lambda\sum_{i = R}^M a_i (\xi_i - \mu) \bigg)\right) \right] \leq \frac{(\gamma\lambda)^2}{2} \sum_{i = R}^\infty a_i^2 < \infty.
    \end{align*}
    Because $\log$ is continuous and monotonically increasing, Fatou's lemma implies that
    \begin{align*}
        \log \left[\E \left( \exp \bigg(\lambda\sum_{i = R}^\infty a_i (\xi_i - \mu) \bigg)\right) \right] &\leq \liminf_{M \to \infty} \log \left[\E \left( \exp \bigg(\lambda\sum_{i = R}^M a_i (\xi_i - \mu)\bigg)\right) \right] \\
        &\leq \frac{(\gamma\lambda)^2}{2} \sum_{i = R}^\infty a_i^2.
    \end{align*}
    Thus $\sum_{i = R}^\infty a_i \xi_i$ is sub-Gaussian with sub-Gaussian parameter  $\gamma^2 \sum_{i = R}^\infty a_i^2$ so that Proposition 2.5 in \citet{wainwright_2019} yields, for $t > 0$,
    \begin{equation*}
        \P \bigg(\sum_{i = R}^\infty a_i (\xi_i^2 - 1) \geq t \bigg) \leq \exp \bigg(- \frac{t^2}{2\gamma^2\sum_{i = R}^\infty a_i^2} \bigg). \qedhere
    \end{equation*}
\end{proof}

\begin{namedtheorem}[Proposition \ref{prop:smoothness_limit}]
  Let $s_0 > d/2$ and let $\{\xi_i\}$ be the sequence of i.i.d. random variables used in the definition of $u$ in \eqref{eq:true_model}.
    Then almost surely, there is a positive constant $c$ such that 
    \begin{align*}
        \sum_{i = R}^\infty (\tau_0 + \lambda_i)^{-s_0} \xi_i^2 \geq c R^{-2s_0/d + 1} \quad \text{ for all } \quad R \in \mathbb{N}.
    \end{align*} 
\end{namedtheorem}
\begin{proof}
    If $\xi_i$ satisfies $0 < \E(\xi_i^2) < \infty$, we may select $b$ be such that 
    \begin{align*}
        \tilde{\xi}_i^2 \coloneqq \min\{ \xi_i^2, b\}
    \end{align*}
    has positive mean.
    Then $\tilde{\xi}_i^2$ is bounded and thus sub-Gaussian.
    Then, because 
    \begin{align*}
        \sum_{i = R}^\infty (\tau_0 + \lambda_i)^{-s_0} \xi_i^2 \geq \sum_{i = R}^\infty (\tau_0 + \lambda_i)^{-s_0} \tilde{\xi}_i^2,
    \end{align*}
    it suffices to prove the claim for $\xi_i$ such that $\xi_i^2$ is sub-Gaussian.
    First, note that Weyl's law~\eqref{eq:weyls-law} implies that
    \begin{equation}
      \sum_{i = R}^\infty (\tau_0 + \lambda_i)^{-s_0} \geq c \sum_{i = R}^\infty i^{-2s_0/d} \geq \int_{R}^\infty x^{-2s_0/d} \, \mathrm{d}x = c R^{-2s_0/d + 1} \label{eq:R_equiv-1}
    \end{equation}
    and
    \begin{equation}
        \sum_{i = R}^\infty (\tau_0 + \lambda_i)^{-2s_0} \leq C \sum_{i = R}^\infty i^{-4s_0/d} \leq \int_{R+1}^\infty x^{-4s_0/d} \, \mathrm{d}x = C (R+1)^{-4 s_0/d + 1}. \label{eq:R_equiv-2}
    \end{equation}
    Next, we can write 
    \begin{align*}
        \sum_{i = R}^\infty (\tau_0 + \lambda_i)^{-s_0} \xi_i^2 = \sum_{i = R}^\infty (\tau_0 + \lambda_i)^{-s_0} (\xi_i^2 - 1) + \sum_{i = R}^\infty (\tau_0 + \lambda_i)^{-s_0} .
    \end{align*}
    Dividing through by $R^{-\alpha s_0 + 1}$ combined with \eqref{eq:R_equiv-1} yields
    \begin{align}
        R^{2 s_0/d - 1}\sum_{i = R}^\infty (\tau_0 + \lambda_i)^{-s_0} \xi_i^2 &= R^{2 s_0/d - 1} \bigg( \, \sum_{i = R}^\infty (\tau_0 + \lambda_i)^{-s_0} (\xi_i^2 - 1) + \sum_{i = R}^\infty (\tau_0 + \lambda_i)^{-s_0}\bigg) \nonumber \\
        &\geq c + R^{2s_0/d - 1} \sum_{i = R}^\infty (\tau_0 + \lambda_i)^{-s_0}(\xi_i^2 - 1). \label{eq:lower_bound}
    \end{align}
    Lemma~\ref{l:tail_bound} and \eqref{eq:R_equiv-2} give
    \begin{align*}
        \P \bigg(\bigg\lvert  R^{2 s_0/d - 1}\sum_{i = R}^\infty (\tau_0 + \lambda_i)^{-s_0}(\xi_i^2 - 1) \bigg\rvert \geq \frac{c}{2}\bigg) &\leq 2 \exp \left(- \frac{c^2R^{-4s_0/d + 2}}{8\gamma^2\sum_{i = R}^\infty (\tau_0 + \lambda_i)^{-2s_0}}\right) \\
        &\leq 2 \exp \left(- \frac{c^2R^{-4s_0/d + 2}}{8\gamma^2 C (R+1)^{-4s_0/d + 1}}\right)\\
        &\leq 2 \exp \left(- CR \right).
    \end{align*}
    For the sets
    \begin{align*}
        A_R = \left\{ \left\lvert  R^{2s_0/d - 1} \sum_{i = R}^\infty (\tau_0 + \lambda_i)^{-s_0}(\xi_i^2 - 1) \right\rvert \geq \frac{c}{2}\right\},
    \end{align*}
    we thus have
    \begin{align*}
        \sum_{R = 1}^\infty \P (A_R) \leq \sum_{R = 1}^\infty 2\exp(-CR) < \infty .
    \end{align*}
    An application of the Borel--Cantelli lemma yields that almost surely there exists an $R_0 < \infty$ such that
    \begin{align*}
        \left\lvert R^{2s_0/d - 1} \sum_{i = R}^\infty (\tau_0 + \lambda_i)^{-s_0}(\xi_i^2 - 1)\right\rvert \leq \frac{c}{2}
    \end{align*}
    for all $R \geq R_0$.
    Plugging this back into \eqref{eq:lower_bound}, we obtain
    \begin{align*}
        R^{2 s_0/d - 1}\sum_{i = R}^\infty (\tau_0 + \lambda_i)^{-s_0} \xi_i^2 \geq c + R^{2s_0/d - 1} \sum_{i = R}^\infty (\tau_0 + \lambda_i)^{-s_0}(\xi_i^2 - 1) \geq \frac{c}{2}
    \end{align*}
    for $R \geq R_0$.
    This is equivalent to 
    \begin{align*}
        \sum_{i = R}^\infty (\tau_0 + \lambda_i)^{-s_0} \xi_i^2 \geq \frac{c}{2}R^{-2s_0/d + 1}.
    \end{align*}
    Finally, note that for any $R_0$, the sum of the first $R_0-1$ is strictly positive almost surely and so adjusting the constant $c$ yields the claim for any $R \in \mathbb{N}$.
\end{proof}

\begin{namedtheorem}[Proposition \ref{prop:lower_bound_norm_interpolant}]
  Let $\theta = (s, \tau)$ with $s > d/2$.
  Suppose that $f \in \mathcal{H}^s$ and, almost surely,
    \begin{align}
        \lVert u - f \rVert_0 \leq C_1 n^{- s_0/d + 1/2 +\epsilon/d} + C_2 n^{-s/d} \lVert f \rVert_s \label{eq:app_assumption}
    \end{align}
    for every $\epsilon > 0$, where $C_1,C_2 > 0$ can depend on $s, s_0, \epsilon,\mathcal{M}$ and the sample path.
    Then almost surely
    \begin{align}
        \lVert f \rVert_{s}^2 \geq Cn^{1 + 2(s-s_0)/d - \epsilon^\prime}
    \end{align} 
    for every $\epsilon^\prime > 0$, where $C > 0$ depends on $s, s_0,\epsilon^\prime, \mathcal{M}$ and the sample path.
\end{namedtheorem}
\begin{proof}
The proof is similar to that of Theorem~8 in~\citet{VaartZanten2011}.
Recall the spaces $\hat{H}^{s}_L$ and the isomorphism $\iota: H^{s}_L \to \hat{H}^{s}_L$ given by $\iota(f)_i = \langle f, e_i \rangle_0$ from Section~\ref{sec:function-spaces}.
To simplify notation, we let $\hat{f} = \iota(f)$.
For two sequences $a, b \in \hat{H}^{s}_L$ we define $(ab)_i = a_ib_i$.
Because $f \in H^{s}_L$, we have $\hat{f} \in \hat{H}^s_L$.
Define $\mb{1}_{R}$ as the sequence such that $(\mb{1}_{R})_i = 0$ if $i < R$ and $(\mb{1}_{R})_i = 1$ if $i \geq R$. 
By norm equivalence,
\begin{equation}
    \begin{aligned}
        \lVert \hat{f}\mb{1}_R \rVert_{\hat{0}}^2 = \sum_{i = R}^\infty (1 + \lambda_i)^{s-s}\hat{f}_i^2 &\leq (1 + \lambda_R)^{-s} \sum_{i = R}^\infty (1 + \lambda_i)^{s}\hat{f}_i^2 \\
        &\leq C(1 + \lambda_R)^{-s} \lVert f \rVert_{s}^2. \label{eq:app_interpolant_lower_bound}
    \end{aligned}
\end{equation}
By Proposition \ref{prop:smoothness_limit}, almost surely there exists $0 < c < \infty$ such that, for all $R \in \mathbb{N}$,
\begin{align}
    \lVert \hat{u} \mb{1}_R \rVert_{\hat{0}}^2 &= \sum_{i = R}^\infty \hat{u}_i^2  = v(\theta_0)\sum_{i = R}^\infty (\tau_0 + \lambda_i)^{-s_0}\xi_i^2  \geq c v(\theta_0) R^{-2 s_0/d + 1}.\label{eq:app_norm_lower_bound_u}
\end{align}
Via the reverse triangle inequality, we obtain
\begin{align*}
    \lVert \hat{f} \mb{1}_R \rVert_{\hat{0}} \geq \lVert \hat{u} \mb{1}_R \rVert_{\hat{0}} - \lVert (\hat{f} - \hat{u}) \mb{1}_R \rVert_{\hat{0}} \geq \lVert \hat{u} \mb{1}_R \rVert_{\hat{0}} - \lVert f - u \rVert_{0}.
\end{align*}
Now set 
\begin{align*}
    R = \bigg\lceil \bigg( \frac{2C_1 n^{-s_0/d + 1/2 + \epsilon/d}}{\sqrt{\smash[b]{c v(\theta_0)}}} \bigg)^{1/(-s_0/d+1/2)} \bigg\rceil.
\end{align*}
Using \eqref{eq:app_assumption} and \eqref{eq:app_norm_lower_bound_u} and the above choice of $R$ we get
\begin{align*}
  \lVert \hat{f} \mb{1}_R \rVert_{\hat{0}} &\geq \lVert \hat{u} \mb{1}_R \rVert_{\hat{0}} - \lVert f - u \rVert_{0} \\
  &\geq \sqrt{\smash[b]{cv(\theta_0)}} R^{-s_0/d + 1/2} - C_1 n^{-s_0 /d + 1/2 + \epsilon/d} - C_2 n^{-s/d}\lVert f \rVert_{s} \\ 
    &\geq C_1 n^{-s_0/d + 1/2 +\epsilon/d} - C_2 n^{-s/d}\lVert f \rVert_{s}.
\end{align*}
Then~\eqref{eq:app_interpolant_lower_bound} and Weyl's law in~\eqref{eq:weyls-law} give 
\begin{align*}
    \lVert f \rVert_{s} &\geq c (1 + \lambda_R)^{s/2} \lVert \hat{f} \mb{1}_R\rVert_{\hat{0}}\\
    &\geq c R^{ s/d}\big(C_1 n^{-s_0/d + 1/2 +\epsilon/d} - C_2 n^{-s/d} \lVert f \rVert_{s} \big) \\
    &\geq c n^{s/d - 2\epsilon s / (d(2s_0 - d))} \big(C_1 n^{-s_0/d + 1/2 +\epsilon/d}  - C_2 n^{-s/d} \lVert f \rVert_{s} \big) \\ 
    &=c \big(C_1 n^{1/2 + (s-s_0)/d +  \epsilon/d - 2\epsilon s /(d(2s_0 - d))} - C_2 n^{-2\epsilon s / (d(2s_0 - d))} \lVert f \rVert_{s} \big) \\
    &\geq c \big(C_1 n^{1/2 + (s-s_0)/d +  \epsilon/d - 2\epsilon s /(d(2s_0 - d))} - C_2 \lVert f \rVert_{s}\big),
\end{align*}
where on the last line we used the fact that $2s_0 > d$.
Rearranging then yields 
\begin{align*}
    \lVert f \rVert_{s} \geq c n^{1/2 - (s-s_0)/d +  \epsilon/d - 2\epsilon s /(d(2s_0 - d))}.
\end{align*}
Since $\epsilon/d - 2\epsilon s/(d(2s_0-d)) \to 0 $ as $\epsilon \to 0$, relabelling completes the proof.
\end{proof}
\subsection{Function Approximation in RKHS on Manifolds}

To prepare the remaining proofs, the following lemma will be convenient.

\begin{lemma}[Lemma 1 in \citealp{krieg_manifold_2022}]\label{l:equiv_distances}
    The exponential map is bi-Lipschitz. 
    That is, 
    \begin{align*}
        c_j \lVert x - y \rVert \leq d_{\mathcal{M}}(\phi_j^{-1}(x), \phi_j^{-1}(y)) \leq C_j \lVert x - y \rVert
    \end{align*}
    for some $c_j, C_j > 0$ and all $x, y \in \phi_j(\Omega_j)$.
    Consequently, choosing $c = \min_{1 \leq j \leq m}$ and $C = \max_{1 \leq j \leq m} C_j$ we have 
    \begin{align*}
        c \lVert x - y \rVert \leq d_{\mathcal{M}}(\phi_j^{-1}(x), \phi_j^{-1}(y)) \leq C \lVert x - y \rVert ~~~ \text{ for all }  x, y \in \Omega_j \text{ and } 1 \leq j \leq m .
    \end{align*}
\end{lemma}

\begin{namedtheorem}[Proposition \ref{prop:escape_manifold}]
    Let $f \in H^{s_0}_L$ with $s \geq s_0 > d$ and let $m_{s,n}^f$ be its minimum norm interpolant from the RKHS $\mathcal{H}^s$.
    Then there exists a positive constant $C$ that depends on $s_0$ and $\mathcal{M}$ such that we have almost surely
    \begin{align}
        \lVert f - m_{s,n}^f \rVert_{0} \leq C h_{\mb{x}}^{s_0} \rho_{\mb{x}}^{s-s_0} \lVert f \rVert_{s_0} 
    \end{align}
    when $h_{\mb{x}}$ is small enough.
    Thus, if $h_{\mb{x}} \to 0 $ and $C$ is allowed to depend on $f$, then \eqref{eq:escape_inequality} holds for all $n$.
\end{namedtheorem}
\begin{proof}
    We make extensive use of Theorem \ref{thm:manifold_extension_theorem}.
    To that end, recall that $T_{\mathcal{M}} f = f \vert_{\mathcal{M}}$ is a bounded operator $T_{\mathcal{M}} : H^{s + (k-d)/2}(\mathbb{R}^k) \to H^{s}_L$. 
    Its inverse $E_{\mathcal{M}}$ on $H^{s}_L$ is bounded as well.
    Furthermore, note that from Lemma \ref{l:equiv_distances} it follows that the point set $\left\{\phi_j(x_i) ~ \middle \vert ~ x_i \in \Omega_j \right\}$ has separation radius $q_{\mb{x}}^j$  such that 
    \begin{align*}
        c q_{\mb{x}} \leq q_{\mb{x}}^j \leq C q_{\mb{x}.}
    \end{align*}

    Throughout the remainder of the proof $C$ denotes universal constants only depending on $s_0, s, \tau, \mathcal{M}$ that may be different from one occurence to the next.
    Note that as in Lemma 10 in \citet{fuselier_2012}, we have for $h_{\mb{x}}$ small enough
    \begin{align}
        \lVert f - m_{s,n}^f \rVert_0 \leq Ch_{\mb{x}}^{s_0} \lVert f - m_{s, n}^f \rVert_{s_0}.\label{eq:sobolev_error_calc1}
    \end{align}
    Now let $E_{\mathcal{M}}f$ be the extension of $f$. 
    It follows from Theorem 3.4 and Corollary 3.5 in \citet{narcowich_2006} that there exists a sequence $\eta = \eta(n)$ such that $\eta \asymp q_{\mb{x}}^{-1}$ and corresponding functions $f_{\eta} : \mathbb{R}^k \to \mathbb{R}$ satisfying the following:
    \begin{enumerate}
        \item For the Fourier transform $\hat{f}_{\eta}$, we have $\supp{\hat{f}_{\eta}} \subset B_{\eta}(\mb{0})$. That is, $f_{\eta}$ is \textit{bandlimited}. \label{enum:bandlimited}
        \item The function $f_{\eta}$ interpolates $u$, that is 
        \begin{align*}
            f_{\eta}(x_i) = f(x_i),~~ 1 \leq i \leq n.
        \end{align*}
        \item We have $\lVert E_{\mathcal{M}} f - f_{\eta} \rVert_{H^{s_0 + (k-d)/2}(\mathbb{R}^k)} \leq 5 \lVert E_{\mathcal{M}} f \rVert_{H^{s_0 + (k-d)/2}(\mathbb{R}^k)}$, which, together with the boundedness of the restriction operator, implies \label{enum:bandlimited_approximation}
        \begin{align*}
            \lVert f - T_{\mathcal{M}} f_{\eta} \rVert_{s_0} \leq C\lVert f \rVert_{s_0}.
        \end{align*}
        \item We have $\lVert f_{\eta} \rVert_{H^{s_0+ (k-d)/2}(\mathbb{R}^k)} \leq C \lVert E_{\mathcal{M}} f \rVert_{H^{s_0+ (k-d)/2}(\mathbb{R}^k)}$, which, together with norm equivalence, implies \label{enum:bandlimited_norm_upper}
        \begin{align*}
            \lVert f_{\eta} \rVert_{H^{s_0+ (k-d)/2}(\mathbb{R}^k)} \leq C \lVert f \rVert_{s_0}.
        \end{align*}
    \end{enumerate}
    Moreover, the functions $f_{\eta}$ satisfy the Bernstein inequality 
    \begin{align*}
        \lVert f_{\eta} \rVert_{H^{s}(\mathbb{R}^k)} \leq C\eta^{s - s^\prime} \lVert f_{\eta} \rVert_{H^{s^\prime}(\mathbb{R}^k)}
    \end{align*}
    for any $s \geq s'$.
    
    Because $u_{\eta}$ is bandlimited, it follows that $f_{\eta} \in H^s(\mathbb{R}^k)$ for all $s$.
    Then, norm equivalence and standard Hilbert space theory imply 
    \begin{align}
        \lVert T_{\mathcal{M}} f_{\eta} - m_{s, n}^f \rVert_{H^{s}(\mathcal{M})} \leq C\lVert T_{\mathcal{M}}f_{\eta} - m_{s,n}^f \rVert_{s} \leq C\lVert  T_{\mathcal{M}}f_{\eta} \rVert_{s}. \label{eq:bandlimited_prop1}
    \end{align}
    Furthermore, norm equivalence, the Bernstein inequality of bandlimited functions, property~\ref{enum:bandlimited_norm_upper} above and boundedness of the extension and restriction operators imply
    \begin{align}
        C\lVert  T_{\mathcal{M}}f_{\eta} \rVert_{s} \leq C\eta^{s-s_0}\lVert f_{\eta} \rVert_{H^{s_0 + (k-d)/2}(\mathbb{R}^k)}  \leq C \eta^{s-s_0} \lVert f \rVert_{s_0}.\label{eq:bandlimited_prop2}
    \end{align}
    Combining \eqref{eq:bandlimited_prop1} and \eqref{eq:bandlimited_prop2} with property \ref{enum:bandlimited_approximation} and using once more Lemma 10 in \cite{fuselier_2012} gives
    \begin{equation}
    \begin{aligned}
        \lVert f - m_{s, n}^f \rVert_{s_0} &\leq \lVert f - T_{\mathcal{M}} f_{\eta} \rVert_{s_0} + \lVert T_{\mathcal{M}} f_{\eta} - m_{s, n}^f \rVert_{s_0} \\
        &\leq C \lVert f \rVert_{s_0} + C h_{\mb{x}}^{s-s_0}\lVert T_{\mathcal{M}} f_{\eta} - m_{s, n}^f \rVert_{s} \\
        &\leq C \lVert f \rVert_{s_0} + C h_{\mb{x}}^{s-s_0}\lVert T_{\mathcal{M}} f_{\eta} \rVert_{s} \\
        &\leq C \lVert f \rVert_{s_0} + C \left(\frac{h_{\mb{x}}}{q_{\mb{x}}} \right)^{s-s_0}  C \lVert f \rVert_{s_0} \\
        &\leq (C + C\rho_{\mb{x}}^{s-s_0}) \lVert f \rVert_{s_0}.\label{eq:sobolev_error_calc2}
    \end{aligned}
    \end{equation}
    Because $\rho_{\mb{x}} \geq 1$, it follows that $(C+C\rho_{\mb{x}}^{s-s_0}) \leq C  \rho_{\mb{x}}^{s-s_0}$.
    Thus combining \eqref{eq:sobolev_error_calc1} and \eqref{eq:sobolev_error_calc2} we obtain for $h_{\mb{x}}$ small enough
    \begin{equation*}
        \lVert f - m_{s, n}^f \rVert_{0} \leq C h_{\mb{x}}^{s_0} \rho_{\mb{x}}^{s-s_0} \lVert f \rVert_{s_0}. \qedhere
    \end{equation*}
\end{proof}

\begin{namedtheorem}[Proposition \ref{prop:upper_bound_minimum_norm_interpolant}]
    Let $u$ be as in \eqref{eq:true_model}.
    Suppose that $s_0 > d$ and $s > d/2$ satisfy $s \geq s_0 - d/2$.
    Let $m_{s,n}$ be the minimum norm interpolant for $u$.
    Then almost surely
    \begin{align}
        \lVert m_{s,n} \rVert_s^2 \leq C q_{\mb{x}}^{s_0 - s  - d/2 - \epsilon}
        \label{eq:norm_upper_bound_interpolant_app}
    \end{align}
    for all $\epsilon > 0$, where $C$ depends on $s, \theta_0, \epsilon$ and the sample path.
    If $\xi_i$ are Gaussian, then it is only necessary to assume $s_0 > d/2$ to obtain
    \begin{align}
        c n \leq \lVert m_{s_0, n} \rVert_{s_0}^2 \leq C n \label{eq:norm_upper_bound_interpolant_gaussian_app}
    \end{align}
    almost surely for positive constants $c$ and $C$ that depend only on $\theta_0$ and the sample path.
\end{namedtheorem}
\begin{proof}
    Let $T_{\mathcal{M}}$  and $E_{\mathcal{M}}$ be the bounded restriction and extension operators from Theorem \ref{thm:manifold_extension_theorem}.
    Let $u_{\eta}$ be a sequence of bandlimited function as in the proof of Proposition \ref{prop:escape_manifold} such that $u(x_i) = E_{\mathcal{M}}u(x_i) = u_{\eta}(x_i)$.
    Then
    \begin{align*}
        \lVert u_{\eta} \rVert_{H^{s_0 + (k-d)/2}(\mathbb{R}^k)} \leq C \lVert E_{\mathcal{M}} u \rVert_{H^{s_0 + (k-d)/2}(\mathbb{R}^k)}.
    \end{align*}
    Theorem \ref{thm:manifold_extension_theorem}, norm equivalence, the minimum norm property and the Bernstein inequality for bandlimited functions yield
    \begin{align*}
        \lVert m_{s, n} \rVert_{s} \leq \lVert T_{\mathcal{M}}u_{\eta} \rVert_{s} \leq C\lVert u_{\eta} \rVert_{H^{s + (k-d)/2}(\mathbb{R}^k)} &\leq C\eta^{s_0 - s - d/2 - \epsilon} \lVert u_{\eta} \rVert_{H^{s_0 + k/2 -d - \epsilon}(\mathbb{R}^k)} \\
        &\leq C q_{\mb{x}}^{s_0 - s  - d/2 - \epsilon} \lVert u \rVert_{s_0 - d/2 - \epsilon}.
    \end{align*}
    This proves \eqref{eq:norm_upper_bound_interpolant_app}.

    To prove \eqref{eq:norm_upper_bound_interpolant_gaussian_app}, let $m_{\theta_0, n}$ be the minimum norm interpolant from $H^{s_0}_L$ equipped with the norm $\lVert \cdot \rVert_{\theta_0}$.
    Since $\V(u(\mb{x})) = \sigma_0^2 K_{\theta_0}(\mb{x})$, we can write
    \begin{align*}
        \lVert m_{\theta_0, n} \rVert^2_{\theta_0} = u(\mb{x})^{\top} K_{\theta_0}(\mb{x})^{-1}u(\mb{x}) = \sigma_0^2 \tilde{\mb{\xi}}_n^{\top}K_{\theta_0}(\mb{x})^{1/2}K_{\theta_0}(\mb{x})^{-1}K_{\theta_0}(\mb{x})^{1/2} \tilde{\mb{\xi}}_n = \sigma_0^2 \sum_{i = 1}^n \tilde{\xi}_{n,i}^2,
    \end{align*}
    where $\{\tilde{\mb{\xi}}_n\} = \{(\tilde{\xi}_{n,1}, \dots, \tilde{\xi}_{n,n})\}$ forms a triangular array of independent chi-squared random variables.
    By the strong law of large numbers,
    \begin{align*}
        \frac{1}{n} \sum_{i = 1}^n \tilde{\xi}_{n,i}^2 \overset{\text{a.s.}}{\longrightarrow}  1,
    \end{align*}
    which yields the claim for $\lVert m_{\theta_0, n} \rVert_{\theta_0}$.
    For any other kernel the claim then follows from the minimum norm property and norm equivalence .
\end{proof}

In the proof of Proposition~\ref{prop:cond_var_bounds} it will be convenient to use the following equivalent definition of Sobolev spaces on $\mathcal{M}$.
Let $\{(\Omega_i, \phi_i)\}_{i = 1}^m$ denote the atlas of normal coordinates with subordinate partition of unity $\{\psi_i\}$.
We define the Sobolev spaces 
\begin{align}
    H^s(\mathcal{M}) = \left\{ f\in L_2(\mathcal{M}) ~ \middle \vert ~ \lVert f \rVert_{H^s(\mathcal{M})}^2 \coloneqq \sum_{i = 1}^m \left\lVert \pi_i(f) \right\rVert_{H^s(\mathbb{R}^d)}^2 < \infty\right\}, \label{eq:sobolev_manifold}
\end{align}
where 
\begin{align}
    \pi_i(f) (x) = \begin{cases}
        \psi_i f \circ \phi_i^{-1}(x) & \text{ if } x \in \phi_i(\Omega_i) \\
        0 & \text{ else}.
    \end{cases} \label{eq:projection_pi}
\end{align}
As in Appendix~\ref{sec:appendix_manifold}, we have identified $T_p\mathcal{M}$ with $\mathbb{R}^d$.
It follows from \citet[Theorem~7.2.3]{triebel_1991} that this definition is independent of the choice of normal coordinates and the partition of unity up to equivalent norms and it follows from \citet[Theorem~7.4.5]{triebel_1991} that $H^s(\mathcal{M})$ is norm equivalent to $H^s_L$.
Define balls in a bounded Euclidean domain $M$ and on a manifold $\mathcal{M}$ as
\begin{align*}
    B_{r}(x) = \left\{y \in M ~ \vert ~ \lVert x - y \rVert \leq r \right\} \quad \text{ and } \quad B_{r}^{\mathcal{M}}(x) = \left\{y \in \mathcal{M} ~ \vert ~  d_{\mathcal{M}}(x,y)  \leq r \right\} .
\end{align*}

\begin{namedtheorem}[Proposition \ref{prop:cond_var_bounds}]
    Let $s > d/2$.
    \begin{enumerate}
        \item  Suppose the family of norms $\{\lVert \cdot \rVert_s\}_{s > d/2}$ is such that $\lVert f \rVert_{s} \leq C \lVert f \rVert_{s^\prime}$ for every $f$ whenever $s \leq s^\prime$.
          Let $\mb{x}^n_i$ be all points in $\mb{x}$ except $x_i$.
          Then, for $s \leq S$, there is a constant $c > 0$, that only depends on $S$, such that
        \begin{align*}
            \min_{1 \leq i \leq n}\V_{s} (x_i \vert \mb{x}^n_i) \geq c q_{\mb{x}}^{2S-d}.
        \end{align*}
        \label{eq:cond_var_bounds_appendix1}
        \item There exists a positive constant $C$, that only depends on $s$ and $\mathcal{M}$, such that
        \begin{align*}
            \sup_{y \in \mathcal{M}}\V_{s}(y \vert \mb{x}) \leq C h_{\mb{x}}^{2s-d}.
        \end{align*}
        \label{eq:cond_var_bounds_appendix2}
    \end{enumerate}
\end{namedtheorem}
\begin{proof}
    We start by proving \eqref{eq:cond_var_bounds_appendix1} and follow the same strategy as in the proof of Proposition~3.5 in~\cite{karvonen2023asymptotic}.
    We have 
    \begin{align*}
        \V_{s}(x_i \vert \mb{x}^{i-1}) = \sup_{\lVert f \rVert_{\mathcal{H}^s} \leq 1} \lvert f(x_i) - m^f_{s,n}(x_i \vert \mb{x}^{i-1}) \rvert^2, 
    \end{align*}
    where $m_{s,n}^f(\cdot \vert \mb{x}^n_i)$ is the minimum norm interpolant based on the observations $\{f(x_j)\vert j \neq i\}$.
    The goal is to construct a function $f$ such that $\lVert f \rVert_{\mathcal{H}^s} \leq 1 $ and $f(x_j) = 0$ for all $j \neq i$.
    In this case, it is clear that $m^f_{s,n}(\cdot \vert \mb{x}^n_i)  \equiv 0$ and so $f(x_i) - m^f_{s,n}(x_i \vert \mb{x}^n_i)  = f(x_i)$.
    This is a well-known strategy and usually so-called bump functions are used.
    Our construction here is due to \citet[Section~2.2]{krieg_manifold_2022}.
    Recall the atlas $\{(\Omega_j, \phi_j)\}_{j = 1}^m$ and let $\{\psi_j\}_{j = 1}^m$ be a subordinate partition of unity.
    We now generate a finite set of new atlasses based on geodesic coordinates that are particularly well suited for our purposes.
    First, we construct a new covering of $\mathcal{M}$ via $\left\{(4\Omega_j, \phi_j)\right\}$, where 
    \begin{align*}
        4\Omega_j = \left\{ x \in \mathcal{M} ~ \middle \lvert ~ d_{\mathcal{M}}(x, p_j) < 4r \right\}
    \end{align*}
    and $p_j$ is the center of $\Omega_j$, that is, $ \mb{0}$ in local coordinates of $(\Omega_j, \phi_j)$ and the $\phi_i$ are now to be understood as inverses of the exponential maps on $B_{4r}(\mb{0})$.
    This is possible because $r < r_0/32$ and so the extensions are still diffeomorphisms and the radii of the $4 \Omega_j$ are still less than $r_0/8$ as required by Chapter 7 in \cite{triebel_1991}.
    We construct new atlasses $\mathcal{A}_k$ and partitions of unity $\Psi_k$ in the following way. 
    First, let $I_k = \left\{ j ~ \middle \vert ~ \Omega_j \cap 1.5 \Omega_k = \emptyset\right\}$.
    Then define the atlas $\mathcal{A}_k = \{ (4\Omega_j, \phi_j)\}_{j \in I_k} \cup \{(4\Omega_k, \phi_k)\}$.
    Note that $4\Omega_k \cup \bigcup_{j \in I_k} 4\Omega_j$ still covers $\mathcal{M}$. 
    Finally, define $\tilde{\psi}_k$ as
    \begin{align*}
        \tilde{\psi}_k = \sum_{ j \leq m, ~ j \notin J_k } \psi_j
    \end{align*}
    and \smash{$\Psi_k = \left\{\psi_j\right\}_{j \in J_k} \cup \tilde{\psi_k}$}.
    This is still a partition of unity such that \smash{$\supp \psi_j \subset 4\Omega_j$} and  $\supp \tilde{\psi}_k \subset 4 \Omega_k$ with the convenient property that by construction, $\tilde{\psi}_k \equiv 1$ on $1.5 \Omega_k$.
    Finally, it follows from Theorem 7.2.3 in \cite{triebel_1991} that each atlas $\mathcal{A}_k$ together with partition of unity $\Psi_k$ generates a space that is norm equivalent to $H^s(\mathcal{M})$.
    
    Now suppose $q_{\mb{x}} < r/2$.
    Then $B^{\mathcal{M}}_{q_{\mb{x}}}(x_i) \subset  1.5\Omega_k$ whenever $x_i \in \Omega_k$ and $x_j \notin B^{\mathcal{M}}_{q_{\mb{x}}}(x_i)$ for $j \neq i$.
    Lemma~\ref{l:equiv_distances} implies that 
    \begin{align*}
        \phi_k^{-1}(B_{c q_{\mb{x}}}(\phi_k(x_i)) := \phi_k^{-1} \left(\left\{ y \in \phi_k(\Omega_k) ~ \middle \vert ~ \lVert y - \phi_k(x_i) \rVert \leq c q_{\mb{x}} \right\}\right) \subset B_{q_{\mb{x}}}^{\mathcal{M}}(x_i) \subset 1.5 \Omega_k.
    \end{align*}
    Now, define $f:\mathcal{M} \to \mathbb{R}$ such that $f(x) = 0$ if $x \notin \Omega_k$ and, for $x \in \phi_k^{-1}(\Omega_k)$,
    \begin{equation*}
      f \circ \phi_k^{-1}(x) =
      \begin{dcases}
        \exp \left(- \frac{1}{1- \left\lVert \frac{x - \phi_k(x_i)}{c q_{\mb{x}}}  \right\rVert^2}\right) & \text{ if } \quad \lVert x - \phi_k(x_i) \rVert \leq c q_{\mb{x}}, \\
            0 & \text{ else}.
        \end{dcases}      
    \end{equation*}
    That is, $f \vert_{\Omega_k} = g \circ \phi_k$ for $g: \mathbb{R}^d \to \mathbb{R}$ such that
    \begin{align*}
        g(x) = \begin{dcases}
            \exp \left(- \frac{1}{1- \left\lVert \frac{x - \phi_i(x_i)}{c q_{\mb{x}}}  \right\rVert^2}\right) & \text{ if } \quad\lVert x - \phi_k(x_i) \rVert \leq c q_{\mb{x}}, \\
            0 & \text{ else.}
        \end{dcases}
    \end{align*}
    Note that $\tilde{\psi}_k f = f$ and  $\psi_j f \equiv 0$ for $j \neq k$.
    Following the exact same steps as in the proof of Proposition~3.5 in~\cite{karvonen2023asymptotic}, it follows that there exists a constant $C$ such that
    \begin{align*}
        \lVert \psi_k f \circ \phi_k^{-1} \rVert_{H^{S}(\mathbb{R}^d)} = \lVert g \rVert_{H^S(\mathbb{R}^d)} \leq C q_{\mb{x}}^{d/2 - S}.
    \end{align*}
    Thus, by norm equivalence, norm monotonicity and the fact that each atlas $\mathcal{A}_k$ yields a norm equivalent space and there are only finitely many such atlasses, we have that there exists an $f$ such that $f(x_i) = 1$ and 
    \begin{align*}
        \lVert f\rVert_{s} &\leq C \lVert f \rVert_{H^{S}(\mathcal{M})} = C \bigg(\sum_{j = 1}^m \lVert \pi^k_j(f) \rVert_{H^{S}(\mathbb{R}^d)}^2 \bigg)^{1/2} = C \lVert \pi_k^k(f) \rVert_{H^S_2(\mathbb{R}^d)} = C\lVert g \rVert_{H^S(\mathbb{R}^d)}\\ 
        &\leq  C q_{\mb{x}}^{d/2-S},
    \end{align*}
    where $\pi_j^k$ is the analog of the function defined in \eqref{eq:projection_pi} based on $\mathcal{A}_k$ and $\Psi_k$.
    The constant $C$ depends on $S$ and $\mathcal{H}^s$.
    Appropriately rescaling $f$ to $f_\textup{sc}$ such that $\lVert f_\textup{sc} \rVert_{\mathcal{H}^s} \leq 1$ for all $s \leq S$ yields $f_\textup{sc}(x_i) \geq C q_{\mb{x}}^{S-d/2}$.
    In the case $q_{\mb{x}} \geq r/2$ we repeat the procedure with $q_{\mb{x}}$ replaced by $r/2$.
    From this it follows that 
    \begin{align*}
        \V_{s}(u(x_i) \vert \mb{x}^n_i) &= \sup_{\lVert f \rVert_{\mathcal{H}^s} \leq 1} \lvert f(x_i) - m^f_{s,n}(x_i \vert \mb{x}^n_i)  \rvert^2 \\
        &\geq  f_\textup{sc}(x_i)^2 \\
        &\geq \min\{ C(r/2)^{2S-d}, C q_{\mb{x}}^{2S-d}\}
    \end{align*}
    as claimed.

    Next we prove~\eqref{eq:cond_var_bounds_appendix2}.
    It follows from Lemma~10 in~\cite{fuselier_2012} that, for every $f \in H^{s}_L$,
    \begin{align*}
        \lVert f - m_{s, n}^f \rVert_{\infty} \leq C h_{\mb{x}}^{s - d/2} \lVert f - m_{s, n}^f \rVert_{\mathcal{H}^s} \leq C h_{\mb{x}}^{s - d/2} \lVert f \rVert_{\mathcal{H}^s}
    \end{align*}
    for $h_{\mb{x}}$ small enough, see the comment below Proposition \ref{prop:escape_manifold}.
    Then
    \begin{align*}
        \sup_{x \in \mathcal{M}}\V_{s}(x \vert \mb{x}) &= \sup_{x \in \mathcal{M}}\sup_{\lVert f \rVert_{\mathcal{H}^s} \leq 1} \lvert f(x) - m_{s, n}^f(x) \rvert^2 = \sup_{\lVert f \rVert_{\mathcal{H}^s} \leq 1} \lVert f - m_{s, n}^f \rVert_{\infty}^2 \leq Ch_{\mb{x}}^{2s-d} \lVert f \rVert_{\mathcal{H}^s}^2 \\
        &\leq Ch_{\mb{x}}^{2s-d}
    \end{align*}
    for $h_{\mb{x}}$ small enough, where $C$ does not depend on $f$. 
    Adjusting the constant $C$, the bound holds for all $n$.
\end{proof}

\subsection{Proofs for Restricted Kernels}

Here we give proofs for some results in Section~\ref{sec:other_kernels}.

\begin{namedtheorem}[Proposition \ref{prop:inherited_norm_monotonicity}]
  Let $\tilde{K}_s(\cdot, \cdot)$ be as in Theorem \ref{thm:fusilier_theorem}.
  Let $\mathcal{H}^s(\mathbb{R}^k)$ be the RKHS of $\tilde{K}_s$ and $\mathcal{H}^s$ the RKHS of $K_s(\cdot, \cdot) = \tilde{K}_s(\cdot, \cdot)\vert_{\mathcal{M} \times \mathcal{M}}$.
    Let $s \leq s^\prime$.
    If $\lVert f \rVert_{\mathcal{H}^s(\mathbb{R}^k)} \leq C \lVert f \rVert_{\mathcal{H}^{s^\prime}(\mathbb{R}^k)}$ for all $f \in \mathcal{H}^{s^\prime}(\mathbb{R}^k)$ and some $C$ independent of $s$, then $\lVert f \rVert_{s} \leq C \lVert f \rVert_{s^\prime}$ for all $f \in \mathcal{H}^s$.  
\end{namedtheorem}
\begin{proof}
    Lemma 4 in \cite{fuselier_2012} shows that for $f \in \mathcal{H}^{s^\prime}$ there exists an extension operator $E$ such that $\lVert f \rVert_{s^\prime} = \lVert E f \rVert_{\mathcal{H}^{s^\prime}(\mathbb{R}^k)}$.
    Moreover, for every $s$ the restriction operator $T_{\mathcal{M}}$ has norm bounded by 1 as operator $T_{\mathcal{M}} : \mathcal{H}^s(\mathbb{R}^k) \to \mathcal{H}_s$.
    It follows that $f = T_{\mathcal{M}}E f$ and so 
    \begin{equation*}
        \lVert f \rVert_s = \lVert T_{\mathcal{M}}E f \rVert_s \leq \lVert E f \rVert_{\mathcal{H}^s(\mathbb{R}^k)} \leq C \lVert E f \rVert_{\mathcal{H}^{s^\prime}(\mathbb{R}^k)} = C \lVert f \rVert_{s^\prime}. \qedhere
    \end{equation*}
\end{proof}
\begin{proposition}\label{prop:generalized_wendland_monotone}
    Set $\mu = \kappa + (d+1)/2$.
    Let $\kappa \leq \kappa^\prime$ and let $\mathcal{H}^\kappa(\mathbb{R}^k)$ and $\mathcal{H}^{\kappa^\prime}(\mathbb{R}^k)$ be the RKHS of $\Phi_{\kappa}$ and $\Phi_{\kappa^\prime}$ with norms  $\lVert \cdot \rVert_{\kappa}$ and $\lVert \cdot \rVert_{\kappa^\prime}$, respectively.
    Then $\mathcal{H}^{\kappa^\prime}(\mathbb{R}^k) \subset \mathcal{H}^{\kappa}(\mathbb{R}^k)$ and
    \begin{align*}
        \lVert f \rVert_{\kappa} \leq \lVert f \rVert_{\kappa^\prime}.
    \end{align*}
    for all $f \in \mathcal{H}_{\kappa^\prime}$.
\end{proposition}
\begin{proof}
    By Theorem 10.12 in~\cite{wendland_2004}, for any integrable and continuous radial basis function $\Phi$ on $\mathbb{R}^k$ the corresponding RKHS norm can be expressed via
    \begin{align*}
        \lVert f \rVert_{\Phi}^2 = \int_{\mathbb{R}^k} \lvert \hat{f}(\xi) \rvert^2 \hat{\Phi}(\lVert \xi \rVert)^{-1} \, \mathrm{d}\xi.
    \end{align*}
    The generalized Wendland functions $\Phi_{\kappa}$ are such radial basis functions \citep{hubbert_2023}.
    This implies $\mathcal{H}_{\kappa^\prime} \subset \mathcal{H}_{\kappa}$
    It thus suffices to show that $\hat{\Phi}_{\kappa^\prime}(z) \leq \hat{\Phi}_{\kappa}(z)$ whenever $\kappa \leq \kappa^\prime$.
    Equation 3.17 in \cite{hubbert_2023} provides a closed form for the Fourier transform of $\Phi_\kappa$.
    Set $\lambda = (d+1)/2 + \kappa$ and note that $\lambda > 1$.
    Then
    \begin{align*}
        \hat{\Phi}_{\kappa}(z) &= \frac{2^{\lambda}\Gamma(\lambda)\Gamma(\lambda + 1) }{\Gamma(3\lambda)}~ {}_1F_2\left(\lambda; \frac{3}{2}\lambda; \frac{3\lambda + 1}{2}; -\left(\frac{z\beta}{2}\right)^2\right), 
    \end{align*}
    where ${}_1F_2$ is the hypergeometric function~\citep[Chapter 15]{abramowitz_1964}. Note that the fourth argument is negative.
    Display~15.3.1 in~\cite{abramowitz_1964} shows that the hypergeometric function has the integral representation 
    \begin{equation}
    \begin{aligned}
        {}_1F_2&\left(\lambda; \frac{3}{2}\lambda; \frac{3\lambda + 1}{2}; -\left(\frac{z\beta}{2}\right)^2\right) \\
        &= \frac{\Gamma(3/2\lambda + 1/2)}{\Gamma(3/2\lambda) \Gamma(1/2)} \int_{0}^1 t^{3/2\lambda - 1}(1-t)^{-1/2} \bigg(1 + t\bigg(\frac{z \beta}{2}\bigg)^2 \bigg)^{-\lambda} \, \mathrm{d}t. \label{eq:factor1}
    \end{aligned}
    \end{equation}
    Now consider 
    \begin{align}
        \vartheta(\lambda) = \frac{\Gamma(3/2\lambda + 1/2)}{\Gamma(3/2\lambda) \Gamma(1/2)} \cdot \frac{2^{\lambda}\Gamma(\lambda)\Gamma(\lambda + 1) }{\Gamma(3\lambda)} \label{eq:factor2}
    \end{align}
    and let $\psi^{(0)}$ be the digamma function, that is $\psi^{(0)}(z) \Gamma(z) = \Gamma^\prime(z)$ \citep[Chapter 6]{abramowitz_1964}.
    It is an easy consequence of display~6.4.1 in~\cite{abramowitz_1964} that $\psi^{(0)}$ is monotonically increasing. 
    Straightforward calculation yields 
    \begin{align*}
        \frac{\vartheta^\prime(\lambda)}{\vartheta(\lambda)} = \frac{3}{2}\psi^{(0)}\left(\frac{3}{2}\lambda+\frac{1}{2}\right) + \log(2) + \psi^{(0)}(\lambda) + \psi^{(0)}(\lambda+1) - \frac{3}{2}\psi^{(0)}\left(\frac{3}{2}\lambda\right)- 3\psi^{(0)}(3\lambda).
    \end{align*}
    Because $\vartheta(\lambda) > 0$, the derivative $\vartheta^\prime$ is negative if and only if the right hand side is negative. 
    Due to the monotonicity of of $\psi^{(0)}$, for $\lambda > 1/2$ we have
    \begin{align}
        \frac{3}{2}\psi^{(0)}\left(\frac{3}{2}\lambda + \frac{1}{2}\right) + \psi^{(0)}(\lambda+1) - \frac{5}{2}\psi^{(0)}(3\lambda) < 0. \label{eq:factor2_deriv1}
    \end{align}
    Moreover, for $\lambda = 1$ we have by displays~6.3.2 and~6.3.5 in~\cite{abramowitz_1964} that
    \begin{align*}
        -\frac{1}{2}\left(\psi^{(0)}\left(\frac{3}{2}\lambda\right)+ \psi^{(0)}(3\lambda)\right) - \psi^{(0)}\left(\frac{3}{2}\lambda\right)+ \log(2)  + \psi^{(0)}(\lambda) \leq 0.
    \end{align*}
    Again by monotonicity, for all $\lambda \geq 1$ we then have
    \begin{align}
        -\frac{1}{2}\left(\psi^{(0)}\left(\frac{3}{2}\lambda\right)+ \psi^{(0)}(3\lambda)\right) - \psi^{(0)}\left(\frac{3}{2}\lambda\right)+ \log(2)  + \psi^{(0)}(\lambda) \leq 0.\label{eq:factor2_deriv2}
    \end{align}
    Combining \eqref{eq:factor2_deriv1} and \eqref{eq:factor2_deriv2} yields $\vartheta^\prime(\lambda) < 0$ for all $\lambda > 1$.

    To see that the integral in~\eqref{eq:factor1} is decreasing in $\lambda$, note that, for $t \in (0, 1)$,
    \begin{align*}
        \deriv{\lambda}t^{3/2\lambda - 1} \left(1 + t\frac{(\beta z)^2}{4}\right)^{-\lambda} = t^{3/2\lambda - 1} \left(1 + t\frac{(\beta z)^2}{4}\right)^{-\lambda} \left[ \frac{3}{2}\log(t) - \log\left(1 + t\frac{(\beta z)^2}{4}\right)\right] .
    \end{align*}
    Because $0< t < 1$, this implies that the non-negative function inside the integral in~\eqref{eq:factor1} is pointwise (in $t$) decreasing in $\lambda$, and thus the integral must be decreasing in $\lambda$.
    Combining \eqref{eq:factor1} and \eqref{eq:factor2}, we thus see that
    \begin{align*}
        \Phi_{\kappa}(x) = \vartheta(\lambda) \int_{0}^1 t^{3/2\lambda - 1}(1-t)^{-1/2} \left(1 + t\frac{(\beta z)^2}{4}\right)^{-\lambda} \, \mathrm{d}t
    \end{align*}
    is the product of two decreasing functions in $\lambda$.
    Because $\lambda$ is increasing in $\kappa$, it follows that $\Phi_{\kappa}$ is decrasing in $\kappa$ and thus for all $z > 0$ and for $\kappa \leq \kappa^\prime$ we have
    \begin{align*}
        \Phi_{\kappa^\prime}(z) \leq \Phi_{\kappa}(z), 
    \end{align*}
    which proves the claim.
\end{proof}

\subsection{An Asymptotic Normality Result}

Here we formally state and proof an asymptotic normality result from Section~\ref{sec:magnitude-estimation}.

\begin{proposition}\label{prop:asymptotic_normality}
    Let $u$ be as in \eqref{eq:true_model} and assume that $\{\xi_i\}$ are i.i.d. standard Gaussian.
    Assume that the true smoothness parameter is known and fixed at $s_0 > d/2$.
    Let $\hat{\sigma}_n^2$ and $\hat{\tau}_n$ be the maximizers of the log-likelihood 
    \begin{align*}
        \ell(\tau, \sigma^2; u(\mb{x})) = -\frac{n}{2}\log(2\pi) - \frac{1}{2}\log\left(\det( \sigma^2K_{\theta}(\mb{x})\right) - \frac{1}{2\sigma^2}u(\mb{x})^{\top} K_\theta(\mb{x})^{-1} u(\mb{x}),
    \end{align*}
    where $\theta = (s_0, \tau)$ and $\hat{\tau}_n$ is constrained to be in $[\tau_L, \tau_U]$ for some $1 \leq \tau_L < \tau_U < \infty$.
    Let $\hat{\theta}_n = (s_0, \hat{\tau}_n)$.
    If $d \leq 3$, then
    \begin{align*}
        \sqrt{n}(\hat{\sigma}_n^2v(\hat{\theta}_n) - \sigma_0^2v(\theta_0)) \to \mathcal{N}(0,2\sigma_0^4 v(\theta_0)^2) .
    \end{align*}
\end{proposition}

Note that the restriction $\tau_L \geq 1$ is not really a restriction.
We can always rescale the underlying distances to obtain $\tau_L \geq 1$.
The assumption $d \leq 3$ is due to the fact that both $\sigma_0^2$ and $\tau_0$ can be consistently estimated if $d > 3$~\citep{anderes_2010}.

\begin{proof}
    We have 
    \begin{align}
        \sqrt{n}\left[\frac{u(\mb{x})^{\top} K_{\theta_0}(\mb{x})^{-1} u(\mb{x})}{n \sigma_0^2} - 1 \right]= \frac{1}{\sqrt{n}}\sum_{i = 1}^n (\tilde{\xi}_{n,i}^2 - 1)\to \mathcal{N}(0, 2), \label{eq:base_clt}
    \end{align}
    for example by an application of the Lindeberg--Feller theorem \citep[Theorem 6.13]{kallenberg_2002}.
    Now fix an arbitrary $\tau_1$ and let $\theta_1 = (s_0, \tau_1)$. 
    Let furthermore $\sigma_1^2$ be such that $\sigma_1^2 v(\theta_1) = \sigma_0^2v(\theta_0)$ and let $\hat{\sigma}_{\tau_1, n}^2$ be the maximum likelihood estimator when using $\tau = \tau_1$ and $s = s_0$.
    Because $\sigma_1^2v(\theta_1)  =\sigma_0^2 v(\theta_0)$, it suffices to show $\sqrt{n}(\hat{\sigma}_{n,\tau_1}^2/ \sigma_1^2 - 1) \to \mathcal{N}(0,2)$ in order to show that $\sqrt{n}(\hat{\sigma}_{\tau_1, n}^2 v(\theta_1) - \sigma_0^2 v(\theta_0)) \to \mathcal{N}(0, \sigma_0^4 v(\theta_0)^2)$.
    Now
    \begin{align*}
        \sqrt{n}\left(\frac{\hat{\sigma}_{\tau_1, n}^2}{\sigma_1^2} - 1\right) ={}& \frac{1}{\sqrt{n}}\left[\frac{1}{\sigma_1^{2}}u(\mb{x})^{\top}K_{\theta_1}(\mb{x})^{-1}u(\mb{x}) - \frac{1}{\sigma_0^{2}}u(\mb{x})^{\top}K_{\theta_0}(\mb{x})^{-1}u(\mb{x}) \right] \\
        &+ \sqrt{n}\left[\frac{u(\mb{x})^{\top}K_{\theta_0}(\mb{x})^{-1}u(\mb{x})}{n \sigma_0^{2}} - 1 \right] .
    \end{align*}
    Asymptotic normality of the final summand follows from \eqref{eq:base_clt}.
    The choice of $\sigma_1$ yields equivalence of the measures $\mathcal{N}(0, \sigma_1^2 v(\theta_1)(\tau_1 -\Delta)^{s_0})$ and $\mathcal{N}(0, \sigma_0^2 v(\theta_0)(\tau_0 -\Delta)^{s_0})$.
    Therefore the other summand tends to 0 by an application of Chebychev's inequality as the variances of
    \begin{align*}
        \frac{1}{\sigma_1^{2}}u(\mb{x})^{\top}K_{\theta_1}(\mb{x})^{-1}u(\mb{x}) - \frac{1}{\sigma_0^{2}}u(\mb{x})^{\top}K_{\theta_0}(\mb{x})^{-1}u(\mb{x}) 
    \end{align*}
    are uniformly bounded; see for example Theorem~1 combined with display (2.9) in~\cite{ibragimov_1978}.
    This yields the claim for arbitrary $\tau_1$, and in particular for $\tau_L$ and $\tau_U$.
    Let $\theta_L = (s_0, \tau_L)$ and $\theta_U = (s_0, \tau_U)$.
    Now, by construction $\hat{\tau}_n$ is a bounded sequence in $[\tau_L, \tau_U]$ and so the claim follows if we can show that $\hat{\sigma}_{\tau_L,n}v(\theta_L) \leq \hat{\sigma}_{\hat{\tau}_n,n}v(\hat{\tau}_n) \leq \hat{\sigma}_{\tau_U,n}v(\theta_U)$.
    For this purpose let $\tau_1 \leq \tau_2$.
    Then 
    \begin{align*}
        n \big[\hat{\sigma}_{\tau_2, n}^2 v(\theta_2) &- \hat{\sigma}_{\tau_1, n}^2 v(\theta_1) \big]= u(\mb{x})^{\top} \left[v(\theta_2)K_{\theta_2}^{-1}(\mb{x}) - v(\theta_1)K_{\theta_1}^{-1}(\mb{x}) \right]u(\mb{x}) .
    \end{align*}
    Thus it suffices to show that $v(\theta_2)K_{\theta_2}^{-1}(\mb{x}) - v(\theta_1)K_{\theta_1}^{-1}(\mb{x}) $ is positive semi-definite.
    We have 
    \begin{align*}
        \frac{K_{\theta_2}(\mb{x})}{v(\theta_2)} = \sum_{i = 1}^\infty (\tau_2 + \lambda_i)^{-s_0} e_i(\mb{x}) e_i(\mb{x})^{\top} \preceq \sum_{i = 1}^\infty (\tau_1 + \lambda_i)^{-s_0} e_i(\mb{x}) e_i(\mb{x})^{\top} = \frac{K_{\theta_1}(\mb{x})}{v(\theta_1)},
    \end{align*}
    where for two matrices $A,B$ we write $A \preceq B$ if $B - A$ is positive semi-definite.
    This yields the claim as $A \succeq B \iff A^{-1} \preceq B^{-1}$ for positive definite matrices $A,B$.
\end{proof}

\section{Infinite Product Measures and Kakutani's Theorem}\label{sec:measure_kakutani}

In this section, we essentially repeat the content of Section 5 in \citet{kakutani_1948} in order to state Kakutani's theorem and then provide some additional results pertaining to the setting of the present paper.

Let $\{(\Omega_i, \mathcal{F}_i, \mathbb{P}_i)\}$ be a sequence of measure spaces.
Define $\Omega^\infty$ to be the set of sequences
\begin{align*}
    \Omega^\infty = \left\{ \{\omega_i\}_{i = 1}^\infty ~ \middle \vert ~ \omega_i \in  \Omega_i \right\}
\end{align*}
A subset $R$ of $\Omega^{\infty}$ is called \textit{rectangular} if it is of the form 
\begin{align*}
    R = \bigg( \prod_{i = 1}^n B_i \bigg) \times \bigg( \prod_{i = n+1}^\infty \Omega_i \bigg)
\end{align*}
with $B_i \in \mathcal{F}_i$, where the products $\prod$ and $\times$ are the Cartesian products.
An \textit{elementary} set is a subset of $\Omega^\infty$ that is the finite, disjoint union of rectangular sets. 
The family of elementary sets is a field \citep[p.\@~217]{kakutani_1948}.
Let $\mathcal{F}^\infty$ be the smallest $\sigma$-algebra containing all elementary sets.
To define the infinite product measure $\mathbb{P} \coloneqq \bigotimes_{i = 1}^\infty \mathbb{P}_i$, for a rectangular set $R$ and an elementary set $E = \bigcup_{i = 1}^n R_i$ we set
\begin{align*}
    \tilde{\mathbb{P}}(R) = \prod_{i = 1}^{n} \mathbb{P}_i(B_{i}) \quad \text{ and } \quad \tilde{\mathbb{P}}(E) = \sum_{i = 1}^n \tilde{\mathbb{P}}(R_i).
\end{align*}
The measure $\mathbb{P}$ is then the extension \cite[p.\@~217]{kakutani_1948} of $\tilde{\mathbb{P}}$ to $\mathcal{F}^\infty$.
Clearly, $\Omega^\infty$ is a rectangular set and thus if the $\mathbb{P}_i$ are probability measures, then so is $\mathbb{P}$.

To state Kakutani's theorem, consider the two infinite direct product measure spaces $(\Omega^\infty, \mathcal{F}^\infty, \mathbb{P}_1)$ and $(\Omega^\infty, \mathcal{F}^\infty, \mathbb{P}_2)$ constructed from the sequences of measures spaces $\{(\Omega_i, \mathcal{F}_i, \mathbb{P}_{1,i})\}$ and $\{(\Omega_i, \mathcal{F}_i, \mathbb{P}_{2,i})\}$.
Recall that the Hellinger distance is defined as
\begin{align*}
    \rho(\mathbb{P}_{1,i}, \mathbb{P}_{2,i}) = \int_{\Omega_i} \sqrt{\frac{\mathrm{d} \P_{1,i}}{ \mathrm{d} \lambda}} \sqrt{\frac{\mathrm{d} \P_{2,i}}{\mathrm{d} \lambda}} \mathrm{d}\lambda.
\end{align*}

\begin{theorem}[Kakutani's theorem from~\citealp{kakutani_1948}]
    Assume that the measures $\P_{1,i}$ and $\P_{2,i}$ are equivalent for all $i$.
    Then the infinite product measures $\mathbb{P}_1$ and $\mathbb{P}_2$ are either equivalent or orthogonal.
    They are orthogonal if and only if
    \begin{align*}
        \prod_{i = 1}^\infty \rho(\mathbb{P}_{1,i}, \mathbb{P}_{2,i}) = 0 \quad \text{ or equivalently} \quad \sum_{i = 1}^\infty - \log\left[\rho(\mathbb{P}_{1,i}, \mathbb{P}_{2,i})\right] < \infty.
    \end{align*}
    Moreover, 
    \begin{align*}
        \rho(\mathbb{P}_1, \mathbb{P}_2) = \prod_{i = 1}^\infty \rho(\mathbb{P}_{1,i}, \mathbb{P}_{2,i}).
    \end{align*}
\end{theorem}

Next, we consider the special case where $\Omega_i = \mathbb{R}$, and $\mathcal{F}_i = \mathcal{B}(\mathbb{R})$.
Let $l_2$ denote the set of square-summable sequences and $\lVert \cdot \rVert_{\hat{0}}$ the corresponding norm.
Then $l_2 \subset \mathbb{R}^\infty$.
By definition of $\mathcal{F}^{\infty}$, $a \mapsto \lVert a \rVert_{\hat{0}}^2$ is the countable sum of positive, measurable functions and thus measurable on \smash{$(\mathbb{R}^\infty, \mathcal{F}^\infty)$}, see Lemmas 1.10 and 1.13 in \citet{kallenberg_2002}.
Thus \smash{$l_2 = (\lVert \cdot \rVert_{\hat{0}})^{-1}([0, \infty))$} is measurable.

\begin{proposition}\label{prop:sigma_algebra_coincide}
    Let $\mathcal{B}(l_2)$ denote the $\sigma$-algebra generated by the open sets with respect to the norm topology.
    Then $\mathcal{B}(l_2) = \mathcal{F}^\infty\vert_{l_2} \coloneqq \{ A \cap l_2  ~ \vert ~ A \in \mathcal{F}^\infty\}$.
\end{proposition}
\begin{proof}
    To show $\mathcal{B}(l_2) \subset \mathcal{F}^\infty\vert_{l_2}$, reasoning as above, we note that for $b \in l_2$, the function $g(a) = \lVert a-b \rVert_{\hat{0}}$ is measurable and $B_{\delta}(b) = g^{-1}([0, \delta))$, where $B_{\delta}(b)$ is the ball of radius $\delta$ around $b \in l_2$. 
    The claim now follows from separability of $l_2$.
    For the converse, note that the sets $l_2 \cap ((a_i, b_i) \times \prod_{j \neq i} \Omega_j )$ generate $\mathcal{F}^\infty \vert_{l_2}$ and are open in $l_2$.
\end{proof}
Next, let 
\begin{align*}
    u = \sigma \sqrt{v(\theta)}\sum_{i = 1}^\infty (\tau + \lambda_i)^{-s/2} \xi_i e_i,
\end{align*}
where $\{\xi_i\}$ are i.i.d. random variables with finite second moment, and let $\P_{\sigma, \theta}$ be as in Section~\ref{sec:equivalent_measures}.

\begin{proposition}\label{prop:correct_measures}
    For any $B \in \mathcal{B}(L_2(\mathcal{M}))$ we have
    \begin{align*}
        \P (u \in B) = \P_{\sigma, \theta}(\iota(B)).
    \end{align*}
\end{proposition}
\begin{proof}
    Because $\iota^{-1}$ is an isomorphism, $\iota(B)$ is measurable.
    By definition, $\P(u \in B) = \P(\hat{u} \in \iota(B))$.
    Let $R = \prod_{i = 1}^n B_i \times \prod_{i = n+1}^\infty \mathbb{R}$ be a rectangular set.
    Then 
    \begin{align*}
        \P(\hat{u} \in R) = \P \left(\bigcap_{i = 1}^n \big \{\sigma \sqrt{v(\theta)}(\tau + \lambda_i)^{s/2} \xi_i \in B_i \big\} \right) = \prod_{i = 1}^n \P_{\sigma\lambda_i(\theta)}(B_i) = \P_{\sigma, \theta}(R). 
    \end{align*}
    From this the claim follows for elementary sets.
    By definition they generate $\mathcal{F}^\infty$ and so by the uniqueness of the extension
    \begin{align*}
        \P(\hat{u} \in A) = \P_{\sigma, \theta}(A)
    \end{align*}
    for any measurable set $A$, in particular for $A = \iota(B)$ where $B \in \mathcal{B}(L_2(\mathcal{M}))$.
\end{proof}

\end{document}